\documentclass[11pt]{amsart}
\usepackage{amssymb}
\usepackage{amscd}
\usepackage{mathrsfs}
\usepackage{amsfonts}
\usepackage{verbatim}
\usepackage{version}
\usepackage{ulem}
\usepackage{mathtools}
\usepackage{mathscinet}
\usepackage[dvipsnames]{xcolor}
\usepackage{a4wide}
\usepackage[active]{srcltx}
\usepackage[colorlinks,linkcolor={blue},
citecolor={blue},urlcolor={red},]{hyperref}
\usepackage{hyperref}
\usepackage{bbm}

\theoremstyle{plain}
\newtheorem{theorem}{Theorem}[section]
\newtheorem*{theorem*}{Theorem}
\theoremstyle{remark}
\newtheorem{remark}[theorem]{Remark}
\newtheorem{example}[theorem]{Example}

\theoremstyle{plain}
\newtheorem{corollary}[theorem]{Corollary}
\newtheorem{lemma}[theorem]{Lemma}
\newtheorem*{lemma*}{Lemma}
\newtheorem{proposition}[theorem]{Proposition}
\newtheorem{definition}[theorem]{Definition}
\newtheorem*{definition*}{Definition}
\newtheorem{hypothesis}[theorem]{Hypothesis}

\numberwithin{equation}{section}

\newenvironment{claim}[1]{\par\noindent\underline{Claim#1.}\\}{}
\newenvironment{claimproof}[1]{\par\noindent\textit{Proof of Claim #1:}}{\hfill $\qed$}

 \font\Bbbten=msbm10 \font\Bbbseven=msbm7 \font\Bbbfive=msbm5
\newfam\Bbbfam

  \textfont\Bbbfam=\Bbbten
  \scriptfont\Bbbfam=\Bbbseven
  \scriptscriptfont\Bbbfam=\Bbbfive
\def\R{{\mathbb R}}
\def\N{{\mathbb N}}
\def\E{{\mathbb E}}
\def\P{\mathbb P}

\def\bd{\par\medskip\noindent\begin{definition}\label}
\def\ed{\par\medskip\noindent\end{definition}}
\def\bt{\par\medskip\noindent\begin{theorem}\label}
\def\et{\end{theorem}}
\def\be{\begin{equation}\label}
\def\ee{\end{equation}}
\def\bex{\vskip0.05cm\begin{example}\label}
\def\eex{\end{example}}
\def\<{\langle} 
\def\>{\rangle}
\def\bl{\begin{lemma}\label}
\def\el{\end{lemma}}
\def\bexe{\begin{exercise}\label}
\def\eexe{\end{exercise}}
\def\bpr{\begin{proposition}\label}
\def\epr{\end{proposition}}
\def\bc{\begin{corollary}\label}
\def\ec{\end{corollary}}
\def\br{\begin{remark}\label}
\def\er{\end{remark}}
\def\mc{\mathscr}

\def\bal{\[\begin{aligned}}
\def\eal{\end{aligned}\]}
\def\vs{\vskip0.2cm\noindent}
\newcommand*{\bracket}[1]{\langle#1\rangle}
\newcommand{\mcB}{\mathcal{B}}
\newcommand*{\pairing}[3]{\sideset{_{#1}^{}}{_{#3}^{}}{\mathop{\bracket{{#2}}}}}
\newcommand*{\hscalar}[2]{\bracket{{#1},{#2}}_{H}}
\def\phi{\varphi}

\newcommand{\lb}{\langle}
\newcommand{\rb}{\rangle}

\renewcommand{\d}{\delta}
\renewcommand{\P}{{\mathbb P}}

\def\be{\begin{equation}\label}
\def\ee{\end{equation}}
\def\bd{\begin{definition}\label}
\def\ed{\end{definition}}
\def\bt{\begin{theorem}\label}
\def\et{\end{theorem}}
\def\bpr{\begin{proposition}\label}
\def\epr{\end{proposition}}
\def\bl{\begin{lemma}\label}
\def\el{\end{lemma}}
\def\bal{\[\begin{aligned}}
\def\eal{\end{aligned}\]}

\def\vp{\varphi}
\def\ve{\varepsilon}
\def\tc{\tau_\kappa^{\mc C}}

\def\tu{\tau_\kappa^{\mc U}}
\def\tm{\tau_\kappa^{\mc M}}

\def\ck{C_\kappa(E)}
\newcommand{\la}{\left(}
\newcommand{\ra}{\right)}

\newcommand*{\zx}{\color{red}}
\newcommand*{\xz}{\color{black}}

\def\ep{\varepsilon}
\def\ka{\kappa}

\newcommand{\EE}{\mathcal{E}}

\makeatletter
\def\author@andify{%
  \nxandlist {\unskip ,\penalty-1 \space\ignorespaces}%
    {\unskip {} \@@and~}%
    {\unskip \penalty-2 \space \@@and~}%
}
\makeatother 

\begin{document}

\title[Operator semigroups in mixed topologies]
{Operator semigroups in the mixed topology and the infinitesimal description of Markov processes}

\author{Ben Goldys}
\address{School of Mathematics and Statistics\\
The University of Sydney\\
Sydney 2006, Australia} \email{beniamin.goldys@sydney.edu.au}
\author{Max Nendel}
\address{Center for Mathematical Economics\\
Bielefeld University\\
33615 Bielefeld, Germany}
\email{max.nendel@uni-bielefeld.de}
\author{Michael R\"ockner}
\address{Faculty of Mathematics\\
Bielefeld University\\
33615 Bielefeld, Germany\vspace*{-0.9em}}
\address{ and Academy of Mathematics and Systems Science, CAS, Beijing, China}
\email{roeckner@math.uni-bielefeld.de}
\thanks{The authors thank Markus Kunze, Karsten Kruse, and Michael Kupper for valuable comments and remarks related to this work.\ The first named author was supported by the Australian Research Council Discovery Project DP120101886.\ The second and third author were supported by the Deutsche Forschungsgemeinschaft (DFG, German Research Foundation) - SFB 1283/2 2021-317210226.\ The second and third named authors would like to thank the University of Sydney for the hospitality during a very pleasant stay in November and December 2022.}

\keywords{Markov process, stochastic (partial) differential equation, mixed topology on weighted spaces of continuous functions, strongly continuous semigroup, upper semicontinuous weight, infinitesimal generator,
Markov uniqueness, viscosity solution, Fokker-Planck-Kolmogorov equation, generalized Mehler semigroups, Levy-Khintchin representation}

\subjclass[2020]{Primary: 47D06; 47H20; 60H10; 60H15; 60J25; 60J35 Secondary: 35D40; 47J35}

\begin{abstract}
We define a class of not necessarily linear $C_0$-semigroups $(P_t)_{t\geq0}$ on $C_b(E)$ (more generally, on $C_\kappa(E):=\frac{1}{\kappa} C_b(E)$, 
for some bounded function $\kappa$, {which is the pointwise limit of a decreasing sequence of continuous functions}) equipped with the mixed topology $\tau_1^\mathscr{M}$ for a large class of topological (not necessarily Polish) state spaces $E$.\ If these semigroups are linear, classical theory of operator semigroups on locally convex spaces as well as the theory of bicontinuous semigroups apply to them. In particular, they are infinitesimally generated by their generator $(L, D(L))$
and thus reconstructable through an Euler formula from their strong derivative at zero in $(C_b(E), \tau_{1}^{\mathscr M})$.
In the linear case, we prove that such $(P_t)_{t\geq0}$ can be characterized as integral operators given by measure kernels satisfying certain tightness properties.\
As a consequence, transition semigroups of Markov processes are $C_0$-semigroups on $(C_b(E), \tau_{1}^{\mathscr M})$,
if they leave $C_b(E)$ invariant and they are jointly weakly continuous in space and time. Hence, they can be reconstructed from their strong derivative at zero and thus have a fully infinitesimal description.\ This solves an open problem for Markov processes.\
We show that our results apply to a large number of Markov processes, e.g., those given as the laws of solutions to SDEs and SPDEs,
including the stochastic 2D Navier-Stokes equations and the stochastic fast and slow diffusion porous media equations.\
Furthermore, we introduce the notion of a Markov core operator $(L_0, D(L_0))$ for the above generators $(L, D(L))$
and prove that uniqueness of the Fokker-Planck-Kolmogorov equations corresponding to $(L_0,D(L_0))$ for all Dirac initial conditions implies that
$(L_0,D(L_0))$ is a Markov core operator for $(L,D(L))$.\
As a consequence, we identify the Kolmogorov operators of a large number of SDEs on finite and infinite dimensional state spaces
as Markov core operators for the infinitesimal generators of the $C_0$-semigroups on $(C_\kappa(E),\tau_{\kappa}^{\mathscr M})$ given by their transition semigroups.\
Furthermore, if each $P_t$ is merely convex, we prove that $(P_t)_{t \geq 0}$ gives rise to viscosity solutions to the Cauchy problem given by its associated (nonlinear) infinitesimal generator.\
We also show that value functions of optimal control problems, both, in finite and infinite dimensions are particular instances of convex $C_0$-semigroups on $(C_\kappa(E),\tau_{\kappa}^{\mathscr M})$.
\end{abstract}

\maketitle
\today



\section{Introduction}
One motivation of this paper is an old problem in the theory of Markov processes, which we shall describe at first.\
The literature on Markov processes is huge.\
Here, we only refer to a selection of pioneering and/or fundamental books on the subject and to the references therein,
as, e.g., \cite{bliedtner},\cite{blumenthal},\cite{dynkin},\cite{ethier},\cite{freidlin},\cite{fukushima},\cite{kolokoltsov},\cite{liggett},\cite{rogers},\cite{sharpe},\cite{stroock}.\ We briefly recall the definition of a Markov process:
Let $(E, \mathcal{B})$ be a measurable space and, for each $x \in E$, let $(\Omega, \mathcal{F}, (\mathcal{F}_{t})_{t \geq 0}, \P_x)$ be a filtered probability space and $X(t) \colon \Omega \to E \; \mathcal{F}_t/\mathcal{B}$-measurable maps, $t \geq 0$,
such that $\P_x [X(0) = x] = 1$.
Then the tuple $ \mathbb{M} := (\Omega, \mathcal{F}, (\mathcal{F}_{t})_{t \geq 0},(X(t))_{t \geq 0}, (\P_x)_{x \in E})$ is called
a (time-homogeneous) \textit{Markov process} with state space $E$, if it satisfies the \textit{Markov property},
i.e., for all $x \in E, A \in \mathcal{B}, t, s \geq 0$,
\begin{align}\label{eq:1.1}
\P_x [X(s+t) \in A | \mathcal{F}_s] = \P_{X(s)}[X(t) \in A] \quad \P_x \text{-a.s.},
\end{align} 
where $\P_x [\; \cdot \;| \mathcal{F}_s]$ denotes the conditional probability of $\P_x$ given $\mathcal{F}_s$.
Its corresponding transition semigroup of probability kernels is defined by the time marginal laws of $\P_x$ under $X(t), t \geq 0$, i.e.,
\begin{align}\label{eq:1.2}
p_t(x, dy) \coloneqq (\P_x \circ X(t)^{-1})(dy), \quad x \in E, t \geq 0.
\end{align}
Usually, one also assumes some path regularity on $X(t)$, $t \geq 0$,
by considering topological state spaces $E$ together with the corresponding Borel $\sigma$-algebra $\mathcal{B} := \mathcal{B}(E)$ and assuming that the map $[0, \infty)\; \ni \; t \mapsto X(t)\in E$ is right-continuous.
Define for $f \colon E \to \R$, bounded, $\mathcal{B}$-measurable, 
\begin{align}\label{eq:1.3}
P_t f(x) := \int_E f(y)\;p_t(x, dy) = \E_x [f(X(t))], \quad x \in E, t \geq 0,
\end{align}
where $\E_x$ denotes the expectation with respect to $\P_x$.
Then the Markov property \eqref{eq:1.1} implies the semigroup property
\begin{align}\label{eq:1.4}
P_{t+s} f(x) = P_t (P_s f)(x), \quad x \in E, t,s \geq 0.
\end{align}

A common very natural assumption, which is fulfilled in many situations (in particular, where $\P_x$, $x \in E$, are the laws of the solutions of a stochastic differential equation (SDE) with respective initial data $x \in E$ and where $E$ is, say a Banach space or just $\R^d$) is the so-called \textit{Feller property}, i.e., 
\begin{align}\label{eq:1.5}
P_t f \in C_b(E), \text{\;if\;} f \in C_b(E), \quad t \geq 0.
\end{align}
Here $C_b(E)$ denotes the set of all bounded real-valued continuous functions on $E$. Let $\mathcal{P}(E)$ denote the set of all probability measures an $(E, \mathcal{B}(E))$. Then \eqref{eq:1.5} means:
\begin{align}\label{eq:1.6}
E \ni x \mapsto p_t (x, dy) \in \mathcal{P}(E) &\text{\;is continuous in the weak topology} \\ 
	&\text{on $\mathcal{P}(E)$ for all $t \geq 0$.}\nonumber
\end{align}
By the assumed right continuity of sample paths and by \eqref{eq:1.5} we also have 
\begin{align}\label{eq:1.7}
[0, \infty) \ni t \mapsto p_t (x, dy) \in\; &\mathcal{P}(E) \text{\;is right continuous in the weak topology on}\\
	&\text{ $\mathcal{P}(E)$ for all $x \in E$.} \nonumber
\end{align}

It is well-known that, if we consider $C_b(E)$ with its supremum norm $\|\cdot\|_\infty$, then $t \mapsto P_t f$
is (in general) not continuous at $t = 0$ for all $f \in C_b(E)$,
i.e., $(P_t)_{t \geq 0}$ is not a $C_0$-semigroup on $(C_b(E), \|\cdot\|_\infty)$.\\

If $E$ is metric space, then the next natural choice is the space $UC_b(E)$ of bounded uniformly continuous functions which,
when endowed with the norm $\|\cdot\|_\infty$, is a closed subspace of $C_b(E)$.
It turns out that the gain is very limited.
It can be shown that if $E$ is a separable Hilbert space and $\la P_t\ra_{t\geq 0}$ is a transiton semigroup of an $E$-valued Wiener process
that $\la P_t\ra_{t\geq 0}$ is a $C_0$-semigroup on $UC_b(E)$, see Proposition 3.5.1 in \cite{dz3}.
This result can be easily extended to a general L\'evy process.
However, the transition semigroup of an Ornstein-Uhlenbeck process in $E=\R$, while it turns out to leave the space $UC_b(\R)$ invariant,
is not strongly continuous there, see Example 6.1 in \cite{cerrai} and Theorem 2.1 in \cite{neerven-Zabczyk}.\
The latter result also implies that the transition semigroup of a general Ornstein-Uhlenbeck process with non-zero drift is never strongly continuous on $C_b(E)$. 
\\

Hence, the theory of $C_0$-semigroups on Banach spaces (see e.g. \cite{pazy},\cite{engel}) does not apply.
If it did, $P_t, t \geq 0$, would be uniquely determined by its derivative at $t = 0$, i.e.,
\begin{align}\label{eq:1.8}
L f:= \frac{d}{dt}_{\upharpoonright t=0} P_t f = \|\cdot\|_\infty - \lim_{t \to 0} \frac{1}{t}(P_t f - f), \quad f \in D(L),
\end{align}
which defines a linear operator $L \colon D(L) \subset C_b(E) \to C_b(E)$ with $D(L)$ being the set of all $f \in C_b(E)$ for which the limit in \eqref{eq:1.8} exists. In this case $P_t, t \geq 0$, can be recalculated from the operator $(L, D(L))$, called the \textit{infinitesimal generator} of $(P_t)_{t \geq 0}$, through Euler's formula. But as said, this is in general not possible on $(C_b(E), \|\cdot\|_\infty)$.\\

A way out of this, which only works if $E$ is locally compact
(hence excludes, e.g., that $E$ is an infinite dimensional Banach space,
which in turn are the typical state spaces for solutions $X(t)$, $t \geq 0$, to stochastic partial differential equations (SPDEs) or measure-valued Markov processes)
is to replace \eqref{eq:1.5} by
\begin{align}\label{eq:1.9}
P_t f \in C_\infty(E), \text{\;if\;} f\in C_\infty(E),\quad t \geq 0,
\end{align}
where $C_\infty(E)$ denotes the subset of all elements in $C_b(E)$ which vanish at infinity. $(P_t)_{t \geq 0}$,
satisfying \eqref{eq:1.9} are called \textit{Feller semigroups} in the literature, which sometimes leads to confusion,
since the much weaker property \eqref{eq:1.5} is usually called Feller property and the latter makes sense on general topological spaces (see, e.g., \cite{priola}).
But, if $E$ is locally compact and \eqref{eq:1.9} holds, there are a large number of examples,
for which $(P_t)_{t \geq 0}$ is a $C_0$-semigroup on $(C_\infty(E), \|\;\cdot\;\|_\infty)$
and thus uniquely determined by and reconstructable from its infinitesimal generator $(L, D(L))$, i.e.\ from its strong derivative at zero (see e.g.\ \cite{engel}).\
This is usually expressed by the symbolic writing $P_t = e^{tL}, t \geq 0$.
On the other hand, condition \eqref{eq:1.9} is very strong and in general, of course, not fulfilled, even if $E = \R^d$.\\

Another approach is to avoid the $C_0$ (i.e., strong continuity) property and associate to $(P_t)_{t \geq 0}$ an operator $(L, D(L))$, also called \textit{generator} of $(P_t)_{t \geq 0}$, which is obtained by inverting the resolvent of $(P_t)_{t \geq 0}$, which in turn is given by the Laplace transform of $(P_t)_{t \geq 0}$ (see, e.g., \cite{priola}). But this definition of generator uses the whole semigroup $(P_t)_{t \geq 0}$ and is thus definitely not an \textit{infinitesimal} generator of $(P_t)_{t \geq 0}$.\\

Finally, another way out is to replace $C_b(E)$ by an $L^p(E,\mu)$-space, $p \in [1, \infty)$,
for some suitable reference measure $\mu$ on $(E, \mathcal{B}(E))$ (e.g., an invariant measure for $(P_t)_{t \geq 0}$).\
Then $(P_t)_{t \geq 0}$ extends to a $C_0$-semigroup on $L^p(E,\mu)$,
which has a true \textit{infinitesimal} generator there (see e.g.\ \cite{R}, \cite[Section 4]{BKR}, and also \cite{fukushima})
for symmetrizing measures $\mu$.\
Clearly, a symmetrizing or invariant measure does not exist in general for $(P_t)_{t \geq 0}$.
In \cite[Proposition 2.4]{RT}, however, it was proved that a natural reference measure $\mu$ always exists
so that the transition semigroup $(P_t)_{t \geq 0}$ of a Markov process $\mathbb{M}$ as above extends to a $C_0$-semigroup on $L^p(E,\mu)$.
But this measure $\mu$ again is contructed through the resolvent of $(P_t)_{t \geq 0}$, hence again uses the whole semigroup $(P_t)_{t \geq 0}$.
So, the infinitesimal generator of $(P_t)_{t \geq 0}$, extended to a $C_0$-semigroup on $L^p(E,\mu)$, is not really "infinitesimal".
In addition, the analysis of this extension of $(P_t)_{t \geq 0}$, depends on the measure $\mu$ and statements can always be only made $\mu$-a.e.,
and the measure $\mu$ is in no sense unique.\\

So, it has been an open problem whether the transition semigroup of a general Markov process $\mathbb{M}$ as above, which has the Feller property \eqref{eq:1.5}, is infinitesimally generated by its strong derivative at zero and can be reconstructed from the latter through an Euler formula in a ``suitable'' topology on $C_b(E)$. It turns out that such a "suitable" topology is
the well-known mixed topology $\tau_{1}^{\mc M}$ on $C_b(E)$,
i.e. the strongest locally convex topology on $C_b(E)$ which on $\|\cdot\|_\infty$-bounded subsets of $C_b(E)$
coincides with the topology of uniform convergence on compact subsets of $E$ (see Section 2 and Appendix \ref{app.A} for details).\\

The first main contribution of this paper is to identify a notion of (not necessarily linear) $C_0$-semigroups on $(C_b(E),\tau_1^{\mc M})$ (defined solely in terms of $\tau_1^{\mathscr M}$), which, in the linear case, provides a solution to the above problem for a large class of Markov processes.\ The latter follows from a characterizing representation for such $C_0$-semigroups by measure kernels satisfying certain compactness conditions, proved in Theorem \ref{ts1} below.\ As a second main result, we show that, in the nonlinear case, they solve nonlinear partial differential equations of HJB-type in a viscosity sense on general state spaces, cf.\ Theorem \ref{thm.viscosity}.\\

 A sufficient condition that the transition semigroup $(P_t)_{t\geq0}$ of a Markov process as above is such a $C_0$-semigroup is, for instance, the following quite general and in many cases checkable condition (cf.\,\eqref{eq:1.6} and \eqref{eq:1.7} above):
\begin{align}\label{eq:1.10}
[0, \infty) \times K \ni (t, x) \mapsto &\; p_t(x, dy) \in \mathcal{P}(E) \text{\;is continuous in the weak topology on\;}\\
& \mathcal{P}(E) \text{\;for all compact\;} K \subset E, \nonumber
\end{align}
see Proposition \ref{prop:3.5} below. In fact, this is true for very general (not necessarily Polish) state spaces $E$ (see Hypothesis 2.1 below).
So, in such a very general case, we prove that the transition semigroup of a Markov process with right continuous sample paths is uniquely determined
by its strong derivative at zero with respect to the mixed topology $\tau_1^\mc{M}$ on $C_b(E)$ and can be reconstructed from it through an Euler formula (as a consequence of Proposition \ref{pro_gen1} (f)).\\

We would like to mention here that for a special class of stochastic evolution equations on a Hilbert space (see the fundamental book \cite{dz} for the general theory),
the strong continuity of the transition semigroups of their solutions at $t=0$ in the mixed topology was first proved in \cite{gk}
(and the reference therein).
Paper \cite{gk} was a strong motivation for the present paper. The necessity to relax the norm topology on $C_b(E)$ has been well known in the SPDE community and the problem was approached using many different constructions. {The first step towards solving this problem was made in the papers \cite{cerrai},\cite{cg}.  
	Their approach was further developed in \cite{priola} and \cite{dz3},}
but without using an 
underlying topology on $C_b(E)$  making the semigroups of interest strongly continuous. {Therefore, the \textit{infinitesimal} generator could not be identified via the formula analogous to \eqref{eq:1.8}: 
	\[Lf=\lim_{t\to 0}\frac{P_tf-f}{t}\]
	with convergence holding in the underlying topology of the space $C_b(E)$. 
}
\\

For Markov processes arising as laws of solutions to S(P)DEs, whose coefficients are typically unbounded, it is often more suitable to consider its transition semigroup on the space $C_\kappa(E)$ {of all continuous functions $\varphi\colon E\to \R$ with weighted supremum norm $\Vert \vp \Vert_\kappa:=\|\kappa \vp\|_\infty<\infty$ for a suitable} {bounded upper semi}continuous weight function $\kappa\colon E \to (0, \infty)$ and consider the mixed topology $\tm$ on $C_\kappa(E)$. Therefore, we prove the above results for general (not necessarily Markov or positivity preserving) $C_0$-semigroups on $(C_\kappa(E), \tm)$, {where $\kappa$ is assumed to be the pointwise limit of a decreasing sequence of bounded countinuous functions and to be lower bounded on each compact by a positive constant (depending on the compact). We point out that, while the case $\kappa\equiv 1$ has been extensively discussed in topological measure theory, the case of a general $\kappa$ has not been considered in the context of the mixed topology. If $\kappa$ is continuous, many results transfer using the fact that $C_\kappa (E)= \frac1\kappa \cdot C_b(E)$ and that $\vp \mapsto \kappa \vp$ is an isomorphism between $C_\kappa(E)$ and $C_b(E)$ endowed with the respective norm topology, mixed topology, and topology of uniform convergence on compacts. However, in many examples of interest $\kappa$ cannot be chosen as a continuous function but rather upper semicontinuous with additional properties.\ Since for noncontinuous $\kappa$ the map $\vp \mapsto \kappa \vp$ is not an isomorphism anymore, transferring the classical results from the case $\kappa\equiv1$ to general weight functions $\kappa$ as described above is a delicate issue.\ Relying on measure-theoretic rather than topological arguments, the Appendix \ref{app.A} contains a detailed account how to incorporate this class of noncontinuous weight functions into our setup.}\\ 

In the linear case, our $C_0$-semigroups fall into the class of $C_0$-semigroups on locally convex spaces, see \cite{yosida}, \cite{kraaij}, hence are infinitesimally generated by their strong derivative at zero in $\tm$. Furthermore, we prove that they also fall into the class of abstract bicontinuous semigroups on Banach spaces, cf. \cite{kuhnemund}, though their definition includes an exponential $\|\cdot\|_\kappa$-bound for the semigroup, hence are not defined only using the mixed topology. But this norm bound can be proved to hold for our $C_0$-semigroups, see \eqref{eq:3.2}. Also in the definition of the infinitesimal generator of a bicontinuous semigroup a $\|\cdot\|_\kappa$-bound is required. However, by Theorem \ref{ts2} below, it follows that our infinitesimal generator, which is defined purely in terms of the mixed topology on $\ck$ (see Defnition \ref{def_lw}), automatically satisfies this $\|\cdot\|_\kappa$-bound. Hence, by \cite[Theorem 5.6]{farkas-budde}, resp. \cite{kuhnemund} our $C_0$-smigroups, indeed, can be recalculated from its $\tm$-strong derivative at zero through an Euler formula. A different yet related concept is that of semigroups in norming dual pairs as studied in \cite{kunze}. We refer to Remarks \ref{rem.kunze} and \ref{rem:5.6} below for more details.\\

Another contribution of this paper is to provide a (surprisingly) large class of examples of Markov processes, for which condition \eqref{eq:1.10} on finite and infinite dimensional state spaces holds, see Section \ref{sec:examples}. We start with Markov processes, which are the laws of solutions to SDEs on Hilbert spaces $H$ (taking the role of $E$),
including, e.g., the 2D-stochastic Navier-Stokes equations as well as stochastic (fast and slow diffusion) porous media equations
(see Section \ref{Section 4.1}). Here we consider both the norm topology on $H$ and (in Section \ref{sec:bw}) also the $bw$-topology on $H$.
Furthermore, we look at a class of SPDEs with Levy noise on Banach spaces $E$, more precisely SDEs of Ornstein-Uhlenbeck (O-U) type,
but driven by Levy noise (see equation \eqref{eq:4.17'}).
Their corresponding transition semigroups, called \textit{generalized Mehler semigroups},
also turn out to be $C_0$-semigroups on $(C_b(E), \tau_1^\mc{M})$ both when $E$ is considered with the norm topology (see Section \ref{sec:norm.top}) and,
provided $E$ is reflexive, also with the $bw$-topology (see Section \ref{Section 4.4}).
The interesting feature of the $bw$-topology is that in this case $C_b(E)$ consists of all bounded sequentially weakly continuous functions on $E$ and supremum norm balls are relatively $bw$-compact. \\


Section \ref{Section5.1} is devoted to the \textit{infinitesimal generator} $(L, D(L))$ of a
$C_0$-semigroup on $(C_\kappa(E), \tm)$. We introduce the notion of \textit{Markov core operators} for $(L,D(L))$ (see Definition \ref{def:5.7}) and prove that a sufficient condition for being a Markov core operator $(L_0,D(L_0))$ for $(L,D(L))$ is,
that the Fokker-Planck-Kolmogorov equation for $(L_0,D(L_0))$ has a unique solution for all Dirac measures $\delta_x, x \in E$ (see Theorem \ref{prop:5.9}). This is another main result of this paper, because it can be applied to a number of (highly nontrivial) examples, where we identify the Kolmogorov operator of a large class of SDEs on $\R^d$ (see Section \ref{ex:5.9}) or on a Hilbert space $H$ (see Section \ref{ex:L3}) as a Markov core operator for the infinitesimal generator $(L, D(L))$
of the $C_0$-semigroup on $(C_\kappa(\R^d), \tm)$ and $(C_\kappa(H), \tm)$, respectively, given by the transition semigroup of the SDE's solutions.
Furthermore, in Section \ref{ex:Generalized Mehler semigroups on Hilbert spaces} using results from \cite{lescot}, we identify the Kolmogorov operator of SDE \eqref{eq:4.17'}, i.e., the SDE for the O-U-process with Levy noise on a Hilbert space $E$,
which is a pseudo-differential operator (see equation \eqref{eq:5.17}),
as a core operator for the generator $(L, D(L))$ of the corresponding generalized Mehler semigroup on $(C_b(E), \tau_{1}^{\mc M})$.\\

%

Our main result concerning nonlinear $C_0$-semigroups on $(C_\kappa(E), \tm)$ states that they
give rise to viscosity solutions (see Definition \ref{Def. 6.1}) to the Cauchy problem given by its (non-linear) infinitesimal generator (see Theorem \ref{Theorem6.2}), which is of HJB-type. Moreover, we show that every convex Markov $C_0$-semigroup on $(C_b(E), \tm)$ gives rise to a notion of a nonlinear Markov process under a convex expectation (see Theorem \ref{stochrep}).\
This provides an analytic counterpart to the recent investigations of $G$-expectations and nonlinear Markov processes, see \cite{peng2010nonlinear}.
The latter appear in the context of financial modeling in terms of a Brownian motion under volatility uncertainty. For generalizations to uncertainty in the generators of Levy processes and a class of Feller processes, we refer to \cite{NutzNeuf},\cite{MR3941868},\cite{dkn0},\cite{NR}.\
In this context and, more generally, in Mathematical Finance, the so-called continuity from above on $C_b(E)$ of related risk measures plays an important role.\ {We point out that continuity from above is, for convex increasing operators, equivalent to continuity in the mixed topology, see \cite{nendel}.}\\

The paper is organized as follows.\\

Section \ref{section2} contains our setup and necessary definitions, in particular, those concerning the mixed topology. In Section \ref{sec:semigroups}, we introduce our notion of $C_0$-semigroups on
$(C_\kappa(E), \tm)$ (see Definition \ref{def:3.1_semigroup}) and provide a characterizing measure representation for these in the linear case (see Theorem \ref{ts1}). Section \ref{sec:examples} contains the previously described examples. Section \ref{Section5.1} is devoted to the infinitesimal generator of our (possibly non-linear) $C_0$-semigroups on $(\ck,\tau_\kappa^{\mathcal M})$ and the study of Markov core operators, which is then applied in Section \ref{ex:5.9} and Section \ref{ex:L3}. In Section \ref{ex:Generalized Mehler semigroups on Hilbert spaces}, in the case of generalized Mehler semigroups, the infinitesimal generator is identified as a pseudo-differential operator. In Section \ref{Section6}, we study convex $C_0$-semigroups and show that they provide viscosity solutions to HJB-type differential equations. 
Section \ref{sec.control} contains examples from stochastic optimal control as applications of the results in Section \ref{Section6}, both on finite (Section \ref{dynamics}) and infinite dimensional (Section \ref{Section7.2}) state spaces.\ {In Appendix \ref{app.A0} we show that the dual space of a separable reflexive Banach space endowed with the bw$^\star$-topology satisfies our global Hypothesis \ref{hyp_space}.} For the reader's convenience, we recall some basic facts on the mixed topology in Appendix \ref{app.A},  and give corresponding references there. Since $\kappa$ is not continuous, many classical results from topological measure theory do not apply directly, since they are only proved for $\kappa = 1$. Therefore, we also include proofs of their generalizations to the case of upper semicontinuous weight functions $\kappa$, in Appendix \ref{app.A}.

\section{Basic definitions and setup}\label{section2}
In this section, we recall basic definitions and some properties of the mixed topology on weighted spaces of continuous functions $\vp\colon E\to\R$. A very general definition of this topology was introduced in \cite{wiweger} and, in the special case of the space of bounded continuous functions defined on a completely regular topological space $E$, it was studied in topological measure theory as one of the strict topologies, see \cite{wheeler1}. In this paper we restrict our attention to a special class of completely regular topological spaces $E$, but many results presented in this section, when appropriately reformulated, hold  for larger classes of spaces, or even for every completely regular Hausdorff topological space. \\
The following hypothesis about the space $E$ is assumed to hold throughout the paper and will not be enunciated again. 
\begin{hypothesis}\label{hyp_space}
The space $E$ is a completely regular Hausdorff topological space, such that 
\begin{enumerate}
\item
 compact subsets of $E$ are metrizable, 
 \item
 the Borel $\sigma$-algebra $\mc B(E)$ is identical with the Baire $\sigma$-algebra ${\rm Ba}(E)$, i.e.\ the $\sigma$-algebra generated by $C_b(E)$. 
 \item
 a function $\vp:E\to\R$ is continuous if and only if $\vp$ is continuous on every compact subset of $E$. 
 \end{enumerate}
 \end{hypothesis}
 \begin{remark}\label{rem2.2}\
\begin{enumerate}
	\item[a)] Topological spaces that satisfy condition (3) are known as $k_f$- or $k_{\R}$-spaces, see \cite{jarchow} or \cite{wheeler1}.
	\item[b)] Polish spaces satisfy all three conditions of Hypothesis \ref{hyp_space}.
	\item[c)] Let $E=F^\star$ be the dual of a separable Banach space $F$ endowed with its bw$^\star$ topology. Then $E$ satisfies Hypothesis \ref{hyp_space}, see Appendix \ref{app.A0} for details. In particular, every separable reflexive Banach space endowed with the bw$^\star$-topology satisifies Hypothesis \ref{hyp_space}. In this latter case we also use the term $bw$-topology (i.e. we drop the $\ast$).
\end{enumerate}
\end{remark}
Throughout, we consider a bounded weight function $\kappa\colon E\to (0,\infty)$ that satisfies the following 
\begin{hypothesis}\label{hyp-kappa}\
	\begin{enumerate}
	\item There exist $\kappa_j \in C(E),\ j\in\N$, such that 	$\kappa_j \downarrow \kappa$ on $E$.
	\item $\inf_{x\in C}\kappa(x)>0$ for every compact $C\subset E$.
	\end{enumerate}
\end{hypothesis}

\begin{remark}
If $E$ is a separable Banach space, or if $E$ is the dual of a separable Banach space equipped with the $bw^\ast$-topology, then Hypotesis \ref{hyp-kappa}(1) holds, if $\kappa$ is upper-semicontinuous (see \cite{tong}).
\end{remark}

Let $\ck$ denote the space of continuos functions $\vp:E\to\R$ endowed with the norm 
\[\|\vp\|_\kappa=\sup_{x\in E}|\kappa(x)\vp(x)|<\infty\,.\]
By Hypothesis \ref{hyp-kappa}(2) and Hypothesis \ref{hyp_space}(3) $\ck$ is a Banach space. Clearly, $C_\kappa(E)=C(E) \cap \frac{1}{\kappa}B_b(E)$, where {$B_b(E)$ denotes the space of all $\varphi\colon E \rightarrow \R$, which are bounded and Borel measurable.} If $\kappa$ is continuous, then $C_\kappa (E)= \frac{1}{\kappa}C_b(E)$. For $\kappa\equiv 1$ we use the notation $C_b(E)$ instead of $C_1(E)$. Since $\kappa$ is bounded, we have $C_b(E) \subset C_\kappa(E)$.
\vs
Let $\tu$ and $\tc$ denote respectively, the norm topology on $\ck$ and the topology of uniform convergence on compacts. The topology $\tc$ on $\ck$ is defined by the family of seminorms
\[p_{C}(\vp)=\sup_{x\in C}|\vp(x)|,\quad C\,\,\mathrm{compact}\,.\]
Note that for every $C$ the seminorm $p_C$ is  equivalent to the seminorm 
\[p_{\kappa,C}(\vp)=\sup_{x\in C}|\kappa(x)\vp(x)|\,,\]
by Hypothesis \ref{hyp-kappa}. 
\vs
 The mixed topology $\tm$ on $C_\kappa(E)$ is  defined in \cite{wiweger} as the strongest locally convex topology on $\ck$ that is identical with $\tc$ on the norm-balls of $\ck$. By definition, 
\[\tc\subset \tm\subset\tu.
\]For the reader's convenience the construction and properties of this topology are collected in the Appendix \ref{app.A}. 
For an arbitrary sequence $\left(C_n\right)$ of compact 
subsets of $E$ 
and a sequence $\left(a_n\right)$ of positive numbers with $\lim_{n\to\infty}a_n=0$, we define the seminorm 
\[p_{\kappa,\left(C_n\right),\left(a_n\right)}(\phi ):=\sup_{n\in \N}\big
(a_np_{\kappa,C_n}(\phi )\big)=\sup_{n\in \N} \sup_{x\in C_n}\big( a_n\kappa(x) |\phi(x)|\big).\]
By Theorem \ref{lem_ck_mixed} in Appendix \ref{app.A} the family of seminorms 
\[\{p_{\kappa,(C_n),(a_n)};\,0<a_n\to 0,\,\, C_n\,\,\mathrm{compact}\}\] 
generates the mixed topology $\tm$ on $\ck$. Note that if $\varphi \in C(E)$ such that $p_{\kappa,(C_n),(a_n)}(\varphi) < \infty$, for every such seminorm, then $\varphi \in C_\kappa(E)$. {Indeed, if it were not, there exists a sequence $(x_n)_{n\in \N}\subset E$ with $\kappa(x_n)|\varphi(x_n)|\geq n^2$, so that $p_{\kappa, (\{x_n\}),(\frac1n)}(\varphi)=\infty$.}

We now introduce the dual objects of $C_\kappa(E)$.  Let $M_b(E)$ denote the space of all signed Radon measures $\mu\colon \mathscr B(E)\to \R$ with $|\mu|(E)<\infty$, where $|\mu|$ stands for the total variation measure of $\mu$.\ 
Recall that, under Hypothesis \ref{hyp_space}, every Baire measure is Borel, and that a Borel measure $\mu\colon \mathscr B(E)\to \R$ with $|\mu|(E)<\infty$ is Radon\footnote{In some papers in topological measure theory this is called a tight measure.} if, for every Borel set $B$ and every $\ep>0$, there exists a compact set $C_\ep\subset B$ such that
\[|\mu|\la B\setminus C_\ep\ra<\ep.\]
A family 
$\mathcal F\subset M_b(E)$ is said to be 
tight if, for every $\ve >0$, there exists a compact $C_\ep\subset E$ such 
that 
\[\sup_{\mu\in\mathcal F}|\mu |(E\setminus C_\ep)<\ve .\]
We denote the space of all Radon measures $\mu$ on $\la E,\mc B(E)\ra$ with
\[\int_E\frac{|\mu|(dx)}{\kappa(x)}<\infty\]
by $M_\kappa(E)$, that is, $M_\kappa(E)=\kappa\cdot M_b(E)$. Let $M_\kappa^+(E)$ be the subset of all nonnegative measures in $M_\kappa (E)$.
If $\mu\in M_\kappa(E)$, then the mapping 
\[
C_\kappa(E)\ni\vp\mapsto \int_E \phi\, d\mu \quad\text{is norm-continuous.}
\]
Throughout, we endow $M_\kappa(E)$ with the narrow topology, i.e., the weakest topology such that, for every $\vp\in C_\kappa(E)$, the mapping 
\[M_\kappa(E)\ni\mu\mapsto \int_E\vp\,d\mu\quad\text{is continuous.}\]
By Theorem \ref{th_md} in Appendix \ref{app.A}, the space $M_\kappa(E)$ endowed with the narrow topology is the topological dual of $\la\ck,\tm\ra$. This is a well-known fact if $\kappa = 1$ (see \cite{haydon}, \cite{Le-Cam}, and also \cite{sentilles}).


\section{Strongly continuous semigroups on spaces of continuous functions with mixed topology}\label{sec:semigroups}

In this section, we introduce the notion of strongly continuous and 
locally equicontinuous semigroups on $\left(C_\kappa(E),\tm\right)$, which we will refer to as $C_0$-semigroups. 

\bd{def.semigroup}\label{def:3.1_semigroup}
A family of (possibly nonlinear) operators $P=(P_t)_{t\geq 0}$ on $C_\kappa(E)$ is called a \textit{semigroup} on $C_\kappa(E)$ if
\begin{enumerate}
	\item[(i)] $P_0 \phi=\phi$ for all $\phi\in C_\kappa(E)$,
	\item[(ii)] $P_sP_t\phi=P_{s+t}\phi$ for all $s,t\geq 0$ and $\phi\in C_\kappa(E)$.
\end{enumerate}
The family  $P$ is called a \textit{$C_0$-semigroup} on $\la\ck,\tm\ra$ if it additionally satisfies:
\begin{enumerate}
	\item[(iii)] The semigroup $P$ is locally uniformly equicontinuous on $\tm$-bounded sets, i.e., for every $T\geq 0$ and every $\tm$-bounded set $B\subset \ck$, the family of operators
	$(P_t)_{0\leq t\leq T}$ is $\tm$-uniformly equicontinuous on $B$.\ More precisely, for every $T\geq 0$, every $\tm$-bounded set $B\subset \ck$, every seminorm $p_{{\kappa,\left(K_n\right),\left(a_n\right)}}$, and all $\ve >0$, there 
	exists a seminorm $p_{\kappa,\left(C_n\right),\left(b_n\right)}$ and $
	\delta >0$ such that, for every 
	$0\leq t\le T$ and $\phi_1,\phi_2\in B$,
	\[p_{\kappa,\left(K_n\right),\left(a_n\right)}\left(P_t\phi_1-P_t\phi_2\right)<\ve
	\quad\mbox{\rm if}\quad p_{\kappa,\left(C_n\right),\left(b_n\right)}(\phi_1-\phi_2 
	)<\delta .\]
	\item[(iv)] The semigroup $P$ is strongly $\tm$-right continuous, i.e., $P_t\phi\to\phi$ in $\tm$
	as $t\to 0$ for every $\phi\in C_\kappa(E)$. More precisely,
	for all $\phi\in C_{\kappa}(E)$ and every seminorm $p_{\kappa,\left(K_n\right),\left(a_n\right)}$,
	\[\lim_{t\to 0}p_{\kappa,\left(K_n\right),\left(a_n\right)}\left(P_t\phi 
	-\phi\right)=0.\]
\end{enumerate}
\ed

\begin{remark}\label{rem2}\
\begin{enumerate}
\item[(i)] We note that (iii) and (iv) imply that $P$ is strongly $\tm$-continuous.
Indeed, let $\vp\in \ck$, and observe that there exists some $h_0>0$ such that
\[
r:=\sup_{h\in [0,h_0]}\|P_h\vp\|_\kappa<\infty.
\]
Otherwise, there would exist a sequence $(h_n)_{n\in \N}\subset (0,\infty)$ with $h_n\to 0$ as $n\to \infty$ and $\|P_{h_n}\vp\|_\kappa\geq n$, which is impossible by (iv) and Proposition \ref{t14}.
Let  $T>0$, $B$ be the closed $\|\cdot\|_\kappa$-ball in $\ck$ with radius $r$,  $p_{\kappa,\left(K_n\right),\left(a_n\right)}$ be a seminorm, and  $\epsilon>0$.
Then, by Proposition \ref{t12} (b) and (iii), there exist a seminorm $p_{\kappa,\left(C_n\right),\left(b_n\right)}$ and $\delta>0$ such that,
for all $t\in [0,T]$ and $\phi_1,\phi_2\in B$,
	\[p_{\kappa,\left(K_n\right),\left(a_n\right)}\left(P_t\phi_1-P_t\vp_2\right)<\ve
\quad\mbox{\rm if}\quad p_{\kappa,\left(C_n\right),\left(b_n\right)}(\phi_1-\phi_2 
)<\delta .\]
Now, let $s,t\in [0,T]$ with $t<s$ and $s-t\leq h_0$. Then, $\vp=P_0\vp$ and $P_{s-t}\vp$ are elements of $B$. By (iv),
	\[p_{\kappa,\left(C_n\right),\left(b_n\right)}\left(P_{s-t}\phi 
-\phi\right)<\delta\]
if $s-t$ is sufficiently small, and therefore,
	\[p_{\kappa,\left(K_n\right),\left(a_n\right)}\left(P_s\phi-P_t\vp\right)=p_{\kappa,\left(K_n\right),\left(a_n\right)}\left(P_tP_{s-t}\phi-P_t\phi\right)<\ve.\]
As a consequence, for every $\varphi\in C_\kappa(E)$, by the continuity of $P_t\varphi$ on $E$ we easily obtain that
for all compact $C\subset E$ the map
\[ [0,\infty)\times C \ni (t,x) \mapsto P_t\varphi(x)\]
is continuous.
\item[(ii)] Now let us consider the case, where each $P_t$ of $P$ is a linear operator on $C_\kappa(E)$.
Then
\begin{align}\label{eq:3.1!}
C_T:=\sup\limits_{t \leq T} \sup\limits_{\|\varphi\|_\kappa \leq 1} \|P_t \varphi\|_\kappa < \infty.
\end{align}
To prove this we first fix $t > 0$ and note that $\varphi \mapsto \Vert P_t \varphi\Vert_\kappa$ is lower semicontinuous on $(C_\kappa(E), \Vert \ \Vert_\kappa)$. Indeed, {the map $\varphi\mapsto \kappa(x)|P_t\varphi(x)|$ is continuous for all $x\in E$. Therefore, $\varphi\mapsto \|P_t\varphi\|_\kappa$ is lower semicontinuous since it is the supremum over continuous functions.} 

 Now, let $\varphi\in C_\kappa(E)$.
 By the above and the uniform boundedness principle it suffices to show that
\begin{align*}
	\sup_{t\leq T} \|P_t\varphi\|_\kappa < \infty.
\end{align*}
If this is not the case, there exist $t_n\in[0,T],\ n\in\N$, such that
\begin{align}\label{eq:3.2'}
\|P_{t_n}\varphi\|_\kappa \geq n.
\end{align}
We may assume that $\lim\limits_{n\to\infty} t_n = t \in [0,T]$.
Hence by part (i) of this Remark
\[\tm -\lim_{n\to\infty} P_{t_n}\varphi = P_t\varphi, \]
 and consequently, by Proposition \ref{t14}, $\sup_{n\in\N} \|P_{t_n}\varphi\|_\kappa < \infty$, which contradicts \eqref{eq:3.2'}.\\ \xz
By the semigroup property \eqref{eq:3.1!} is equivalent to: There exist $M \in [1, \infty)$  and $a \in \R$ such that
 \begin{align}\label{eq:3.2} {\| P_t \phi \|_\kappa} \leq M e^{a t} \| \phi \|_\kappa \quad \text{ for all } \vp \in C_\kappa (E)\text{ and } t \geq 0.
 \end{align}
The equivalence of \eqref{eq:3.1!} and \eqref{eq:3.2} is, of course, also true in the nonlinear case. Furthermore, if $P$ consists of linear operators, then (iii) is equivalent to:\\
For every $T>0$ and every seminorm $p_{\kappa, (C_n), ({b_n})}$, there exist a seminorm $p_{\kappa,(K_n),({a_n})}$ such that
\begin{align}\label{eq:3.1}
p_{{\kappa,(C_n),({b_n})}} (P_t\varphi)\leq  p_{\kappa,(K_n),({a_n})}(\vp)\quad \text{for all }\vp\in C_\kappa(E) \text{ and } t\in[0,T].
\end{align}
\end{enumerate}
\end{remark}

\begin{remark}\label{rem.kunze}
	Theorem \ref{ts1} below forms the starting point of our analysis of semigroups with the Feller property \eqref{eq:1.5}, their infinitesimal characterisations in terms of generators, and our main result on Markov uniqueness. It shows that the space $(\ck,\tm)$ provides a natural framework for such a theory. We point out that there is a vast literature that is concerned with building the theory of $C_0$ (in some sense) semigroups on locally convex spaces. We refer to the introduction for some historical comments. Here, we only mention that if $P=(P_t)_{t\geq 0}$ is a $C_0$-semigroup on $(\ck,\tm)$ of linear operators then it is a strongly continuous and locally equicontinuous semigroup of linear operators on the locally convex space $(\ck,\tm)$ in the sense of \cite{yosida}. The definition of \cite{yosida} is very general and does {not yield} necessary and sufficient conditions for the semigroup to be $C_0$, Markov, or to have the Feller property. 
	{Special cases of Theorem \ref{ts1} can also be deduced from} \cite{kunze}, where a very general approach to the theory of strongly continuous semigroups in  norming dual pairs $(X,Y)$ {has been developed}.
It is worth noting that our definition of a $C_0$-semigroup only requires continuity properties in the mixed topology, and a priori we do not postulate, e.g., exponential growth bounds for the operator norm of $P_t$, weak continuity of the semigroup operators, or integrability conditions for orbits, which, however, is crucial for the definition of the resolvent via Laplace transform in \cite{kunze}.
\end{remark}

\begin{theorem}\label{ts1}
	{Let }$P=(P_t)_{t\geq 0}$  be a semigroup of linear operators on $\ck$.\ Then, the following properties are
	equivalent:
	
	\begin{enumerate}
		\item[(a)] {$(P_t)_{t\geq 0}$} is a $C_0$-semigroup on $\left(C_{\kappa}(E),\tm\right)$. 
		\item[(b)] There exists a family of Borel measures 
		$\left\{\mu_t(x,\,\cdot\, )\colon x\in E,\, t\ge 0\right\}\subset M_\kappa(
		E)$ such that:
		\begin{enumerate}
			\item[(1)] The map $E\ni x\mapsto\mu_t(x,B)$ is $\mathcal{B}(E)$- measurable for every 
			$B\in\mathcal B(E)$ and $t\ge 0$.  
			\item[(2)] For every $t \geq 0, \mu_t(\cdot, dy)$ represents $P_t$, i.e., 
			\begin{align}\label{eq:3.3'}
			P_t\phi (x)=\int_E\phi (y)\mu_t(x,dy) \quad \text{for all }\varphi \in C_\kappa(E), 				x \in E.
			\end{align}
			\item[(3)] For every $T\geq 0$, 
			\[\sup_{t\le T}\sup_{x\in E}\la \kappa(x)\int_E\frac{\left|\mu_t\right|(x,dy)}{\kappa(y)}\ra<\infty\,.\]
			\item[(4)] For every $T\geq 0$ and every compact $C\subset E$, the family \label{ts1(3)}
			of measures 
			\[\left\{\frac{\kappa(x)\left|\mu_t\right|(x,dy)}{\kappa(y)}\colon x\in C,\,t\in [0,T]\right\}\]
			is tight. 
         \item[(5)] For every $x\in E$ and any sequence $(x_n)\subset E$ with $\lim_{n\to\infty}x_n=x$ (in $E$), we have 
            \[\lim_{(t,x_n)\to (0,x)}\mu_t(x_n,\, \cdot\, )=\delta_x\]
            in $M_\kappa(E)$, where $\delta_x$ denotes the Dirac measure with barycenter $x$. 
     \end{enumerate}
\end{enumerate}
\end{theorem}

\begin{proof}
{
$(a)\Rightarrow (b)$: \eqref{eq:3.1} and Theorem \ref{thm_B9} from Appendix \ref{app.A} imply the existence of $\mu_t(x, dy)\in M_\kappa(E), t\in [0,T], x\in E,$ satisfying (2).
Assertion (1) follows from a standard monotone class argument, since $C_b(E)\subset C_\kappa(E),$ because $\kappa$ is bounded, and so by \eqref{eq:3.3'} for all $\varphi \in C_b(E)$
\begin{align*}
E \ni x \rightarrow \int \varphi(y) \mu_t(x, dy) \text{ is } \sigma(C_\kappa(E))-{\text{measurable}}
\end{align*}
and since due to Hypothesis \ref{hyp_space}(2) we have $\mathcal{B}(E)=\sigma(C_\kappa(E))$ (see the proof of Theorem \ref{thm_B9}). To prove (3) let $T \geq 0$. Then by Hypothesis \ref{hyp-kappa} for all $x \in E, t \in [0, T]$
\begin{align*}
\int_E \frac{\vert \mu_t\vert(x,dy)}{\kappa(y)}= \lim_{j \rightarrow \infty} \int_E \frac{\vert\mu_t\vert(x,dy)}{\kappa_j(y)}
\end{align*} 
and since $\frac{1}{\kappa_j} \in C_\kappa(E)$, by (2) and (\ref{t14}) we have
	\begin{align*}
	\int_E \frac{\vert \mu_t\vert(x,dy)}{\kappa_j(y)}	
		&\leq 2\sup \bigg\{\bigg\vert\int_E g(y)\mu_t(x,dy)\bigg\vert : 0 \leq g \leq \frac{1}{\kappa_j} \bigg\}\\
		&= 2 \sup\bigg\{\bigg\vert P_t g(x)\bigg\vert: 0\leq g \leq \frac{1}{\kappa_j}\bigg\}\\
		&\leq \frac{2}{\kappa(x)} \sup \bigg\{\Vert P_t g\Vert_\kappa : 0 \leq g \leq \frac{1}{\kappa_j} \bigg\}\\
		&\leq \frac{2C_T}{\kappa(x)} \sup \bigg\{ \Vert g\Vert_\kappa : 0 \leq g \leq \frac{1}{\kappa_j}\bigg\}\\
		&\leq \frac{2C_T}{\kappa(x)},\\	
	\end{align*}
where $C_T$ is the constant on the left hand side of inequality \eqref{eq:3.1!}. This implies (3).\\

To prove (4) let $T \geq 0, C \subset E, C$ compact. Define $C_l :=C, b_l:= \frac{1}{l}, l \in \N$. By Remark \ref{rem2}(ii) there exists a seminorm $p_{\kappa, (K_n), (a_n)}$ such that \eqref{eq:3.1} holds, where we may assume that $K_n \subset K_{n+1}, n\in \N$. Let $\varepsilon > 0$. We are going to use the following claim: 

\begin{claim}{}\label{Claim_0}
There exists $m\in\N$ such that for every $\varphi \in C_b(E)$ with $0 \leq \varphi \leq 1$ and $\varphi=0$ on $K_m$ we have 
	\begin{align*}
		\sup_{t\in[0,T]} \sup_{x\in C} \kappa(x) \int_E \varphi(y) \frac{1}					{\kappa(y)} \vert \mu_t \vert (x, dy) \leq \varepsilon.
	\end{align*}
\end{claim}
To prove the Claim let $m\in \N$ such that
	\begin{align*}
		\sup_{l \geq m} a_l \leq \frac{\varepsilon}{2}.
	\end{align*}
Then for all $\varphi$ as in the Claim by (2), \eqref{eq:b1_prime}, and \eqref{eq:3.1} we have for $t \in [0, T]$ and every $j \in \N$, since $\varphi \cdot \kappa^{-1}_j \in C_\kappa(E)$,
	\begin{align*}
\sup_{x \in C} \kappa(x) \int_E \varphi(y)\frac{1}{\kappa_j(y)}\vert\mu_t\vert (x,dy)
&\leq\ 2 \sup_{x\in C}
	\kappa(x)\sup \bigg\{ \bigg\vert \int_E g(y)\mu_t(x, dy) \bigg\vert : 0 \leq g \leq \frac{\varphi}{\kappa_j}\bigg\}\\
	&=\ 2 \sup_{x\in C}\kappa(x) \sup \bigg\{\bigg \vert P_t g(x)\bigg\vert : 0\leq g \leq \frac{\varphi}{\kappa_j}\bigg\}\\
	&\leq\ 2\sup\bigg\{\sup_{l \in \N} b_l \sup_{x\in C_l}\kappa(x) \bigg\vert P_t g(x) \bigg\vert : 0 \leq g \leq \frac{\varphi}{\kappa_j}\bigg\}\\
	&\leq\ 2\sup \bigg\{p_{\kappa,(K_n),(a_n)}(g): 0\leq g \leq \frac{\varphi}{\kappa_j}\bigg\}\\
	&\leq\ 2 p_{\kappa,(K_n),(a_n)}\bigg(\frac{\varphi}{\kappa_j}\bigg)=\ 2 \sup_{l \geq m} a_l \sup_{x\in K_l}\bigg(\kappa(x)\frac{\varphi(x)}{\kappa_j(x)}\bigg)\\
	&\leq\ 2\sup_{l\geq m}a_l \leq \varepsilon.
		\end{align*}
Letting $j \rightarrow \infty$ we prove the Claim.\\

Let $t \in [0, T], x \in C$. Since $(\frac{1}{\kappa}\vert \mu_t\vert)(x, dy)$ is Radon, we have 
	\begin{align*}
		\kappa(x) \int_{E \backslash K_m} \frac{1}{\kappa(y)} \vert \mu_t\vert(x, 			dy) 
		= \sup \bigg\{\kappa(x)\int_{\tilde{K}}\frac{1}{\kappa(y)} \vert 				\mu_t\vert(x,dy)
		: \tilde{K} \subset E \backslash K_m, \tilde{K} \text{ compact. 				} \bigg\} 
	\end{align*}
So, let $\tilde{K} \subset E \backslash K_m$ be compact and let $\varphi \in C_b(E), 0\leq \varphi \leq 1$ such that $\varphi = 0$ on $K_m$ and $\varphi = 1$ on $\tilde{K}$. Then by the Claim 
	\begin{align*}
	\kappa(x)\int_{\tilde{K}}\frac{1}{\kappa(y)} \vert \mu_t\vert(x,dy) 
	=\kappa(x) \int_{\tilde{K}} \varphi(y) \frac{1}{\kappa(y)} \vert \mu_t \vert(x,dy) 
	\leq \varepsilon.
	\end{align*}
Hence taking the sup over all compacts $\tilde{K}\subset E\backslash K_m$ we obtain (4).
(5) immediately follows from Definition \ref{def:3.1_semigroup}(iv).
}
\\
 $(b)\Rightarrow (a)$: Assume that $(b)$ is satisfied.
We have to show (iii),(iv) in Definition \ref{def.semigroup}.
In order to show that (iii) is satisfied, let $T>0$.
For $n\in\N$ let $b_n \in (0,\infty)$ {such that $b_n \rightarrow 0$,} and $(C_n)$ be an increasing sequence of compact subsets of $E$. By (4), for every $l\in\N$, there exists an increasing sequence $(K_{l,n})_{n\in \N}$ of compacts in $E$ such that 
  \begin{align}\label{3.4}
  \sup_{t\in[0,T]}\sup_{x\in C_l}\left(\kappa(x) \int_{E\setminus K_{l,n}}\frac{\ |\mu_t|(x,dy)}{\kappa(y)} \right)\leq {2^{-n-l}b_{n+1}} \quad \text{for all } n\in\N.
  \end{align}
   For $n\in\N$ define
  \begin{align*}
  K_n:=\bigcap_{j=n}^\infty K_{ j,n}.
  \end{align*}
  Then $(K_n)$ is an increasing sequence of compacts and for all $n\in\N$, $n \geq l$,
  \begin{align}\label{eq:3.5}
    \sup_{t\in[0,T]}\sup_{x\in C_l}\left(\kappa(x)\int_{E\setminus K_n}\frac{\ |\mu_t|(x,dy)}{\kappa(y)}\right)&\leq \sum_{j=n}^\infty\sup_{t\in[0,T]}\sup_{x\in C_j}\left( \kappa(x)\int_{E\setminus K_{ j,n}}\frac{\ |\mu_t|(x,dy)}{\kappa(y)}\right)\\
     &\leq \sum_{j=n}^\infty 2^{-n-j}b_{n+1}\leq 2^{ -n} b_{n+1} \nonumber
  \end{align}
  by \eqref{3.4}. Hence, in particular, for all $t\in [0,T]$, $l\in \N$, $x \in C_l$
  \begin{align}\label{eq:3.6'}
  \int_{E \backslash \bigcup_{n=1}^\infty K_n} \frac{\vert \mu_t\vert(x,dy)}{\kappa(y)}=0
  \end{align}
  Now, we are going to show \eqref{eq:3.1}. To that end, let $\varphi \in  C_\kappa(E)$. By homogeneity we may assume that 
  \[p_{\kappa,(K_n),( b_n)}(\varphi)=1,\]
  hence
  \begin{align}\label{eq:3.6}
  p_{\kappa,K_n}(\varphi)\leq b_n^{-1} \quad \text{for all }n\in\N.
  \end{align}
  Setting $K_0:=\emptyset$, by (2) for all $t\in[0,T]$ we then have by \eqref{eq:3.6'} and \eqref{eq:3.6} 
  \begin{align*}
  p_{\kappa,(C_l),(b_l)}(P_t\varphi)
 & \leq\sup_{l\in\N} b_l\sup_{x\in C_l}\left(\kappa(x)\int_E|\varphi|(y)\ |\mu_t|(x,dy)\right)\\
  & \leq\sup_{l\in\N} b_l\sup_{x\in C_l}\left(\kappa(x)p_{\kappa,K_l}(\varphi)\int_{K_l}\frac{\vert \mu_t\vert(x,dy)}{\kappa(y)}\right)\\
  &\quad +\sup_{l\in\N} b_l\sup_{x\in C_l}\left(\sum_{n=l}^\infty p_{\kappa,K_{ n+1}}(\varphi)\kappa(x) \int_{K_{n+1}\setminus K_{n}}\frac{|\mu_t|(x,dy)}{\kappa(y)}\right)\\
  &  \leq\sup_{l\in\N}\sup_{x\in C_l}\la\kappa(x)\int_{K_l} \frac{\vert \mu_t\vert (x,dy)}{\kappa(y)} \ra + \sup_{l\in\N} b_l \sum_{n=l}^\infty b_{n+1}^{-1} \sup_{x\in C_l} \left( \kappa(x) \int_{E \backslash K_n}\frac{\vert \mu_t\vert (x,dy)}{\kappa(y)}\right) ,
  \end{align*}
  which, by \eqref{eq:3.5}, is dominated by
 \begin{align*}
  \sup_{t\in [0,T]}\sup_{x\in E}\left(\kappa(x) \int_E \frac{\vert \mu_t\vert (x,dy)}{\kappa(y)} \right)+  \sup_{l\in\N}b_l ,
  \end{align*}
  where by (3) this constant is finite. Hence by the last part of Remark \ref{rem2} (ii), Property (iii) follows.\\
  We proceed to the proof of (iv). Since by Remark \ref{rem2} (ii) we know that \eqref{eq:3.1!} holds, by Proposition \ref{t14}
  we have to show that, for every compact $K\subset E$,
  \begin{align}\label{eq:3.7}
  \lim_{t\to0}p_{\kappa,K}(P_t\varphi-\varphi)=0.
  \end{align}
  Suppose this does not hold. Then, we can find a compact $K\subset E, \, \ve >0$, $t_n\to 0$, and $\left(x_n\right)\subset K$ such that 
 \begin{align}\label{eq:3.8}
 \left|P_{t_n}\phi\left(x_n\right)-\phi\left(x_n\right)\right|\ge
 \ve \quad \text{for all }n\in \N.
 \end{align}
 Since $K$ is compact and metrizable, there exists some $x\in K$ such that $x_{n_k}\to 
 x$ for a subsequence $\left(n_k\right)$. Since \eqref{eq:3.8} also holds for this subsequence, we get a contradiction to condition (5).\\
\end{proof}
\begin{remark}\label{remark 3.3''}
As just proved above, the dependence of the {sequence $(a_n)_{n\in \N}$} in \eqref{eq:3.1} on the semigroup $P_t$ for $t \in [0, T]$ is only via {$(b_n)_{n\in \N}$ and} the quantity
\begin{align*}
\sup_{t \in [0, T]} \sup_{x \in E} \kappa(x) \int_E \frac{\vert \mu_t\vert (x, dy)}{\kappa(y)},
\end{align*}
where $\mu_t(x, dy), x \in E, t \geq 0$, are the representing measures for $(P_t)_{t \geq 0}$ in \eqref{eq:3.3'}.
\end{remark}

The following proposition renders a convenient sufficient condition to check conditions (4) and (5) in Theorem \ref{ts1}, if $E$ is a so-called Prohorov space, whose definition we recall first (see \cite[Definition 4.7.1(i)]{bogachev2}).

\bd{def.prohorov}
Let $E$ be as above (i.e., as in Hypothesis \ref{hyp_space}). Then $E$ is called a Prohorov space, if every compact subset of $M_{b}^+(E)$ (equipped with the narrow topology) is tight.
\ed

\begin{proposition}\label{prop:3.5}
{Suppose $\kappa$ is continuous} and let $E$ be Prohorov. Let $\mu_t(x,\cdot)\in M_{\kappa}^+(E)$, $t\geq0$, and $x\in E$, such that $E\ni x\mapsto \mu_t(x,B)$ is $\mathcal B(E)$-measurable for all $B\in\mathcal B(E)$, $t\geq0,$ and $\mu_0(x,\cdot)=\delta_x$ for all $x\in E$. Suppose that (3) in Theorem \ref{ts1} holds and that, for every $T\in(0,\infty)$ and every compact $C\subset E$, the map
\[[0,T]\times C \ni (t,x)\longmapsto\int_E\varphi(y)\ \mu_t(x,dy)\]
is continuous for every $\varphi\in C_\kappa(E)$. Then (4) and (5) in Theorem \ref{ts1} also hold. 
\end{proposition}

\begin{proof}
Since the continuous image of a compact set is compact, by the assumptions, it follows that $\{\mu_t(x,\cdot)\colon x\in C,\; t\in[0,T]\}$ is a compact subset of $M_\kappa^+(E)$. Hence (4) holds, since $E$ is assumed to be Prohorov. Condition (5) is fulfilled since $\{x_n\colon n\in \N\}\cup \{x\}$ is compact for every sequence $(x_n)\subset E$ with $x_n\to x\in E$.
\end{proof}

\begin{remark}\label{remproho}
If $E$ is Polish, then $E$ is Prohorov. Likewise, if $E$ is as in Remark \ref{rem2.2}(3), and equipped with the bounded weak topology $\tau_{bw}$, then \cite[Proposition 4.7.6(i)]{bogachev2} implies that $(E,\tau_{bw})$ is Prohorov.
\end{remark}

\section{Examples for linear $C_0$-semigroups on $\la C_\kappa(E),\tm\ra$}\label{sec:examples}
 We are now going to present large classes of examples for $C_0$-semigroups on $(C_\kappa(E),\tau^{\mathscr M}_\kappa)$ given by transition semigroups of solutions to stochastic differential equations (SDEs) on infinite dimensional state spaces, hence including stochastic partial differential equations (SPDEs) as their main examples. The main tool to show that such transition semigroups are indeed $C_0$-semigroups on $(C_\kappa(E),\tau^{\mathscr M}_\kappa)$ will be Proposition \ref{prop:3.5}.
\subsection{Transition semigroups of solutions to SDEs on Hilbert spaces of locally monotone type}\hfill\\ \label{Section 4.1}
  The first class of examples come from SDEs in Hilbert spaces of locally monotone type, introduced in \cite{LR15}. Let us recall the necessary details from \cite[Section 5.1]{LR15}.\\
    Let $E \coloneqq H$ be a separable Hilbert space with inner product $\hscalar{\;}{\;}$
and $H^*$ its dual. Let $V$ be a reflexive Banach space, such that $V\subset H$
continuously and densely. Then 
for its dual space $V^*$
it follows that $H^*\subset V^*$ continuously and densely. Identifying $H$ and
$H^*$ via the Riesz isomorphism we have that
\begin{align*}
   V\subset H \subset V^*
\end{align*}
continuously and densely and if $\pairing{V^*}{\;,\;}{V}$ denotes the
dualization between $V^*$ and $V$
(i.e. $\pairing{V^*}{z,v}{V}:=z(v)$ for $z\in V^*, v\in V$), it follows that
\begin{align*}
   \pairing{V^*}{z,v}{V} = \hscalar{z}{v}\quad \text{for all }z\in H, v\in V.
\end{align*}
$(V,H,V^*)$ is called a \textit{Gelfand triple}\index{Gelfand triple}\label{symbol28}. Note that since $H\subset V^*$
continuously and densely, also $V^*$ is separable, hence so is $V$.
Furthermore, $\mcB(V)$ is generated by $V^*$ and $\mcB(H)$ by $H^*$.
We also have by Kuratowski's
theorem that $V\in \mcB(H), \, H\in \mcB(V^*)$ and
$\mcB(V) = \mcB(H) \cap V,\,\mcB(H)= \mcB(V^*)\cap H.$\\
Let $W(t), \; t\in[0, \infty)$, be a cylindrical Wiener process in a separable  Hilbert space $U$ on a probability space $(\Omega,\mathcal{F},\mathbb{P})$ with normal filtration $\mathcal{F}_t,\; t\in [0, \infty)$. We consider the following stochastic differential equation on $H$:
\begin{equation}\label{SEE}
d X(t)=A(t,X(t)) d t+B(t,X(t)) d W(t),
\end{equation}
 where for some fixed time $T>0$
$$A: [0,T]\times V\times \Omega\to V^*;\  \  B:
[0,T]\times V\times \Omega\to L_{2}(U,H)$$ are progressively
measurable, where $U$ is another separable Hilbert space and $L_2(U,H)$ denotes the set of all Hilbert-Schmidt operators from $U$ to $H$.\\
The coefficients $A$ and $B$ are assumed to satisfy the following conditions:\\ \\
There exist constants
  $\alpha\in ]1,\infty[$, $\beta\in [0,\infty[$, $\theta\in ]0,\infty[$, $C_0\in\mathbb{R}$ and a nonnegative adapted process $f\in
L^1([0,T]\times \Omega; \d
    t\otimes \mathbb{P})$ such that the
 following
 conditions hold for all $u,v,w\in V$ and $(t,\omega)\in [0,T]\times \Omega$:\\
\begin{enumerate}
\item [$(H1)$] (Hemicontinuity) The map $ \lambda\mapsto { }_{V^*}\<A(t,u+\lambda v),w\>_V$ is  continuous on $\mathbb{R}$.\\
\item [$(H2)$] (Local monotonicity)
$$2 { }_{V^*}\<A(t,u)-A(t,v), u-v\>_V +\|B(t,u)-B(t,v)\|_{L_2(U,H)}^2\\ \le \left( f(t) + \rho(v) \right)\|u-v\|_H^2,$$
where $\rho: V\rightarrow [0,+\infty[$ is a measurable hemicontinuous function which is bounded on bounded sets in $V$.\\
\item [$(H3)$] (Coercivity)
$$ 2 { }_{V^*}\<A(t,v), v\>_V +\|B(t,v)\|_{L_{2}(U,H)}^2  \le C_0\|v\|_H^2-\theta \|v\|_V^{\alpha}+f(t).$$
\item[$(H4)$] (Growth)
$$ \|A(t,v)\|_{V^*}^{\frac{\alpha}{\alpha-1}} \le (f(t) + C_0\|v\|_V^{\alpha} ) ( 1 +\|v\|_H^{\beta} ).$$\\
\end{enumerate}
\begin{definition}\label{def:371} A continuous $H$-valued
$(\mathcal{F}_t)$-adapted process $ (X(t))_{t\in [0,T]}$ is called a
solution of $(\ref{SEE})$, if for its $d t\otimes \mathbb{P}$-equivalent
class $\hat{X}$ we have
$$\hat{X}\in L^\alpha([0,T]\times \Omega, d t\otimes\mathbb{P}; V)\cap L^2([0,T]\times \Omega, d t\otimes\mathbb{P}; H)$$
with $\alpha$ in $(H3)$ and $\P$-$a.s.$
$$X(t)=X(0)+\int_0^t A(s, \bar{X}(s))d s+\int_0^t B(s, \bar{X}(s))d W(s),\quad \text{for all } t\in[0,T],  $$
\end{definition}
\noindent where $\bar{X}$ is any $V$-valued progressively measurable $d t\otimes \P$-version of $\hat{X}$.\\ \\
The main existence and uniqueness for \eqref{SEE} then reads as follows (see \cite[Theorem 5.1.3]{LR15}).
\begin{theorem}\label{T1}
Suppose $(H1),(H2),(H3),(H4)$ hold for some  $f\in L^{p/2}([0,T]\times \Omega; d t\otimes \mathbb{P})$ with some $p\ge \beta+2$, and there exists a constant $C$  such that
\begin{equation*}\begin{split}
& \|B(t,v)\|_{L^2(U,H)}^2 \le C(f(t)+\|v\|_H^2), \ t\in[0,T], v\in V; \\
 & \rho(v) \le C(1+\|v\|_V^\alpha) (1+\|v\|_H^\beta), \
v\in V.
\end{split}
\end{equation*}
    Then for every $X_0\in L^{p}(\Omega, \mathcal{F}_0,P;H)$,
    $(\ref{SEE})$
    has a unique solution $(X(t))_{t\in [0,T]}$ such that $X(0)=X_0$. Furthermore, there exsists $C \in [0, \infty)$ such that 
\begin{align}\label{eq:(4.1)tilde}
\mathbb{E} \left(\sup_{t\in[0,T]}\|X(t)\|_H^p\right) \leq C \; \E \big( \|X_0\|^{p}_H + \int_0^T f^{\frac{p}{2}}(t)dt)\big),
\end{align}
where $\mathbb{E}$ denotes expectation w.r.t. $\mathbb{P}$.\\
 Moreover, if
 $A(t,\cdot)(\omega), B(t,\cdot)(\omega)$  are independent of $t\in[0,T]$ and $\omega\in \Omega$,
then the laws $\mathbb{P}\circ X(\cdot,x)^{-1}$, $x\in H$, of the solutions $X(t,x)$, $t\in[0,\infty)$, of \eqref{SEE} started at $x\in H$, form a time-homogeneous Markov process.
\end{theorem}
As shown in \cite[Section 5.1]{LR15} the above framework and Theorem \ref{T1} apply to a large class of SPDEs including the stochastic heat equation (see \cite[Remark 4.1.8]{LR15}),
the stochastic $p$-Laplace equation (see \cite[Example 4.1.9]{LR15}), the stochastic slow diffusion-porous media equation (see \cite[Example 4.1.11]{LR15}),
the stochastic fast diffusion-porous media equation (see \cite{jiagang}), both with general diffusivity,
the perturbed stochastic Burgers equation (see \cite[Lemma 5.1.6 (1) and Example 5.1.8]{LR15}) and the stochastic $2D$ Navier-Stokes equation (see \cite[Example 5.1.10]{LR15}).\\

\noindent For later use we need the following:
\begin{lemma}\label{lemma373}
Consider the situation of Theorem \ref{T1} and let $X(t,x)$, $t\in[0,T]$, be the unique solution of \eqref{SEE} with $X(0,x)=x\in H$. Assume, in addition, that there exists $C_B\in(0,\infty)$ such that
\begin{align}\label{eq:supB}
\sup_{s\in[0,T]}\|B(s,x)-B(s,y)\|_{L_2(U,H)}\leq C_B\|x-y\|_{H},\quad \text{for all } x,y\in V.
\end{align}
Then for all $x,y\in H$
\begin{align*}
&\mathbb{E}\left[exp(-\int_0^T (f(s)+\rho(X(s,y)))ds)\sup_{s\in[0,T]}\|X(s,x)-X(s,y)\|^2_H\right]\\
&\leq e^{\frac92C^2_BT}\|x-y\|^2_H.
\end{align*}
In particular, if $x_n,y\in H$ such that $\lim_{n\to\infty}x_n=y$, then 
\begin{align*}
\sup_{t\in[0,T]}\|X(t,x_n)-X(t,y)\|_H\underset{n\to\infty}{\longrightarrow}0\ \text{ in $\mathbb P$-measure.}
\end{align*}
\end{lemma}
\begin{proof}
Letting
\begin{align*}
F(t):=\exp\bigg(-\int_0^t(f(s)+\rho(X(s,y)))ds\bigg)\ (>0),\quad \text{for all }t\in[0,T],
\end{align*}
we have by It\^o's formula (see e.g. \cite[Theorem 4.2.5]{LR15}) and (H2) that $\forall t\in [0,T]$
\begin{align*}
F(t)&\|X(t,x)-X(t,y)\|^2_H=\|x-y\|^2\\
&+2\int_0^tF(s)\Big({}_{V^*}\<A(s,X(s,x))-A(s,X(s,y)),X(s,x)-X(s,y)\>_V\\
&+\|B(s,X(s,x))-B(s,X(s,y))\|^2_{L_2(U,H)}\Big)ds\\
&-\int_0^tF(s)\big(f(s)+\rho(X(s,y))\big)\|X(s,x)-X(s,y))\|^2_Hds\\
&+\int_0^tF(s)\<X(s,x)-X(s,y),\big(B(s,X(s,x))-B(s,X(s,y))\big)dW(s)\>_H.
\end{align*}
Hence by (H2), the Burkholder-Davis-Gundy inequality with $p=1$ and \eqref{eq:supB}
\begin{align*}
\mathbb{E}&\left[\sup_{s\in[0,t]}\Big(F(s)\|X(s,x)-X(s,y)\|^2_H\Big)\right]\leq\|x-y\|^2_H\\
&\quad +3 \mathbb{E}\left[\bigg(\int_0^tF(s)^2\|B(s,X(s,x))-B(s,X(s,y))\|^2_{L_2(U,H)}\|X(s,x)-X(s,y)\|^2_H\ ds\bigg)^{\frac12}\right]\\
&\leq\|x-y\|^2_H\\
&\quad +3C_B^2\mathbb{E}\left[\sup_{s\in[0,t]}\Big(F(s)^{\frac12}\|X(s,x)-X(s,y)\|_H\Big)\bigg(\int_0^tF(s)\|X(s,x)-X(s,y)\|^2_H\ ds\bigg)^{\frac12}\right]\\
&\leq\|x-y\|^2_H+\frac12\mathbb{E}\left[\sup_{s\in[0,t]}\big(F(s)\|X(s,x)-X(s,y)\|^2_H\big)\right]\\
&\quad +\frac129C_B^2\int_0^t\mathbb{E}\left[\sup_{r\in[0,s]}\big(F(r)\|X(r,x)-X(r,y)\|^2_H\big)\right]ds.
\end{align*}
Hence by Gronwall's lemma $\forall t\geq0$
\begin{align*}
\mathbb{E}\left[\exp\bigg(-\int_0^T(f(s)+\rho(X(s,y)))\,ds\bigg)\sup_{s\in[0,t]}\|X(s,x)-X(s,y)\|^2_H\right]\leq \|x-y\|^2_H\ e^{\frac92C_B^2T}.
\end{align*}
So, if $x_n\to y$ w.r.t. $\|\cdot\|_H$, then 
\[\pushQED{\qed}
\sup_{t\in[0,T]}\|X(t,x_n)-X(t,y)\|_H\underset{n\to\infty}{\longrightarrow}0\quad\text{in $\mathbb{P}$-measure}. \qedhere\popQED
\]
\renewcommand{\qedsymbol}{}
\end{proof}

From now on in this section we assume that the coefficients A and B above do not depend on $\omega \in \Omega,\; t \in [0, \infty)$,
and that (H1)-(H4) hold with some constant $f \in [0, \infty)$ replacing the function $f$.
Furthermore, we assume that \eqref{eq:supB} holds.\\

So, let us now consider the transition semigroup of the unique solution from Theorem \ref{T1}, i.e. for $\varphi\in C_b(H)$, $x\in H$, $t\geq0$,
\begin{align}\label{eq:3.9'}
P_t\varphi(x):&=\mathbb{E}[\varphi(X(t,x))]=\int_\Omega\varphi(X(t,x)(\omega))\mathbb{P}(d\omega)=\int \varphi(y)\ \mu_t(x,dy),
\end{align}
where $X(t,x)$, $t\geq0$, denotes the solution of \eqref{SEE} with initial condition $X(0,x)=x\in H$ and 
\begin{align*}
\mu_t(x,dy):=(\mathbb{P}\circ X(t,x)^{-1})(dy).
\end{align*}
\begin{claim}{1}\label{Claim_1'}
$(P_t)_{t \geq 0}$ is a Markov $C_0$-semigroup on $(C_b(H), \tau_{1}^{\mc M})$.
\end{claim}
\begin{claim}{2}\label{Claim_2'}
Let $m\in [1,\infty)$ and
\begin{align}\label{eq 4.1'}
\kappa(x) := (1 + \|x\|^{m}_H)^{-1},\quad x \in H.
\end{align}
Then $(P_t)_{t \geq 0}$ is a Markov $C_0$-semigroup on $\la C_\kappa(H),\tm\ra$.
\end{claim}
\begin{claimproof}{1}
Clearly, $(P_t)_{t\geq0}$ and $\mu_t(x,dy)$, $t\geq0$, $x\in H$, satisfy conditions (1), (2), (3) in Theorem \ref{ts1}
with $E:=H$ (equipped with its norm topology) and $\kappa=1$.
To show that also (4) and (5) hold, by Proposition \ref{prop:3.5} we have to show that for every compact $C\subset H$ and every $\varphi \in C_b(H)$
\begin{align}\label{ptcont}
[0,T]\times C\ni(t,x)\longmapsto P_t\varphi(x) \text{ is continuous.}
\end{align}
So, let $t,t_n\in [0,T]$ and $x ,x_n\in H$, $n\in\N$, such that
\begin{align*}
(t_n,x_n)\longrightarrow (t,x)\text{ in $[0,T]\times H$ as } n\to\infty.
\end{align*}
Clearly, it then follows by Lemma \ref{lemma373} that
\begin{align*}
X(t_n,x_n)\longrightarrow X(t,x)\quad \text{in $\mathbb{P}$-measure},
\end{align*}
since $X(t_n,x)\to X(t,x)$ $\mathbb{P}$-a.s. Hence $\mu_{t_n}(x_n,\cdot)=\mathbb{P}\circ X(t_n,x_n)^{-1}\longrightarrow\mathbb{P}\circ X(t,x)^{-1}=\mu_t(x,\cdot)$ weakly as $n\to\infty$ and \eqref{ptcont} follows.
Therefore, $(P_t)$ defined in \eqref{eq:3.9'} is a $C_0$-semigroup on $(C_b(E),\tau^{\mathscr M}_1)$ by Theorem \ref{ts1}, and Claim 1 is proved.
\end{claimproof}\\
\begin{claimproof}{2}\label{Claim2'Proof}
Obviously, $(P_t)_{t \geq 0}$ satisfies (1),(2) and (for $\kappa$ as in \eqref{eq 4.1'}) also (3) in Theorem \ref{ts1}, since by \eqref{eq:(4.1)tilde} (applied with $p = m$) for some $C_T \in (0, \infty)$ we have 
\begin{align}\label{eq:4.4'}
P_t\Big(\frac{1}{\kappa}\Big)(x)\leq C_T \frac{1}{\kappa(x)}, \quad \text{for all } t \in [0, T],\, x \in H.
\end{align}
As above (4) and (5) in Theorem \ref{ts1} follow from \eqref{ptcont} above, however, to be proved for all $\varphi \in C_\kappa(H)$. So, let $\varphi \in C_\kappa(H)$ and $t_n \to t$ in $[0, T],\; x_n \to x$ in $H$. Then 
\begin{align*}
 \vert P_{t_n}\varphi (x_n) - P_t \varphi(x)\vert &= \big\vert \E[\varphi (X(t_n, x_n)) - \varphi(X(t, x))]\big\vert\\
&\leq \| \varphi \|_\kappa \E \Big[ \big\vert \|X(t_n, x_n)\|^m_H - \|X(t, x)\|^m_H \big\vert\Big]\\
&\quad + \E \Big[  \big\vert (\varphi\kappa)(X(t_n, x_n)) - (\varphi\kappa)(X(t, x)) \big\vert (1 + \|X(t, x)\|^m_H )\Big],
\end{align*}
which converges to zero as $n \to \infty$, since we already know from the proof of Claim 1 that $X(t_n, x_n) \to X(t, x)$ in $\P$-measure and since $\|X(t_n, x_n)\|_H^m,\; n \in \N$, are uniformly integrable by \eqref{eq:(4.1)tilde},
so that the generalized Lebesgue's dominated convergence theorem applies.
Hence Theorem \ref{ts1} implies Claim 2.
\end{claimproof}

\subsection{Transition semigroups of mild solutions to SDEs on Hilbert spaces with bounded weak topology}\label{sec:bw}\hfill\\
Let $H$ be a separable Hilbert space with inner product $\lb\cdot,\cdot\rb_H$ and norm $\|\cdot\|_H$. We denote the space $H$ endowed with the bounded weak topology by $H_{bw}$. 
In this section we will consider the following stochastic evolution equation in $H$: 
\be{eq_fg}
dX(t)=\la AX(t)+F(X(t))\ra\,dt+G(X(t))\,dW(t),\quad X(0)=x\in H\,.
\ee
We assume that
\begin{itemize}
 \item $W$ is a cylindrical Wiener process on a separable Hilbert space $U$, \\
\item $A$ generates a $C_0$-semigroup $T_t,\ t\geq 0,$ on $H$, \\
\item $F\colon H\to H$ satisfies the Lipschitz condition with a constant $L$:
\[\|F(x)-F(y)\|_H\le L\|x-y\|_H,\quad \text{for all } x,y\in H,\]
\item $G\colon E\to L(U,E)$ ($:=$ all continuous linear operators from $U$ to $H$) is strongly measurable and satisfies the conditions 
\[\left\|T_t G(x)\right\|^2_{L_2(U,H)}\le k(t)\la 1+\|x\|_H^2\ra,\quad \text{for all }x\in H,\]
and 
\[\left\|T_t\la G(x)-G(y)\ra\right\|^2_{L_2(U,H)}\le k(t)\|x-y\|_H^2,\quad \text{for all }x,y\in H,\]
where $k\in L^1_{loc}(0,\infty)$, $k\ge 0$.
\end{itemize}
Under the above assumptions equation \eqref{eq_fg} has a unique mild  solution in $H$ given by the formula 
\[X(t,x)=T_t x+\int_0^t T_{t-s}F(X(s,x))\,ds+\int_0^t T_{t-s}G(X(s,x))\,dW(s),\quad \text{for all }t\in [0,T].\]
Moreover, by standard arguments we find, that for all $x\in H,\,T>0$ 
\be{eq_fg_est1}
\sup_{t\le T}\E\left\|X(t,x)\right\|_H^m\le C_m(T)\la 1+\|x\|_H^m\ra\,
\ee
and  
\be{eq_fg_est2}
\sup_{x\in B_r}\E\left\|X(t,x)-T_t x\right\|_H^m\le C_m(T,r)K_t^{m/2},\quad\text{for all } t\in [0,T], 
\ee
where $B_r$ denotes the open centered ball of radius $r$ in $H$ and 
\[K_t=\int_0^t(1+k(s))\,ds,\quad \text{for all } t\in[0,T].\]
Let $P_t\vp(x):=\E\vp\la X(t,x)\ra$ for $\vp\in C_{\kappa_m}\la H\ra$. Following the arguments from \cite{gk}, we obtain for 
\begin{align}\label{eq: 4.7'}
\kappa_m := (1 + \Vert x \Vert_H^m)^{-1}, \quad m \geq 2,
\end{align}
that
\be{eq_growth}
\|P_t\|_{C_{\kappa_m}(H)\to C_{\kappa_m}(H)}\le M_me^{\gamma_m }t,\quad\text{for all } t\geq0.
\ee
By an easy modification of the proof in \cite{gk} one can prove that under the above assumptions the semigroup $(P_t)$ is a $C_0$-semigroup in $\la C_m\la H\ra,\tau^{\mc M}_{\kappa_m}\ra$
We will show that $\la P_t\ra$ defines also a $C_0$-semigroup on the space $\la C_m\la H_{bw}\ra,\tau^{\mc M}_{\kappa_m}\ra$. The next proposition extends the result in \cite{ms}. 
\bpr{pr_ms1}
Assume that that the semigroup $T_t,\, t>0,$ is compact on $(H,\|\cdot\|_H)$. Then, the semigroup $(P_t)_{ t\geq 0}$ defines a $C_0$-semigroup on $\la C_{\kappa_m}\la H_{bw}\ra,\tau_{\kappa_m}^{\mc M}\ra$. 
\epr
\begin{proof}
By a result in \cite{ms} we have  $P_t C_b\la H_{bw}\ra\subset C_b\la H_{bw}\ra$ for any $t>0$.
Hence by \eqref{eq_growth} it easily follows that $P_t\colon C_{\kappa_m}(H_{bw})\subset C_{\kappa_m}(H_{bw})$ and that
\be{eq_growth_H_bw}
\|P_t\|_{C_{\kappa_m}(H_{bw})\to C_{\kappa_m}(H_{bw})}\le M_me^{\gamma_m t},\quad \text{for all } t\geq0.
\ee
This is the only part of the proof where compactness of $T_t$ is required.\\
We will show that the semigroup $(S_t)$ satisfies conditions (1) - (5) in part (b) of Theorem \ref{ts1}, where $\mu_t(x,V)=\P(X(t,x)\in V)$ for Borel sets $V\subset H$. We recall here that the Borel $\sigma$-algebras of $H$ and $H_{bw}$ coincide and clearly, the mapping 
\[H_{bw}\ni x\longmapsto \mu_t(x,V)\]
is $\mc B\la H_{bw}\ra$-measurable for every $t\ge 0$ and $V\in B\la H_{bw}\ra$, hence condition (1) of Theorem \ref{ts1} holds.
By \eqref{eq_growth_H_bw} condition (2) of Theorem \ref{ts1} is satisfied as well. 
Invoking \eqref{eq_fg_est1} we obtain for all $x\in H,\,T>0$ 
\[
\int_{H_{bw}}\frac{\mu_t(x,dy)}{\kappa_m(y)}=\E\la 1+\|X(t,x)\|_H^m\ra\le C_m(T)\la 1+\|x\|_H^m\ra,\quad \text{for all }t\in [0,T],\]
and condition (3) of Theorem \ref{ts1} follows.
Since $B_r$ is $bw$-compact for every $r>0$ we can use \eqref{eq_fg_est1} again to show that for every $T>0$ and every $r>0$ the family of measures 
\[\left\{\frac{\kappa_m(x)\mu_t(x,dy)}{\kappa_m(y)}\colon x\in B_r,\,t\le T\right\}\]
is tight, which yields condition (4) of Theorem \ref{ts1}.
It remains to prove that conditon (5) is satisfied and it is enough to prove this condition  for $m=0$. Let $\vp\in C_b\la H_{bw}\ra$.
Let $t_n\to 0$ and $x_n\to x$ weakly with $\sup_{n\ge 1}\left\|x_n\right\|_H\le r$ for a certain $r>0$.
For any $\ep>0$ and $T>0$ we can choose $R\geq r$ such that 
\[\sup_{x\in B_r}\sup_{t\le T}\|T_t x\|_H<R,\]
and 
\[\sup_{x\in B_r}\sup_{t\le T}\P\la\left\|X(t,x)\right\|_H>R\ra<\ep\,.\]
Let $\left\{f_k;\,\left\|f\right\|_H=1\,,k\ge 1\right\}$ be a dense set in the sphere $\left\{f\in H;\, \|f\|_H=1\right\}$. We recall that the metric 
\[\rho(x,y)=\sum_{k=1}^\infty\frac{1}{2^k}\,\frac{\left|\lb x-y,f_k\rb\right|}{1+\left|\lb x-y,f_k\rb\right|},\quad \text{for all }x,y\in B_R,\]
defines a Polish topology identical with the weak topology on $B_R$.
We have 
\[
\begin{aligned}
P_{t_n}\vp\la x_n\ra&=\E\vp\la X\la t_n,x_n\ra\ra  I_{B_R}\la  X\la t_n,x_n\ra\ra+\E\vp\la X\la t_n,x_n\ra\ra I_{B^c_R}\la  X\la t_n,x_n\ra\ra\\
&=\vp(x)\E I_{B_R}\la  X\la t_n,x_n\ra\ra+\delta\la  R,t_n,x_n\ra,
\end{aligned}
\]
where 
\[\begin{aligned}
\delta\la R,t_n,x_n\ra&=\E\la\vp\la X\la t_n,x_n\ra\ra-\vp\la x\ra\ra I_{B_R}\la  X\la t_n,x_n\ra\ra+\E\vp\la X\la t_n,x_n\ra\ra  I_{B^c_R}\la  X\la t_n,x_n\ra\ra\\
&=\delta_1\la R,t_n,x_n\ra+\delta_2\la R,t_n,x_n\ra,
\end{aligned}\]
hence 
\begin{align} 
\left|\delta\la R,t_n,x_n\ra\right|\le \left|\delta_1\la R,t_n,x_n\ra\right|+\frac{C\|\vp\|_\infty}{R}. \label{eq_delta}
\end{align}
Let $\omega$ be the modulus of continuity of the function $\vp$ on $B_R$. Then, 
\[\left|\delta_1\la R,t_n,x_n\ra\right|\le \E\,\omega\la\rho\la X\la t_n,x_n\ra,x\ra\ra  I_{B_R}\la  X\la t_n,x_n\ra\ra.
\]
Setting $\psi(t,x):=T_t x$ we obtain 
\[\rho\la X\la t_n,x_n\ra,x\ra\le \rho\la X\la t_n,x_n\ra,\psi\la t_n,x_n\ra\ra+\rho\la \psi\la t_n,x_n\ra,x\ra.\]
For every $f\in H$ the function 
\[[0,T]\times H_{bw}\ni (t,x)\mapsto\lb \psi(t,x),f\rb_H\]
is continuous, hence 
\[\lim_{n\to\infty}\rho\la \psi\la t_n,x_n\ra,x\ra=0.\]
Therefore, invoking \eqref{eq_fg_est2} we find that for every $\ve>0$ 
\[\lim_{n\to\infty}\P\la\rho\la X\la t_n,x_n\ra,x\ra>\ve\ra=0,\]
hence $\left|\delta_1\la R,t_n,x_n\ra\right|\to 0$ for $n\to\infty$. Finally, again by \eqref{eq_fg_est1},
\[\begin{aligned}
\limsup_{n\to \infty}\left|P_{t_n}\vp\la x_n\ra-\vp(x)\right|&\le|\vp(x)|\limsup_{n\to \infty} \E I_{B^c_R}\la  X\la t_n,x_n\ra\ra+\frac{C\|\vp\|_\infty}{R}\\
&\le \frac{2C\|\vp\|_\infty}{R},
\end{aligned}
\]
and condition (5) of Theorem \ref{ts1} follows by taking $R\to\infty$. 
\end{proof}
\subsection{Generalized Mehler semigroups on Banach spaces with norm topology}\label{sec:norm.top}\hfill\\
Let $E$ be a separable Banach space and let $\kappa=1$. Let $(T_t)_{t\geq0}$ be a $C_0$-semigroup of linear operators on $E$. Furthermore, let $\mu_t$, $t\in[0,\infty)$, be probability measures on $(E,\mathcal{B}(E))$ such that:
\begin{align}
&[0,\infty)\ni t\longmapsto \mu_t\in{M}_b(E) \text{ is narrowly (i.e., $\sigma({M}_b(E),C_b(E))$-) continuous.}\label{eq:3.14}\\
&\mu_{t+s}=\left(\mu_t\circ T_s^{-1}\right)*\mu_s\quad t,s\in[0,\infty).\label{eq:3.15}
\end{align}
Define, for $t\in[0,\infty)$, $x\in E$,
\begin{align}\label{eq:3.16}
P_t\varphi(x):=\int_E \varphi(T_tx+y)\ \mu_t(dy),\quad \varphi\in C_b(E).
\end{align}
Then, $(P_t)_{t\geq0}$ is (by \eqref{eq:3.15}) a semigroup of linear operators on $C_b(E)$, called a "generalized Mehler semigroup". In this generality such semigroups have been first introduced in \cite{bogroeschmu} and then further analyzed in \cite{fr} and many other papers (see e.g. the very recent work \cite{afpp} and the references therein). They appear as transition semigroups of Ornstein-Uhlenbeck process with Levy noise, i.e. solutions to the following SDEs on $E$
\begin{align}\label{eq:4.17'}
d X(t)=AX(t)dt+dY(t),
\end{align}
where $A$ is the generator of $(T_t)$ on $E$ and $Y(t)$, $t\geq0$, is the underlying Levy process corresponding to the Levy characteristics appearing in the Levy-Khintchine representation of the exponent of the Fourier transforms of $\mu_t$, $t\geq0$. We refer to \cite{fr} for details. Obviously, $(P_t)$ has a representation as in \eqref{eq:3.3'} with
\begin{align}\label{eq:3.17}
\mu_t(x,dy):=(\delta_{T_tx}*\mu_t)(dy),\quad t\in[0,\infty),\ x\in E.
\end{align}
So, clearly conditions (1)--(3) in Theorem \ref{ts1} hold. To show that $(P_t)$ in \eqref{eq:3.16} is a $C_0$-semigroup on $(C_b(E),\tau_1^{\mathscr M})$, it remains to prove that (4) and (5) hold, for which by Proposition \ref{prop:3.5} it suffices to show that for all $\varphi\in C_b(E)$ the map
\begin{align*}
[0,\infty)\times E\ni(t,x)\longmapsto \int_E\varphi(T_tx+y)\ \mu_t(dy)
\end{align*}
is continuous. So, let $x_n,x\in E$, $t_n,t\in[0,\infty)$ such that $\lim_{n\to\infty}t_n=t$ and $\lim_{n\to\infty}x_n=x$ (w.r.t. the norm topology on $E$). Then we have to show that for all $\varphi\in C_b(E)$
\begin{align*}
\int_E\varphi\ d(\delta_{T_{t_n}x_n}*\mu_{t_n})\longrightarrow \int_E\varphi\ d(\delta_{T_tx}*\mu_t)\quad\text{as $n\to\infty$}.
\end{align*}
By the Portemanteau theorem we may assume that $\varphi$ is Lipschitz with Lipschitz constant less or equal to one. Then, we have
\begin{align*}
\Big|\int_E& \varphi (T_tx+y)\ \mu_t(dy)-\int_E\varphi(T_{t_n}x_n+y)\ \mu_{t_n}(dy) \Big|\\
&\leq\Big|\int_E \varphi(T_tx+y)\ (\mu_t-\mu_{t_n})(dy)\Big|+\|T_tx-T_{t_n}x_n\|_E,
\end{align*}
which clearly converges to zero as $n\to\infty$ by \eqref{eq:3.14} and since $(T_t)$ is a $C_0$-semigroup on $E$.\\ \\
\subsection{Generalized Mehler semigroups on Banach spaces with bounded weak topology}\label{Section 4.4}\hfill\\
Let $\kappa=1$ and $E$ be a reflexive separable Banach space (in particular, $E$ is as in Remark \ref{rem2.2} (3)). Let us now consider $(E,\tau_{bw})$, i.e., $E$ equipped with the bounded weak topology (see Remark \ref{rem2.2} (3)). Then, since $E$ is separable, we have that $\mathscr B((E,\|\cdot\|_E))=\mathscr B((E,\tau_{bw}))$. Let $\mu_t$, $t\in [0,\infty)$, be as in (iii) above, satisfying \eqref{eq:3.15}, but instead of \eqref{eq:3.14}, we assume the weaker condition
\begin{align}\label{eq:3.18}
&[0,\infty)\ni t\longmapsto \mu_t\in\mathscr M_b((E,\tau_{bw}))\ \\
&\text{is narrowly \Big(i.e., $\sigma\big(\mathscr M_b((E,\tau_{bw})),C_b((E,\tau_{bw}))\big)$\Big) continuous.}\nonumber
\end{align}
Let $(P_t)$ be defined as in \eqref{eq:3.16}. We want to show that again by Theorem \ref{ts1} and Proposition \ref{prop:3.5} $(P_t)$ is a $C_0$-semigroup on $C_b\big(((E,\tau_{bw})),\tau^{\mathscr M}_1\big)$. We recall that $C_b((E,\tau_{bw}))$ are exactly the bounded sequentially weak$^*$-continuous functions on $E$ and that each $\tau_{bw}$-compact $C\subset E$ is metrizable (see Remark \ref{rem2.2} (3)). Obviously $(P_t)$ is a semigroup of linear operators on $C_b((E,\tau_{bw}))$ satisfying conditions (1)-(3) in Theorem \ref{ts1}. It remains to prove (4) and (5), which again will follow by Proposition \ref{prop:3.5}. So let $t_n\to t$ in $[0,T]$, $x_n\to x$ in $(E,\tau_{bw})$ and $\varphi\in C_b((E,\tau_{bw}))$. We have to show that
\begin{align}\label{eq:3.19}
\lim_{n\to\infty} P_{t_n}\varphi(x_n)=P_t\varphi(x).
\end{align}
Let us recall the definition of the finitely based $C_b^1$-functions, i.e.
\begin{align*}
\mathscr FC_b^1:=\{f(l_1,\dots,l_m)\ |\ m\in\N,\ f\in C_b^1(\R^m),\ l_1,\dots,l_m\in E^*\}.
\end{align*}
By Corollary \ref{cor_dense} in the Appendix and the Hahn-Banach theorem $\mathscr FC_b^1$ is dense in\\ $C_b((E,\tau_{bw}),\tau_1^{\mathscr M})$. By \eqref{eq:3.16} we have 
\begin{align}\label{eq:3.20}
|P_t\varphi(x)-P_{t_n}\varphi(x_n)|&\leq\left|\int_E\varphi(T_tx+y)\ (\mu_t-\mu_{t_n})(dy) \right|\\
&\quad +\int_E|\varphi(T_tx+y)-\varphi(T_{t_n}x_n+y)|\ \mu_{t_n}(dy).\notag
\end{align}
Clearly, since its integrand is in $C_b((E,\tau_{bw}))$, the first integral on the r.h.s. of \eqref{eq:3.20} converges to zero as $n\to\infty$ by assumption \eqref{eq:3.18}. To see that this also holds for the second, let $\ve>0$. Since $(E,\tau_{bw})$ is a Skorohod space (see Remark \ref{remproho}), by \eqref{eq:3.18} there exists a $\tau_{bw}$-compact set $K_\ve\subset E$ such that 
\begin{align}\label{eq:3.21}
\sup_{t\in[0,T]}\mu_t(K_\ve^c)<\ve.
\end{align}
Since $T_{t_n}x_n\underset{n\to\infty}{\longrightarrow}T_tx$ weakly, there exists a $\tau_{bw}$-compact set $C\subset E$ such that
\begin{align*}
\{T_{t_n}x_n\ |\ n\in\N\}\cup\{T_tx\}\subset C.
\end{align*}
Furthermore, since $K_\ve+C$ is $\tau_{bw}$-compact, there exists $\psi=f(l_1,\dots,l_m)\in\mathscr FC_b^1$ such that 
\begin{align}\label{eq:3.22}
p_{1,K_\ve+C}(\varphi-\psi)<\ve.
\end{align}
Clearly, we may assume that $\|\psi\|_\infty\leq\|\varphi\|_\infty$. Then we can estimate the second term on the r.h.s. of \eqref{eq:3.20} by
\begin{align*}
\int_{K_\ve} &|\varphi-\psi|(T_tx+y)\ \mu_{t_n}(dy)+\int_{K_\ve}|\varphi-\psi|(T_{t_n}x_n+y)\ \mu_{t_n}(dy)\\
&+4\|\varphi\|_\infty\ \mu_{t_n}(K_\ve^c)+\|Df\|_\infty\ |P_m(T_tx-T_{t_n}x_n)|_{\R^m},
\end{align*}
where $P_m(z)=(l_1(z),\dots,l_m(z))$, $z\in E$. Letting first $n\to\infty$ and then $\ve\to0$ by \eqref{eq:3.21}, \eqref{eq:3.22} we obtain \eqref{eq:3.19}.

\section{Strong and Weak infinitesimal generators}

\subsection{Generators and (Markov) core operators}

As ususal, we define the infinitesimal generator as the time derivative at time zero in the underlying topology and the corresponding weak generator.

\bd{def.generator}\label{Section5.1}
Let $\la P_t\ra_{t\geq 0}$ be a $C_0$-semigroup on $\la\ck,\tm\ra$. Then, we define its 
\textit{infinitesimal generator} $L$ by the formula 
\begin{equation}L\phi := \tm-\lim_{t\to 0}\frac {P_t\phi -\phi}t\quad \text{for}\quad \phi\in D(L):= \left\{\psi\in C_{\kappa}(E)\colon \tm- \lim_{t\to 0}\frac {
		P_t\psi -\psi}t\;\; \text{\rm exists}\right\}. \label{a1}
\end{equation}
\ed

In order to formulate the next result, we first recall that, if $X$ is any sequentially complete locally convex linear space, then a continuous function $f\colon [0,T]\to X$ is Riemann integrable, and the function $F(t)=\int_0^tf(s)\,ds$ is differentiable with $\frac{dF}{dt}=f(t)$ for every $t\in(0,T)$ (see \cite{falb} for details). We also recall that, by Theorem \ref{cor_comp}, the space $\la\ck,\tm\ra$ is complete. 

In the next propostion we collect some known properties of $C_0$-semigroups of operators on $\ck$. Parts (b)--(e) were proved in a more general framework in \cite{komura}, part (a) in \cite{albanese2016}. The appearing integrals are all Riemann integrals taking values in the locally convex space $(\ck,\tm)$.
\bpr{pro_gen1}\label{pro_gen1}
Let $P=\la P_t\ra_{t\geq 0}$ be a $C_0$-semigroup on $\la\ck,\tm\ra$ consisting of linear operators with generator $L$.\ Then, the following holds:
\begin{enumerate}
\item[(a)] The $\tm$-closure of $D(L)$ is identical with $\ck$.
	\item[(b)] The generator $L$ is $\tm$-closed, that is for every net $\la \vp_\alpha\ra\subset D(L)$, such that $\vp_\alpha\to\vp$ and $L\vp_\alpha\to\psi$ we have $\vp\in D(L)$ and $L\vp=\psi$. 
\item[(c)] For every $\vp\in D(L)$ we have $P_t\vp\in D(L)$ and $LP_t\vp=P_tL\vp$.
In particular, each $P_t \colon D(L) \to D(L)$ is continuous in the $\tm$-graph topology of $L$ on $D(L)$. 
\item [(d)] For every $\vp\in D(L)$ 
	\[P_t\vp-\vp=\int_{0}^{t} P_sL\vp\,ds\,. \]
Moreover, for every $\vp\in\ck$ 
\[\int_0^tP_s\vp\,ds\in D(L),\quad\mathrm{and}\quad P_t\vp-\vp=L\int_{0}^{t} P_s\vp\,ds\,. \]
	\item[(e)] 
	For every $\lambda >\omega$ and $\vp\in\ck$, the Riemann integral 
	\[J(\lambda )\vp =\int_0^{\infty}e^{-\lambda t}P_t\vp\,dt ,\]
	is convergent in the topology $\tm$ and $J(\lambda)=(\lambda-L)^{-1}$. In particular, $P$ is the unique $C_0$-semigroup
	on $(\ck, \tm)$ of linear operators with generator $L$.
    \item[(f)]
		The Euler formula holds, i.e., for all $\phi \in \ck$		
		\[P_t\phi = \tm - \lim_{n\to\infty}\la\frac{n}{t}(\frac{n}{t}-L)^{-1} \ra^{n}\phi. \]

\end{enumerate} 
\begin{proof}
According to the  remarks preceding this theorem,
we only have to prove (\zx f\xz). But by Remark \ref{rem:5.6} below this is an immediate consequence of Theorem \ref{rem:5.6} in \cite{bff} (attributed there to F. K\"uhnemund, see \cite{kuhnemund})
\end{proof}

\epr
\bpr{pr_core} \label{pr_core}
Let $\mc D\subset D(L)$ be a $\tm$-dense set in $\ck$ such that $P_t\mc D\subset \overline{\mc D}^L$ for every $t>0$, where $\overline{\mc D}^L$ denotes the closure of $\mc D$ in the $\tm$-graph topology of $L$ on $D(L).$
Then $\mc D$ is a $\tm$-core for $L$. 
\epr
\begin{proof}
Since by Proposition \ref{pro_gen1} (c) each $P_{t}$ is continuous in the $\tm$ -graph topology of $D(L)$,
we have that $P_t(\overline{{\mc D}^{L}})\subset\overline{\mc D}^L,$ so we may assume that ${\mc D}=\overline{\mc D}^L.$
Now let $\varphi \in D(L)$.
We have to show that $\varphi \in \overline{\mc D}^L= {\mc D}$. There exists a net $(\varphi_\alpha) \subset {\mc D}$ such that
\[ \tm-\lim_{\alpha}\phi_\alpha=\phi\]
We claim that for every $\alpha$ and $ t\geq0$
\begin{align}\label{(5.1')}
 \int_{0}^{t}P_s\phi_{\alpha}\;ds\in {\mc D}.
\end{align}
Indeed, this integral initially converges in $(\ck , \tm)$,
but again by Proposition \ref{pro_gen1} (c) it also converges in $\overline{\mc D}^L={\mc D}$ in the $\tm$ -graph topology of $L$ on $D(L)$.
So \eqref{(5.1')} holds.
Furthermore, by definition of the Riemann integral in $(\ck , \tm)$ we have for all $t\geq0$
\[ \tm -\lim_{\alpha}\int_{0}^{t}P_s\phi_\alpha\; ds=\int_{0}^{t}P_s\phi\; ds\]
and since by Preposition \ref{pro_gen1} (d)
\[ \tm -\lim_\alpha L\int_{0}^t P_s \varphi_\alpha\; ds=\tm -\lim_\alpha(P_t\phi_\alpha-\phi_\alpha)=P_t\phi-\phi ,\]
we conclude that
\begin{align}\label{(5.1'')}
 \int_0^t P_s\phi\; ds \in \overline{\mc D}^L={\mc D}.
\end{align}
Finally, since $\phi \in D(L)$, it follows that
\[\tm-\lim_{t\to0}\frac{1}{t}\int_0^tP_s\phi\; ds=\phi\] 
and
\[\tm-\lim_{t\to0}L\bigg(\frac{1}{t}\int_0^tP_s\phi\; ds\bigg)=\tm-\lim_{t\to0}\frac{1}{t}\int_0^tP_sL\phi\; ds=L\phi.\]
Therefore, by \eqref{(5.1'')}, $\phi \in \overline{\mc D}^L={\mc D}.$
\\
\end{proof}

\bd{def_lw}  Let $P=\la P_t\ra_{t\geq 0}$ be a $C_0$-semigroup on $\la\ck, \tm\ra$.\ We say that $\vp\in D\la L_w\ra\subset\ck$ if and only if
there exists some $f\in\ck$ such that 
$\frac{1}{t}(P_t\varphi-\varphi) \xrightarrow{t\to 0} f$ weakly in $\la C_\kappa(E),\ \tm\ra$, i.e.,
\begin{equation}\label{a2}
	\lim_{t\to 0}\int_E\frac {P_t\phi (x)-\phi (x)}t\;\nu 
	(dx)=\int_Ef(x)\;\nu (dx)
\end{equation}
for each $\nu\in M_\kappa(E)$. In this case, we define the operator $L_w$ by the formula 
\[L_w\phi =f.\]
We say that $L_w$ is the weak generator of the $C_0$-semigroup $P$ on $\la\ck,\tm\ra$ 
with domain $D\left(L_w\right)$.
\ed


\begin{theorem}\label{ts2}
	Let $\la P_t\ra_{t\geq 0}$ be a $C_0$-semigroup on $\la\ck,\tm\ra$ consisting of linear operators. Then,
	$$L=L_w.$$
	Moreoever, $\phi\in D(L)$ if and only if 
	\begin{equation}\sup_{t\le 1}\left(\frac 1t\left\|P_t\phi -\phi\right
		\|_\kappa\right)<\infty,\label{a3}\end{equation}
	\begin{equation}f(x):=\lim_{t\to 0}\frac {P_t\phi (x)-\phi (x)}t\quad\mbox{\rm exists for all }x\in E,\label{a4}\end{equation}
	and $f\in\ck$. In this case, $f=L\vp$. \\
\end{theorem}

\begin{proof}
We start with the proof that $L=L_w$, which is a modification of the proof of \cite[p.\ 43, Corollary 1.2]{pazy} given for $C_0$-semigroups in Banach spaces.
If $\phi \in D(L)$, then \eqref{a2} follows by Theorem \ref{th_md}.
Hence, $L\subset L_w$. To show that $L_w\subset L$ choose $\phi\in D\left(L_w\right)$.
Then, by the semigroup property, $t\mapsto P_t\varphi$ is a weakly continuously differentiable curve in $\la C_\kappa(E),\ \tm\ra$.
Hence,
\[\int_E\big(P_t\phi (x)-\phi (x)\big)\;\nu (dx)=
 \int_0^t\int_E P_s L_w \phi(x)\; \nu(dx)\; ds= 
\int_E\left(\int_
0^tP_sL_w\phi\;ds\right)(x)\;\nu (dx),\]
where for the secound equality we used that a continuous linear functional on $(C_\kappa(E), \tm)$ interchanges with the $C_\kappa(E)$-valued Riemann integral.
Taking $\nu := \delta_{x}$, for $x \in E$, we get, for all $x \in E$,
\begin{align}\label{eq:5.6'}
P_t\phi(x) -\phi(x) =\int_0^tP_sL_w\phi(x)\; ds = \bigg(\int_0^tP_s L_w\;\phi\; ds\bigg)(x),
\end{align} where the last integral is the Riemann integral in  $\la\ck,\tm\ra$ of $s \mapsto P_sL_w\phi$, which is a continuous curve in $(C_\kappa(E), \tm)$, since $L_w\phi \in C_\kappa (E)$. 
It follows that 
\[\tm-\lim_{t\downarrow 0}\frac {P_t\phi -\phi}t=L_w\phi,\]
 which shows that $\phi\in D(L)$ and concludes the proof that $L=L_w$. 
\\
Assume that \eqref{a3} and \eqref{a4} hold. Then, one immediately sees that \eqref{a2} is satisfied for 
every measure $\nu\in M_\kappa(E)$, hence $\phi\in D(L_w)= D(L)$ with $f=L\phi$ by the first part of the proof. 
Conversely, assume that $\phi\in D(L)$. Then, \eqref{a4} with $f = L\phi$ is obvious and,
by Proposition \ref{t14}, \eqref{a3} holds.
\end{proof}

\begin{remark}\label{rem:5.6}
It is very easy to check that in the linear case our $C_0$-semigroups on $(C_\kappa(E), \tm)$ are special cases of the bi-continuous semigroups introduced in \cite{kuhnemund}.
We also refer to \cite{farkas-budde}, \cite{farkas-sarhir}, \cite{kraaij} and \cite{kunze} for further developments,
and furthermore to \cite{federico}, where only a sequential $C_0$-property is required.
In particular, according to the main result in \cite{kuhnemund} there is a Hille-Yosida-type theorem for characterizing their infinitesimal generators defined in Definition \ref{def.generator}.
Likewise we have a characterization for the latter through Lumer-Phillips-type theorems in the recent papers \cite{budde-wegner} (see Theorems 3.6 and 3.15 therein) and \cite{kruse-seifert}.
\end{remark}
 Next we discuss examples of infinitesimal generators for $C_0$-semigroups on $(C_\kappa(E), \tm)$,
which are given by transition semigroups of solutions to S(P)DEs,
and the relation to the Kolmogorov operators associated to the latter.
In each case, we shall proceed in two steps.
First, we shall prove that the respective infinitesimal generator is an extension of the Kolmogorov operator associated to the latter.
Second, we shall prove ``strong uniqueness'' or at least ``Markov uniqueness'' for the respective Kolmogorov operator,
which thus uniquely determines the infinitesimal generator of the corresponding $C_0$-semigroup on $(C_\kappa, \tm)$.
We start with the following definitions.
\begin{definition}\label{def:5.7}
Let $P_t,\, t\geq 0$, be a $C_0$-semigroup on $\la C_\kappa(E),\,\tm\ra$ with infinitesimal generator $\la L,\, D(L)\ra$
and let $\la L_0,\, D(L_0)\ra$ be a densely defined (i.e., $D(L_0)$ is dense in $\la C_\kappa(E),\,\tm\ra$) linear operator on $C_\kappa(E)$
such that $L_0\subset L$ (i.e., $D(L_0)\subset D(L)$ and $L_0\varphi = L\vp$ for all $\vp\in D(L_0)$).
\begin{enumerate}
	\item[(i)] The operator $\la L_0,\, D(L_0)\ra$ is called a \textit{core operator} for $\la L,\, D(L)\ra$ if the closure of its graph \hspace{0.05cm}
$\Gamma(L_0) = \left\{ \la \vp,\,L_0\vp\ra\in C_\kappa(E)\times C_\kappa(E) \mid \vp\in D(L_0) \right\}$
		in \\
		$\la C_\kappa(E),\,\tm\ra \times \la C_\kappa(E),\,\tm\ra$ coincides with the graph $\Gamma(L)$.
	\item[(ii)] Suppose that $\kappa$ is bounded and that $\la P_t\ra_{t\geq 0}$ is \textit{Markov},
		i.e. $C_\kappa(E)\ni\vp\geq 0\Rightarrow P_t\vp\geq 0,\, t\geq0;$ and $P_t 1 = 1,\, t\geq 0$.
		The operator $\la L_0,\, D(L_0)\ra$ is called a \textit{Markov core operator} for $\la L,\, D(L)\ra$
		if $\la L,\, D(L)\ra$ is the only operator with $L_0\subset L$, which is the infinitesimal generator of a Markov $C_0$-semigroup on $\la C_\kappa(E),\,\tm\ra$.
\end{enumerate}
\end{definition}
\begin{remark}\label{rem:5.8}
Suppose $\la L_0,\, D(L_0)\ra$ is a core operator for $\la L,\, D(L)\ra$.
Then $\la L,\, D(L)\ra$ is the unique operator with $L_0\subset L$, which is the infinitesimal generator of a $C_0$-semigroup on $\la C_\kappa(E),\,\tm\ra$.
Indeed, if $( \tilde L,\, D(\tilde L))$ is another such operator, it follows that $ L\subset \tilde L$, hence $1-L\subset 1-\tilde L$.
But, by Proposition \ref{pro_gen1} (e),
\[ C_\kappa(E) = (1-\tilde L) ( D(\tilde L)) \supset (1- L) ( D(L)) = C_\kappa(E), \]
hence $D(\tilde L) = D(L)$, because $1-\tilde L$ is injective (e.g., again by Proposition \ref{pro_gen1} (e)).
So, $( \tilde L,\, D(\tilde L)) = \la L,\, D(L)\ra$ and, as a consequence, if $\kappa$ is bounded and $(P_t)_{t\geq 0}$ is Markov,
then $\la L_0,\, D(L_0)\ra$ is also a Markov core operator for $\la L,\, D(L)\ra$.
\end{remark}
\begin{theorem}\label{prop:5.9}
Let $\kappa$ be bounded and $(P_t)_{t\geq 0}$ be a Markov $C_0$-semigroup on $\la C_\kappa(E),\,\tm\ra$ with infinitesimal generator $\la L,\, D(L)\ra$
and let $\la L_0,\, D(L_0)\ra$ be a densely defined linear operator on $C_\kappa(E)$ such that $L_0\subset L$.
Suppose that, for every $x\in E$, the Fokker-Planck-Kolmogorov equation
\begin{align}\label{eq:5.5}
\int\vp(y)\;\nu_t(dy) = \int\vp(y)\;\delta_x(dy) + \int_0^t \int L_0\;\vp(y)\;\nu_s(dy)\;ds,\quad t\geq0,\, \vp\in D(L_0),
\end{align}
(see \cite{fpke}) has a unique solution $(\nu_t)_{t\geq0}\in C([0,\infty),\,M_\kappa^+(E))$,
such that $\nu_t(E) = 1$ for all $t \in [0, \infty)$ and such that 
\begin{align}\label{eq:5.7'}
\int_0^T \int_E \frac{1}{\kappa} \; d\nu_t \; dt < \infty, \quad T > 0.
\end{align}
Then, $\la L_0,\, D(L_0)\ra$ is a Markov core operator for $\la L,\, D(L)\ra$, on $(C_\kappa(E), \tm)$.
\end{theorem}
\begin{proof}
Let $( \tilde L,\, D(\tilde L))$ be the infinitesimal generator of a Markov $C_0$-semigroup $(\tilde P_t)_{t\geq0}$ on $\la C_\kappa(E),\,\tm\ra$ such that $L_0\subset\tilde L$.
Let $\tilde\mu_t(x,\cdot),\, x\in E,\, t\geq0$, be its representing measures from Theorem \ref{ts1}.
Clearly, $(\tilde\mu_t(x,\cdot))_{t\geq0}\in C([0,\infty),\,M_\kappa^+(E))$, $\tilde{\mu}_t(x, E) = 1$ for all $t \in [0, \infty)$, and \eqref{eq:5.7'} holds with $\tilde{\mu}_t (x, \cdot)$ replacing $\nu_t$ for all $x\in E$
and by Theorem \ref{ts2} (more precisely, \eqref{eq:5.6'}) it solves \eqref{eq:5.5}.
Hence the assertion follows.
\end{proof}
\noindent Now let us start with an example on a finite-dimensional state space.
In fact, the corresponding SDE on $\R^d$ and the assumptions on the coefficients are the standard ones.
So, in this ``generic case'' our theory of $C_0$-semigroups on $\la C_\kappa(E),\,\tm\ra$ applies
and thus identifies the corresponding Kolmogorov operator $L_0$ with domain $D(L_0) = C_b^2(\R^d)$
as a Markov core operator for the infinitesimal generator of the $C_0$-semigroup on $\la C_\kappa(E),\,\tm\ra$
given by the transition semigroup of the solutions to the SDE.
To the best of our knowledge in this generality this is the first result confirming that
the Kolmogorov operator determines the (truly) infinitesimal generator of the said transition semigroup of the Markov process given by the SDE's solution.
This appears to have been an open problem for many years.

\subsection{Applications to SDEs on $\R^d$}\label{ex:5.9}~\\
Let $E:= \R^d$ and $(\Omega,\mc F,\P)$ be a complete probability space with normal filtration $\mc F_t,\, t\geq0$,
and $(W_t)_{t\geq0}$ be a (standard) $(\mc F_t)$-Wiener process on $\R^{d_1}$.
Let $M(d\times d_1,\, \R)$ denote the set of real $d\times d_1$-matrices equipped with the Hilbert-Schmidt norm $\|\cdot\|$ and let
\begin{align*}
\sigma &\colon \R^d \to M(d\times d_1,\, \R),\\
	 b &\colon \R^d \to \R^d,
\end{align*}
be continuous maps satisfying the following standard assumptions. There exist $K \in L^1_{loc}([0, \infty))$ and $C \in [0, \infty)$ such that for all $R \geq 0$,
\begin{align}\label{eq:5.6}
2\hscalar{x-y}{b(x)-b(y)} + \|\sigma(x)-\sigma(y)\|^2 \leq K(R) |x-y|^2,\quad x,y\in \R^d, \vert x \vert, \vert y \vert \leq R,
\end{align}
and
\begin{align}\label{eq:5.7}
2\hscalar{x}{b(x)} + \|\sigma(x)\|^2 \leq C (1+|x|^2),\quad \text{for all\;} x \in \R^d.
\end{align}
Here $\<\,,\>$ denotes the Euclidean inner product on $\R^d$ and $|\cdot|$ the corresponding norm.
Then it is well-known (see e.g. \cite[Section 3]{LR15} and the references therein) that the SDE
\begin{align}\label{eq:5.8}
dX(t) = b(X(t))dt + \sigma(X(t))dW(t), \quad X(0) = x\in\R^d,
\end{align}
has a unique strong solution $X(t,x),\, t\geq0$, such that for $p \geq 2$ there exists $C_{T, p} \in [0, \infty)$ such that 
\begin{align}\label{eq:5.9}
\E[\sup_{t \in [0, T]}|X(t,x)|^p] \leq C_{T,p}(1+|x|^p),
\end{align}
where $\E$ denotes expectation w.r.t. $\P$. Indeed, \eqref{eq:5.9} is a direct consequence of \eqref{eq:5.7} and It\^o's formula.
For $\geq 1$, let
\begin{align}\label{eq:5.10}
\kappa(x) \coloneqq \la 1+|x|^m\ra^{-1}
\end{align}
and for $\vp\in C_\kappa(\R^d),\, t\geq0,\, x\in\R^d,$
\begin{align}\label{eq:5.10'}
P_t\vp(x) := \E_\P[\vp(X(t,x))] = \int \vp(y) \mu_t(x,dy),
\end{align}
where
\begin{align*}
\mu_t(x,dy)\coloneqq \la\P\circ X(t,x)^{-1}\ra(dy) \in M_\kappa(\R^d)
\end{align*}
(cf.\,\eqref{eq:3.9'}).
By \cite[Proposition 3.2.1]{LR15} exactly the same arguments, which prove Claim 2 in Section 4.1, imply 
that $(P_t)_{t\geq 0}$ is a Markov $C_0$-semigroup on $\la C_\kappa(\R^d),\,\tm\ra$.
Let $\la L,\, D(L)\ra$ be its infinitesimal generator and let us consider the Kolmogorov operator $\la L_0,\, D(L_0)\ra$ corresponding to \eqref{eq:5.8}, defined as
\begin{align}\label{eq:5.10''}
& L_0\vp(x)\coloneqq \frac{1}{2} \sum_{i,j=1}^d \la \sigma(x) \sigma(x)^T \ra_{i,j} \frac{\partial}{\partial x_i}\frac{\partial}{\partial x_j}\vp(x) + \hscalar{b(x)}{\nabla\vp(x)}, \quad x\in\R^d, \\
& \vp \in D(L_0)\coloneqq C_b^2(\R^d).\nonumber
\end{align}
By Corollary \ref{cor_dense} below it easily follows that $C_b^2(\R^d)$ is dense in $\la C_\kappa(\R^d),\,\tm\ra$.
To show that 
\begin{align}\label{eq:5.17!}
L_0 \subset L,
\end{align}
we need one more condition on $b$ and $\sigma$, namely we additionally assume
\begin{align}\label{eq:5.11}
\sup_{x\in\R^d} \frac{|b(x)| + \|\sigma(x)\|}{1+|x|^m} < \infty.
\end{align}
Now let us first show that \eqref{eq:5.17!}.
So, let $\vp\in C_b^2(\R^d)$.
Then by It\^o's formula and \eqref{eq:5.10'} we have for all $x\in\R^d$
\begin{align}\label{eq:5.10'''}
P_t\vp(x) &= \E[\vp(X(t,x))]\\
		  &= \vp(x) + \int_0^t \E[L_0\vp(X(s,x))]\, ds \nonumber\\
		  &= \vp(x) + \int_0^t \int_\R^d L_0\vp(y) \,\mu_s(x,dy)\, ds. \nonumber
\end{align}
Here we note that by \eqref{eq:5.11} for some $c_0\in (0,\infty)$
\begin{align}\label{eq:5.10iv}
|L_0\vp(x)| \leq c_0 (1+|x|^m) \|\vp\|_{C_b^2}\quad \forall x\in\R^d.
\end{align}
So, since $\mu_s(x,dy),\,x\in\R^d,\,s\in[0,\infty)$, satisfy (3) in Theorem \ref{ts1} by \eqref{eq:5.9}, the above use of Fubini's Theorem is justified and $L_0\vp \in C_\kappa(\R^d)$.
\eqref{eq:5.10'''} implies that for all $x\in\R^d$
\begin{align}\label{eq:5.19'}
 \frac{1}{t} \la P_t\vp(x) -\vp(x)\ra = \frac{1}{t} \int_0^t P_s(L_0\vp)(x)\, ds,
\end{align}
So, by the fundamental theorem of calculus
\begin{align}\label{eq:5.21'}
\lim\limits_{t\to 0} \frac{1}{t} \la P_t\vp(x) -\vp(x)\ra = L_0\vp(x).
\end{align}
Since by \eqref{eq:5.9}, for $p=m$, we have
\[ \sup_{t\in[0,T]} P_t\frac{1}{\kappa} \leq 2 C_{T,p} \frac{1}{\kappa}, \]
\eqref{eq:5.10iv} and \eqref{eq:5.19'} imply
\begin{align}\label{eq:5.21''}
\sup_{0\leq t\leq1} \frac{1}{t} \|P_t\vp - \vp\|_\kappa < \infty\,,
\end{align}
and thus Theorem \ref{ts2} implies \eqref{eq:5.17!}.\\

We would like to stress at this point that at least if $\kappa = 1$, the proof of \eqref{eq:5.17!} is completely standard (as it is also in Section \ref{ex:L3} below),
as far as the part \eqref{eq:5.21'} is concerned, while part \eqref{eq:5.21''} is to be taken care of case by case.\\

To show that $(L_0, D(L_0))$ is a Markov core operator for $(L, D(L))$, we furthermore assume that
\begin{align}\label{eq:5.12}
\text{For every compact } K \subset \R^d\text{ there exists }c_K\in(0,\infty)\\
\text{such that for all }\xi=(\xi_1,\dots,\xi_d)\in\R^d \nonumber\\
\sum_{i,j=1}^d \la \sigma(x) \sigma(x)^T \ra_{i,j} \xi_i\xi_j \geq c_K |\xi|^2, \quad  x\in\R^d.\nonumber
\end{align}
\begin{align}\label{eq:5.13}
\text{Each } \la \sigma \sigma^T \ra_{i,j} \text{ is locally in } VMO(\R^d).
\end{align}
We recall that a $\mc B(\R^d)$-measurable function $g: \R^d\to \R$ belongs to the class $VMO(\R^d)$,
if it is bounded and for
\[ O(g,R) \coloneqq \sup_{x\in\R^d} \sup_{r\leq R} |B_r(x)|^{-2} \iint\limits_{y,z\in B_r(x)} |g(y)-g(z)|\;dydz, \quad R\in(0,\infty), \]
we have 
\[ \lim\limits_{R\to0} O(g,R) = 0, \]
where $B_r(x)$ denotes the ball in $\R^d$ of radius $r$, centered at $x\in\R^d$, and $|B_r(x)|$ its Lebesgue measure.
g belongs locally to the class $VMO(\R^d)$ if $\zeta g \in VMO(\R^d)$ for every $\zeta \in C_0^\infty(\R^d)$.\\

Under the assumptions \eqref{eq:5.6}, \eqref{eq:5.7}, \eqref{eq:5.11}, \eqref{eq:5.12}, \eqref{eq:5.13} on the continuous maps $b$ and $\sigma$ above,
it now follows by Proposition \ref{prop:5.9} and Theorem 9.3.6 in \cite{fpke}
that the Kolmogorov operator $\la L_0,\, D(L_0)\ra$ in \eqref{eq:5.10''} corresponding to SDE \eqref{eq:5.8} is a Markov core operator for $\la L,\, D(L)\ra$.
In this case one also says that Markov uniqueness holds for $\la L_0,\, D(L_0)\ra$ on $\la C_\kappa(\R^d),\,\tm\ra$.\\

Now let us give an example on an infinitely dimensional state space,
where we even have that the Kolmogorov operator $\la L_0,\, D(L_0)\ra$ is a core-operator for $\la L,\, D(L)\ra$.
\subsection{Applications to generalized Mehler semigroups (or OU-processes with Levy noise) on Hilbert spaces}\label{ex:Generalized Mehler semigroups on Hilbert spaces}~\\
Let $E$ be a separable Hilbert space with inner product $\langle \cdot \; , \; \cdot \rangle $ and norm $\| \cdot\|_E$
and let us come back to Section \ref{lemma373}, i.e., $(P_t)_{t \geq 0}$ is the semigroup defined in \eqref{eq:3.16},
which as shown there, is a $C_0$-semigroup on $(C_b(E), \tau_{1}^{\mc M})$,  (so $\kappa \equiv 1$).
In order to calculate its infinitesimal generator $(L, D(L))$ explicitly on a core domain, we need some assumptions.
Let $\lambda : E \to \mathbb{C}$ satisfy the following hypothesis:
\begin{enumerate}
\item [$(H1)$] $\lambda$ is negative definite and Sazonov continuous with $\lambda(0) = 0$.
\end{enumerate}
We refer e.g. to \cite[Section 2]{fr} for the corresponding definitions.
Then, as is well-known (cf., e.g., \cite[Theorem VI. 4.10]{partha}), $\lambda$ posesses a unique Levy-Khintchin representation of the form 
\begin{align}\label{eq:5.15'}
\lambda(\xi) := -i \langle \xi, a \rangle + \frac{1}{2} \langle \xi, R\xi \rangle - \int_E \Biggl( e^{i\langle \xi, x \rangle} - 1 - \frac{i\langle \xi, x\rangle}{1 + \|x\|^2_E}\Biggr) M(dx),
\end{align}
where $a \in E, R:E \to E$ a symmetric trace class operator and $M$ a Levy measure on $(E, \mathcal{B}(E))$, i.e. $M(\{0\}) = 0$
and $\int_E \|x\|^2_E \wedge 1 \, M(dx) < \infty$. We note that each $\lambda$ of the form \eqref{eq:5.15'} is automatically Sazonov continuous on $E$.
Obviously, then there exists $D \in [0, \infty)$ such that 
\begin{align}\label{eq:5.15''}
\vert\lambda(\xi)\vert \leq D (1+ \|\xi\|^2_E),\; \xi \in E,\text{ and } \lambda (-\xi) = \overline{\lambda(\xi)}, \quad \xi \in E, 
\end{align}
and the real part of $\lambda$ is non-negative.
Now we shall choose the measures in \eqref{eq:3.16} in the following way.
By \cite[Section 2.1]{fr} the functions 
\begin{align*}
E \ni \xi \mapsto \int_0^t \lambda (T_s^\ast \xi) \;ds,\quad t \geq 0,
\end{align*}
are also negative definite, zero for $\xi = 0$ and Sazonov continuous, where $T_s^\ast$ denotes the adjoint operator of $T_s$ on $E$.
Hence by the Minlos-Sazonov Theorem (see \cite{vakhania}) for each $t \geq 0$ there exists a unique probability measure $\mu_t$ on $(E, \mathcal{B}(E))$ with Fourier transform
\begin{align}\label{eq:5.16}
 \hat{\mu_t}(\xi) := exp \biggl( -\int_0^t \lambda (T_s^\ast \xi)\, ds \biggr) ,\quad \xi \in E.
\end{align}
\eqref{eq:5.16} implies that 
\begin{align*}
\hat{\mu}_{t+s}(\xi) = \hat{\mu_s} (\xi) \hat{\mu_t}(T_s^\ast \xi),\; \xi \in E',\quad t, s \geq 0,
\end{align*}
which in turn is equivalent to \eqref{eq:3.15}.
Hence for such $\mu_t$, $P_t, t \geq 0$ defined as in \eqref{eq:3.16}, form indeed a generalized Mehler semigroup,
hence by Section \ref{sec:norm.top} a Markov $C_0$-semigroup on $\la C_\kappa(E),\tau_1^{\mc M}\ra$.
Now we want to identify its generator $(L, D(L))$ on a convenient and large enough domain which was suggested in \cite{lescot}.
Let us recall its definition.
Let $W_0$ be the set of functions $\varphi$ that have a representation of the form 
\begin{align}\label{eq:5.16'}
\varphi(x) = f \Bigl( \langle \xi_1, x \rangle, \dots,  \langle \xi_m, x\rangle \Bigr),\, x\in E,
\end{align}
for $m \in\N$ and $f \in \mathcal{S} (\R^m, \mathbb{C})$ (i.e., the Schwartz space of complex-valued functions, "rapidly decreasing" at infinity as well as their derivatives).
Obviously, $W_0$ is closed under multiplication.
With the notations above, let $f_0: \R^m \to \mathbb{C}$ denote the inverse Fourier transform of $f$, i.e. the function $f_0$, such that for all $y \in \R^m$
$$f(y) = \int_{\R^m} e^{i\langle y, v \rangle} f_0 (v) dv,$$
and let $\nu(dv) := f_0(v)dv$, where $dv$ denotes Lebesgue measure on $\R^m$. Let $\Pi_m \colon \R^m \to E$ be defined by 
$$\Pi_m (v_1, \dots, v_m) := v_1\xi_1 + \dots + v_m\xi_m,$$
and let $\nu := (\Pi_m)_{\ast}\nu_0$, i.e., the image measure of $\nu_0$ under $\Pi_m$.
Then a simple computation yields that $\varphi = \hat{\nu}$. Let $W$ be the $(\R-)$vector space, generated by the $\R$-valued elements of $W_0$, i.e. those for which 
$$f_0 (- v) = \overline{f_0(v)}, \; v \in \R^m.$$
Let us now recall one of the main results in \cite{lescot}, for which we need to assume the following condition:
\begin{enumerate}
\item[(H2)] There exists an orthonormal basis $\{\xi_n| n \in \N\}$ of $E$, consisting of eigenvectors of the adjoint operator $A^\ast$ of $A$ on E. 
\end{enumerate}
From this it is easy to see that $W$ is an algebra separating the points of $E$. Note also that by definition of the Fr\'{e}chet derivative $\varphi '$ of $\varphi$ it follows that $\varphi'(x) \in D(A^\ast)$ for all $x \in E$. 
\begin{theorem}\label{theorem 5.10} \cite[Theorem 1.1]{lescot}
For all real-valued $\varphi = \hat{\nu} \in W_0$ and any $x \in E$ define 
\begin{align}\label{eq:5.17}
L_0\varphi(x) := \int_{E} \Bigl( i\langle A^{\ast}\xi, x\rangle - \lambda(\xi) \Bigr) e^{i\langle\xi,x\rangle}\nu(d\xi)\quad \text{(Kolmogorov operator)}
\end{align}
and extend $L_0$ by $\R-$linearity to $D(L_0) := W$. Suppose that (H1) and (H2) hold. Then,
\begin{enumerate}
\item[(i)] $L_0$ maps $D(L_0)$ into $C_b (E)$.
\item[(ii)] $P_t \varphi (x) - \varphi(x) = \int_0^t P_s L_0 \varphi (x)\, ds$ for all $\varphi \in D(L_0), x \in E$ and $t \geq 0$. 
\end{enumerate}
\end{theorem}
From this result if follows easily that for $(L_0, D(L_0))$ we have $L_0 \subset L$.
Indeed, by Theorem \ref{theorem 5.10} and its consequence that $s \mapsto P_s L_0 \varphi(x)$ is continuous on $[0, \infty)$ for all $\varphi \in D(L_0), x \in E$, we have that 
\begin{align*}
\frac{d}{dt}_{\upharpoonright t=0} P_t \varphi (x) = L_0 \varphi (x)
\end{align*} 
and for all $t \in [0, 1]$
\begin{align*}
\frac{1}{t} \big|P_t \varphi(x)-\varphi(x)\big| \leq  \sup_{0 \leq s \leq 1} \|P_s L_0 \varphi\|_1
\end{align*}
(where we recall that $\kappa \equiv 1$).
Hence Theorem \ref{ts2} implies that $D(L_0) \subset D(L)$ and $L_0 \varphi = L \varphi$ for all $\varphi \in D(L_0)$, i.e., 
\begin{align}\label{(5.17)'}
L_0 \subset L.
\end{align}
Now we shall prove that $(L_0, D(L_0))$ is a core operator for $(L, D(L))$.
For this, according to Proposition \ref{pr_core} it suffices to prove
\begin{align}\label{5.18}
P_t (D(L_0)) \subset \overline{D(L_0)}^L, \quad t > 0.
\end{align}
\begin{remark}\label{rem:5.12'}
If $\lambda$, restricted to $span\{\xi_1, \cdots, \xi_n\}$, is infinitely often differentiable for all $n \in \N$, then by \cite[Theorem 1.3(i)]{lescot}
\begin{align*}
P_t (D(L_0) \subset D(L_0), \quad t > 0.
\end{align*}
Hence \eqref{5.18} holds. But this is in general not true for general $\lambda$ as above.
\end{remark}
So, to prove \eqref{5.18} let us fix $t > 0$ and $\varphi \in D(L_0)$.
Because $P_t$ is linear, we may assume that $\varphi$ is of type \eqref{eq:5.16'} with $f$ real-valued.
It is easily seen that such a $\varphi$ can be approximated in the $\tm$-graph topology of $L$ on $D(L)$ by $\varphi_n,
n \in \N$, of type \eqref{eq:5.16'} with the corresponding $f_n \in S(\R^m, \mathbb{C})$ being Fourier transforms of $f_{n, 0} \in S(\R^m, \mathbb{C})$ with compact supports.
Hence we may assume that $\varphi$ is of the form \eqref{eq:5.16'} with compactly supported $f_0$. Consider the following approximation of $\lambda$ (see \eqref{eq:5.15'}) for $\epsilon \in (0, 1)$
\begin{align}\label{eq:5.29'tilde}
\lambda_{\varepsilon}(\xi) := - i\langle \xi, a \rangle + \frac{1}{2} \langle \xi, R\xi \rangle - \int_E \Big( e^{i \langle \xi, x \rangle} - 1 - \frac{i\langle \xi, x \rangle}{1 + \|x\|^2_E} \Big) M_\varepsilon(dx),
\end{align}
where
\begin{align*}
M_\varepsilon (dx) := \mathbbm{1}_{\{\varepsilon\; \leq \;\|\cdot\|_E \;\leq \;\frac{1}{\varepsilon}\}} M(dx).
\end{align*}
Obviously, $M_\epsilon$ is again a L\'{e}vy measure on $(E, \mathcal{B}(E))$, and $\lambda_\varepsilon$ satisfies (H1).
Let $\mu^{(\varepsilon)}_t,\, t \geq 0$, be defined analogously to $\mu_t, \, t \geq 0$,
through \eqref{eq:5.16} with $\lambda_\varepsilon$ replacing $\lambda$, and $P_t^{(\varepsilon)}, \, t \geq 0$,
correspondingly through \eqref{eq:3.16} with $\mu_t^{(\varepsilon)}$ replacing $\mu_t$.
Let $(L^{(\varepsilon)}, D(L^{\varepsilon}))$ be the infinitesimal generator of the $C_0-$semigroup $(P_t^{(\varepsilon)})_{t \geq 0}$ on $(C_b(E), \tau_{1}^{\mc M})$.
Since by \cite[Proposition 3.3]{lescot}, each $\lambda_\varepsilon$ fulfills the condition of Remark \ref{rem:5.12'}, the latter implies 
\begin{align}\label{eq:(5.18)'}
P_t^{(\varepsilon)}\varphi \in D(L_0), \varepsilon \in (0, 1).
\end{align}
Hence, if we can prove 
\begin{align}\label{eq:(5.19)}
P_t^{(\varepsilon)}\varphi \;\; \xrightarrow{\varepsilon \to 0}{} P_t \varphi ,
\end{align}
in the $\tau_1^{\mc M}$-graph topology of $L$ on $D(L)$, we obtain that $P_t \varphi \in \overline{D(L_0)}^L$ and \eqref{5.18} is proved.
\eqref{eq:(5.19)} follows from the following two claims, under the additional condition \eqref{eq:5.31'} in Claim 1,
which is, however, is always fulfilled, if $\lambda$ is real-valued (see Lemma \ref{lem:5.12''} below).
\begin{claim}{1}\label{claim1}
Assume that
\begin{align}\label{eq:5.31'}
\{\mu_t^{(\varepsilon)} | \varepsilon \in (0, 1)\}
\end{align}
is tight.\\
Let $g \in C_b(E)$. Then
\begin{align*}
\tau_{1}^{\mc M} - \lim_{\varepsilon \to 0} P_t^{(\varepsilon)} g = P_t g.
\end{align*}
\end{claim}
\begin{claim}{2}\label{claim2}
\begin{align*}
\tau_{1}^{\mc M} - \lim_{\varepsilon \to 0} L_0 P_t^{(\varepsilon)} \varphi = LP_t \varphi.
\end{align*}
\end{claim}
\begin{claimproof}{1}\label{proof:claim1}
By \cite[Corollary 3.5]{lescot} for $\psi \in D(L_0)$ we have
\begin{align}\label{eq:5.20}
\lim_{\varepsilon \to 0} \|P_t^{(\varepsilon)} \psi - P_t \psi\|_1 = 0.
\end{align}
Let $p_{1, (C_n),(a_n)}$ be any of the seminorms generating $\tau_1^{\mc M}$ on $C_b(E)$.
Then there exists a seminorm $p_{1, (K_n),(b_n)}$ and $C \in (0, \infty)$ such that 
\begin{align*}
p_{1, (C_n),(a_n)}( P_t^{(\varepsilon)} \psi ) \leq C \; p_{1, (K_n),(b_n)}, ( \psi )
\end{align*}
for all $\psi \in D(L_0)$ and all $\varepsilon \in [0, 1)$, where we set $P_t^{(0)}:= P_t$. 
The fact that the seminorm $p_{1, (C_n),(a_n)}$ can indeed be taken independent of $\varepsilon \in [0, 1)$ is due to assumption \eqref{eq:5.31'}.
This can be seen as follows: \\
Consider the representing measures $\mu_t^{(\varepsilon)}(x, dy),\, x \in E,\, t \geq 0$, of $P_t^{(\varepsilon)},\, t \geq 0,\, \varepsilon \in [0, 1)$,
which are given by (see \eqref{eq:3.17})
\begin{align}\label{eq:5.32'}
\mu_t^{(\varepsilon)}(x, dy) := (\delta_{T_tx}\ast \mu_t^{(\varepsilon)})(dy).
\end{align}
But by \eqref{eq:5.31'} for every $\delta > 0$ there exists a compact $K_\delta \subset E$ such that 
\begin{align*}
\mu_t^{(\varepsilon)}(K_\delta^c) < \delta \quad \text{ for all \;} \varepsilon \in [0, 1). 
\end{align*}
Let $C \subset E$ be compact. Define
\begin{align*}
\tilde{K_\delta} := K_\delta + T_t C.
\end{align*}
Then $\tilde{K_\delta}$ is a compact subset of $E$ and 
\begin{align*}
(K_\delta + T_t C - T_t x)^c \subset K_\delta^c \quad x \in C. 
\end{align*}
Hence, since $(K_\delta + T_t C)^c - T_tx \subset (K_\delta + T_t C - T_tx)^c$, we have for all $x \in C$
\begin{align*}
(\delta_{T_tx} \ast \mu_t^{(\varepsilon)})(\tilde{K_\delta^c}) &= \mu_t^{(\varepsilon)}((K_\delta + T_tC)^c - T_tx)\leq \mu_t^{(\varepsilon)}(K_\delta^c) < \delta \quad \varepsilon \in [0, 1).
\end{align*}
Therefore, by \eqref{eq:5.32'}, $\{\mu_t^{(\varepsilon)}(x, dy)| \varepsilon \in [0, 1), x \in C\}$ is tight.
Hence exactly the same arguments as in the proof of $"(4) \Rightarrow (iv)"$ in the proof of Theorem \ref{ts1},
applied to $\{\mu_t^{(\varepsilon)}| \varepsilon \in [0, 1), x \in C\}$ (with $t$ fixed),
and Remark \ref{remark 3.3''} imply that $p_{1, (K_n),(b_n)}$ can be taken independent of $\varepsilon \in [0, 1)$.\\

Hence for all $\psi \in D(L_0)$
\begin{align}\label{eq:5.21}
p_{1, (C_n),(a_n)}\Big( P_t g - P_t^{(\varepsilon)}g \Big) \leq 2C p_{1, (K_n),(b_n)} (g - \psi) + \sup_{n \in \N} a_n \|P_t^{(\varepsilon)} \psi - P_t \psi\|_1.
\end{align}
Since $D(L_0)$ is dense in $(C_b(E), \tau_{1}^{\mc M})$ by Corollary \ref{cor_dense} below, \eqref{eq:5.20} and  \eqref{eq:5.21} imply Claim 1.
\end{claimproof}
\begin{claimproof}{2}\label{proof:claim2}
By Proposition \ref{pro_gen1}(c) and \eqref{(5.17)'}, \eqref{eq:(5.18)'} we have for all $\varepsilon \in (0,1)$
\begin{align}\label{eq:5.22}
LP_t \varphi - L P_t^{(\varepsilon)}\varphi = \big( P_tL_0 \varphi - P_t^{(\varepsilon)} L_0 \varphi\big) + P_t^{(\varepsilon)} \big( L_0 - L^{(\varepsilon)}\big)\varphi + \big( L^{(\varepsilon)} - L_0\big) P_t^{(\varepsilon)} \varphi.
\end{align}
By Claim 1 we have that as $\varepsilon \to 0$ the first summand converges to zero in $(C_b(E), \tau_{1}^{\mc M})$.
For the second summand we have
\[\|P_t^{(\varepsilon)}(L_0 - L^{(\varepsilon)})\varphi)\|_1 \leq \|(L_0-L^{(\varepsilon)})\varphi\|_1.\]
Defining $L_0^{(\varepsilon)}$ with domain $D(L_0)$ as in \eqref{eq:5.17} with $\lambda_\varepsilon$ replacing $\lambda$ and applying \eqref{(5.17)'} with $L, L_0$ replaced by $L^{(\varepsilon)}, L_0^{(\varepsilon)}$ respectively, we obtain that for every $x \in E$ 
\begin{align}\label{eq:5.34'}
L_0^{(\varepsilon)} \subset L^{(\varepsilon)}
\end{align}
and hence 
\begin{align}\label{eq:5.23}
(L_0 - L^{(\varepsilon)})\varphi(x) &= \int_E (\lambda_\varepsilon(\xi) - \lambda(\xi)) e^{i\langle\xi,x\rangle} \, \nu(d\xi) \\
	&= \int_{\R^m} \big( \lambda_\varepsilon \big( \Pi_m (v)\big) - \lambda \big( \Pi_m(v)) \big) e^{i \langle \Pi_m(v), x\rangle} f_0(v)dv.\nonumber
\end{align}
By \cite[Lemma 3.1]{lescot}, $\lambda_\varepsilon \to\lambda$ uniformly on bounded subsets of $E$ as $\varepsilon\to 0$,
so by \eqref{eq:5.23} and because $f_0$ has compact support,
we obtain
\[\lim_{\varepsilon \to 0} \|(L_0 - L^{(\varepsilon)})\varphi\|_1 = 0,\]
so by \eqref{def.semigroup} and Remark \ref{rem2} (ii) the second summand in the r.h.s. of \eqref{eq:5.22} also converges to zero in $(C_b(E), \tau_{1}^{\mc M})$.\\
Now let us turn to the third summand.
So, let $x \in E, \varepsilon \in (0, 1)$.
Then by an elementary calculation (see \cite[(1.2)]{lescot})
\begin{align*}
P_t^{(\varepsilon)}\varphi(x) &= \int_E e^{i\langle x, T_t^{\ast}\xi\rangle} \; \widehat{\mu_{t}^{(\varepsilon)}}(\xi) \nu(d\xi) = \widehat{\nu_t^{(\varepsilon)}}(x),
\end{align*}
where $\nu_t^\varepsilon$ is the image measure of $\widehat{\mu_{t}^{(\varepsilon)}}(\xi)\nu(d\xi)$ under $T_t^\ast : E \to E$.
By (H2) we know that $A^\ast \xi_j = \alpha_j \xi_j, j \in \N$, for some $\alpha_j \in \R$.
Hence by Euler's formula for $T_t^\ast = e^{A^{\ast}t}$ we have $T_t^\ast \xi_j = e^{t \alpha_j}\xi_j, j \in \N$.
Define the diagonal $m \times m$-matrix $D_{t,m}$ by $(D_{t, m})_{i, j} = \delta_{i, j} e^{t \alpha_j}$. \\
Then, obviously, $T_t^\ast \Pi_m(v) = \Pi_m \big( D_{t, m}(v)\big)$ and hence, since $\nu \equiv (\Pi_m)_\ast (f_0 dv)$, 
\begin{align}\label{eq:5.24}
P_t^{(\varepsilon)} \varphi(x) = \int_{\R^m} exp \big( i \langle x, \Pi_m(v)\rangle \big) g_\varepsilon (v) dv,
\end{align}
where for $v \in \R^m$
\begin{align*}
g_{\varepsilon}(v) := \big( \prod_{j = 1}^{m} e^{-t \alpha_j} \big) \widehat{\mu_t}^{(\varepsilon)} \big( \Pi_m \big( D^{-1}_{t,m}\big)(v)\big) f_0 \big( D^{-1}_{t,m}(v)\big).
\end{align*}
We note that 
\begin{align}\label{eq:5.25}
\vert g_{\varepsilon} \vert \leq \prod_{j=1}^{m} e^{-t\alpha_j} \mathbbm{1}_{supp (f_0 \circ D^{-1}_{t, m})}\sup_{v \in \R^m} \vert f_0(v)\vert
\end{align}
and that $g_\varepsilon \in S(\R^m, \mathbb{C})$ with $g_\varepsilon (-v) = \overline{g_\varepsilon (v)}, v \in \R^m$.
Hence $P_t^{(\varepsilon)}\varphi \in D(L_0)$, and by \eqref{eq:5.34'} analogously to \eqref{eq:5.23} we obtain 
\begin{align*}
\big( L^{(\varepsilon)}- L_0\big) P_t^{(\varepsilon)} \varphi(x) &= \big( L^{(\varepsilon)}_0 - L_0\big) P_t^{(\varepsilon)} \varphi(x) \\
	&= \int_E \big( \lambda(\xi) - \lambda_\varepsilon(\xi)\big) e^{i \langle \xi, x \rangle} \nu_t^{(\varepsilon)}(d\xi)\\
	&= \int_{\R^m} \big( \lambda (\Pi_m (v))\big) - \lambda_\varepsilon (\Pi_m (v)) e^{i \langle \Pi_m(v), x\rangle} g_\varepsilon(v)dv.
\end{align*}
Hence by \cite[Lemma 3.1]{lescot}, \eqref{eq:5.25} and because $supp f_0$ is compact,
we obtain
\[\lim_{\varepsilon \to 0} \|(L^{(\varepsilon)} - L_0)P_t^{(\varepsilon)} \varphi\|_1 = 0,\]
and Claim 2 is proved. 
\end{claimproof}\\\\
\noindent So, we have that $(L_0, D(L_0))$ is a core operator of the infinitesimal generator $(L, D(L))$ of our generalized Mehler semigroup $(P_t)_{t \geq 0}$ defined in \eqref{eq:3.16}, if (H1), (H2) and \eqref{eq:5.31'} hold.\\

We are not entirely sure whether \eqref{eq:5.31'} always holds, but it does, if $\lambda$ is real-valued, according to the following:

\begin{lemma}\label{lem:5.12''}
Suppose $\lambda \colon E \to \R$. Then \eqref{eq:5.31'} holds.
\end{lemma}
\begin{proof}\label{proof_of_lemma_5.12''}
Let $\xi \in E$. Then by assumption, \eqref{eq:5.15'} and \eqref{eq:5.29'tilde} for all $\varepsilon \in [0, 1)$
\begin{align*}
\lambda_\varepsilon(\xi) = \frac{1}{2} \langle \xi, R\xi\rangle + \int_E (1 - cos \langle \xi, x\rangle) M_\varepsilon (dx)
\end{align*}
and $\lambda_\varepsilon(\xi)$ is decreasing in $\varepsilon$.
Then for all $\varepsilon \in [0, 1)$ (recalling that $P_t^{(0)}:= P_t, \mu_t^{(0)} := \mu_t$) we have
\begin{align}\label{eq:5.37'}
\widehat{\mu_{t}^{(\varepsilon)}}(\xi) \geq \widehat{\mu_{t}}(\xi)\; \forall \xi \in E.
\end{align}
Since $\widehat{\mu_{t}}$ is Sazonov continuous, for $\delta > 0$ there exists a nonnegative definite symmetric trace class operator $S_\delta$ on $E$ such that 
\begin{align*}
1 - \widehat{\mu_{t}}(\xi) \leq \langle S_\delta \xi, \xi\rangle + \delta \; \forall \xi \in E.
\end{align*}
(see \cite[Chap. VI, Theorem 2.3]{partha}). Hence by \eqref{eq:5.37'}
\begin{align*}
1 - \widehat{\mu_{t}^{(\varepsilon)}}(\xi) \leq \langle S_\delta \xi, \xi\rangle + \delta \; \forall \xi \in E.
\end{align*}
Hence the assertion follows again by [\cite[Chap. VI, Theorem 2.3]{partha}.
\end{proof}

\noindent In the above example the Kolmogorov operator (see \eqref{eq:5.17}) was a pseudo differential operator on $E$ with symbol $\lambda$ only dependent on $\xi$ (not on $x$), i.e., constant diffusion, and linear drift.
Therefore, finally we give an example on infinite dimensional state space $E$,
but where the Kolmogorov operator is a partial differential operator with non-constant second order (=diffusion) coefficients and nonlinear first order (=drift) coefficients.

\subsection{Applications to SDEs of locally monotone type on Hilbert spaces}\label{ex:L3}~\\
Consider the situation of Section \ref{Section 4.1} (so $E \coloneqq H\coloneqq$ a separable Hilbert space $H$ which is the pivot space of a Gelfand-triple $V \subset H \subset V^\ast$ as defined there).
Let the coefficients $A$ and $B$ be independent both of $\omega\in\Omega$ and $t\in[0,T]$ and that $U = H$.
Assume that $B$ satisfies \eqref{eq:supB} and we assume that $A$ can be written as a sum of two operators $C$ and $F$.
More precisely, let $(C,D(C))$ be a self-adjoint operator on $H$ such that $-C\geq\theta_0\in(0,\infty)$.
Define $V\coloneqq D((-C)^{\frac{1}{2}})$, equipped with the graph norm of $(-C)^{\frac{1}{2}}$,
and $V^\ast$ to be its dual.
Then it is easy to see that $C$ extends uniquely to a continuous linear operator from $V$ to $V^\ast$,
again denoted by $C$ such that for all $u,v\in V$
\begin{align}\label{eq:5.35}
{}_{V^{\ast}}\langle -Cu, v \rangle_V = \langle u,v \rangle_H.
\end{align}
Furthermore, let $F \colon H\to V^\ast$ be $\mc B(H)/\mc B(V^\ast)$-measurable such that $F$ restricted to $V$ satisfies (H1)-(H4) in Section \ref{Section 4.1}
with $B\equiv 0$, f constant and $A$ replaced by $F$ for $\alpha = 2,\ \beta\in[0,\infty)$, and $\theta = 0$.\\
Define
\begin{align}\label{eq:5.36}
A(u)\coloneqq Cu + F(u),\, u\in V.
\end{align}
Then it is easy to check that $A$ satisfies (H1)-(H4) in Section \ref{Section 4.1} with $\theta = \theta_0$, and $\alpha = 2,\ \beta\in[0,\infty)$ and $f$ constant.
So, by Theorem \ref{T1}, $(P_t)_{t\geq0}$ defined in \eqref{eq:3.9'} is a Markov $C_0$-semigroup on $\la C_\kappa(H),\tm\ra$ with $\kappa\coloneqq \la 1+\|\cdot\|_H^m\ra^{-1},\, m \in [1, \infty)$, due to Claim 2 in Section \ref{Section 4.1}.
Let $(L,D(L))$ be its infinitesimal generator.\\\\
\noindent Now fix an orthonormal basis $\{e_n\mid n\in\N\}$ of H consisting of elements in $D(C)$ and define
\begin{align}\label{eq:5.37}
D(L_0) \coloneqq \{ f\la\langle e_1,\cdot\rangle,\dots,\langle e_N,\cdot\rangle \mid N\in\N,\, f\in C_b^2(\R^N) \ra \}.
\end{align}
Define the Kolmogorov operator associated to SDE \eqref{SEE} with $B$ as above and $A$ as in \eqref{eq:5.36} with domain $D(L_0)$ as follows:
\begin{align}\label{eq:5.38}
L\varphi (x) := &\frac{1}{2} \sum_{i,j = 1}^\infty \langle B(x) e_i, B(x) e_j \rangle \frac{\partial}{\partial e_i} \la \frac{\partial\varphi}{\partial e_j} \ra (x) \\
				&+ \sum_{i=1}^\infty \la \langle x , C e_i\rangle + {}_{V^\ast}\langle F(x), e_i\rangle_V \ra \frac{\partial\varphi}{\partial e_i}(x),\quad x\in H,\, \varphi\in D(L_0).\nonumber
\end{align}
Here $\frac{\partial}{\partial e_i}$ denotes partial derivative in directions $e_i$ and we note that all sums in \eqref{eq:5.38} are in fact finite sums, since $\varphi\in D(L_0)$.\\
\noindent Now let us prove that
\begin{align}\label{eq:5.39}
L_0 \subset L.
\end{align}
For this we need one more condition, i.e. we assume:
\begin{equation}\label{eq:5.40}
\begin{aligned}
&\text{The eigenbasis of } (C,D(C))\text{ above can be chosen in such a way}\\
&\text{that } x\mapsto {}_{V^\ast}\langle F(x), e_i\rangle_V \text{ is continuous on } H \text{ and}\\
&\sup_{x\in H} \frac{|{}_{V^\ast}\langle F(x), e_i\rangle_V|}{1+|x|_H^m} < \infty\quad \text{for some }m\in [1,\infty) \text{\;and all\;} i \in \N.
\end{aligned}
\end{equation}
\begin{remark}\label{rem:5.13}
\item[(i)] A typical example for $F: H \to V^{\ast}$ above is a demicontinuous function (i.e., $x \mapsto {}_{V^\ast}\langle F(x), u\rangle_{V}$ is continuous on $H$ for all $u \in V$) with $F(0) \in H$, which is one sided Lipschitz and of at most polynomial growth.
\item[(ii)] A typical example for $(C, D(C))$ is the Laplace operator on an open bound domain $\mathcal{O} \subset \R^d$ with Dirichlet boundary conditions considered on $L^2(\mathcal{O})$.\\

Under condition \eqref{eq:5.40} a straightforward application of It\^o's formula for It\^o-processes in $\R^N, N \in \N$,
yields for all $\varphi \in D(L_0)$ and for the solution $X(t, x), t \geq 0, x \in E$, to \eqref{SEE} with $A$ and $B$ as above:
\begin{align}\label{eq:5.41}
P_t \varphi(x) &= \E\big[\varphi (X(t, x))\big]= \varphi(x) + \int_0^t \E \big[L_0 \;							\varphi(X(s, x))\big]ds \\
		&= \varphi(x) + \int_0^t \int_H L_0 \;						\varphi(y) \; \mu_s(x, dy)ds,
\end{align}
where 
\begin{align*}
\mu_s(x, dy) := (\P \circ X(s, x)^{-1}) (dy) \in M_\kappa (\R^d).
\end{align*}
Now, exactly the same arguments as in Section \ref{ex:5.9} prove that \eqref{eq:5.39} holds.\\

To prove that $(L_0, D(L_0))$ is a Markov core operator for $(L, D(L))$ on $(C_\kappa(H), \tm)$
we shall again use Proposition \ref{prop:5.9}, i.e. we have to prove uniqueness for the corresponding Fokker-Planck-Kolmogorov equation,
which is in general very difficult here, since our state space $H$ is infinite dimensional, and more assumptions are needed.
Though there are such results also when $B$ depends on $x$ (see \cite{bogachev15}), for simplicity we shall assume that $B$ is constant.
More preasely, we additionally assume that:
\begin{align}\label{eq:5.42}
B(x) = B \in L_2(H)
\end{align}
for all $x \in V$ with $B = B^\ast$, $B$ non-negative definite with $kerB = \{0\}$, and that for the eigenvalues $\alpha_k \in (0, \infty), k \in \N$, of $B$. There exists $m \in [1, \infty)$ such that
\begin{align}\label{eq:5.43}
\sup_{x \in H} (1 + \vert x\vert^m)^{-1} \sum_{k = 1}^{\infty} \alpha_k^{-1} \vert {}_{V^{\ast}}\langle F(x), e_k\rangle_V\vert^2 < \infty.
\end{align}
Of course, we may assume that both \eqref{eq:5.40} and \eqref{eq:5.43} hold with the same $m$ (otherwise we take the maximum of the two).
Then taking this $m$ and $\kappa := (1 + \|\cdot\|_H^m)^{-1}$,
it follows by \cite[Remark 2.1 (iii) and Theorem \ref{hyp-kappa}]{bogachev15} and by \eqref{eq:5.39} that all assumtions in Theorem \ref{prop:5.9} are fulfilled.
Hence $(L_0, D(L_0))$ is a Markov core operator for $(L, D(L))$ on $\la C_\kappa(H), \tm\ra$. 
\end{remark}

\section{Convex $C_0$-semigroups on $(C_\kappa(E), \tm)$}\label{Section6}
We now draw our attention to $C_0$-semigroups on $(C_\kappa(E), \tm)$ consisting of convex increasing operators on $\ck$. We show that these lead to viscosity solutions to abstract differential equations that are given in terms of their generator. 
We start by introducing our notion of a viscosity solution for abstract differential equations of the form
\begin{equation}\label{eq:PDE}
	u'(t)= L u(t), \quad \text{for all }t> 0.
\end{equation}
In the following, an operator $T\colon \ck\to\ck$ is called \textit{increasing} if
\[
T\phi_1\leq T\phi_2\quad \text{for all }\phi_1,\phi_2\in \ck\text{ with }\phi_1\leq \phi_2.
\]
We say that an operator $T\colon \ck\to\ck$ is \textit{convex} if
\[
T\big(\lambda\phi_1+(1-\lambda)\phi_2\big)\leq \lambda T\phi_1+(1-\lambda)T\phi_2
\]
for all $\lambda\in [0,1]$ and $\phi_1,\phi_2\in \ck$.\\

\bd{def.viscosity}\label{Def. 6.1}
Let $L\colon D\to C_\kappa(E)$ be a nonlinear operator, defined on a nonempty set $D\subset\ck$. We say that $u\colon [0,\infty)\to C_\ka(E)$ is a $D$-\textit{viscosity subsolution} to the abstract differential equation \eqref{eq:PDE} if $u$ is continuous w.r.t.\ the mixed topology $\tau_\kappa^{\mathscr M}$ and, for every $t>0$, $x\in E$, and every differentiable function $\psi\colon (0,\infty)\to C_\kappa(E)$ with $\psi(t)\in D$, $\big(\psi(t)\big)(x)=\big(u(t)\big)(x)$, and $\psi(s)\geq u(s)$ for all $s>0$,
\[
\big(\psi'(t)\big)(x)\leq \big(L \psi(t)\big)(x).
\]
Analogously, $u$ is called a $D$-\textit{viscosity supersolution} to \eqref{eq:PDE} if $u\colon [0,\infty)\to C_\ka(E)$ is continuous and, for every $t>0$, $x\in E$, and every differentiable function $\psi\colon (0,\infty)\to C_\kappa(E)$ with $\psi(t)\in D$, $\big(\psi(t)\big)(x)=\big(u(t)\big)(x)$, and $\psi(s)\leq u(s)$ for all $s>0$,
\[
\big(\psi'(t)\big)(x)\geq \big(L \psi(t)\big)(x).
\]
We say that $u$ is a $D$-\textit{viscosity solution} to \eqref{eq:PDE} if $u$ is a viscosity subsolution and a viscosity supersolution.
\ed

 Note that the previous definition does, a priori, not require the class of test functions for a viscosity solution to be rich in any sense. Therefore, in order to obtain uniqueness in standard settings, one has to verify on a case by case basis that the operator $L$ is defined on a sufficiently large set $D$ in order to apply standard comparison methods.\ Concerning the existence of $D$-viscosity solutions, we have the following theorem.

\bt{thm.viscosity}\label{Theorem6.2}
 Let $P$ be a $C_0$-semigroup on $(C_\kappa, \tm)$ consisting of convex increasing operators with infinitesimal generator $(L,D(L))$.\ Then, for every $\phi\in C_\kappa(E)$, the function $u\colon [0,\infty)\to C_\kappa(E), \;t\mapsto P_t \phi$ is a $D(L)$-viscosity solution to the abstract initial value problem
\begin{eqnarray*}
  u'(t)&=&L u(t), \quad \text{for all }t> 0,\\
  u(0)&=&\phi.
 \end{eqnarray*}
\et

\begin{proof}
 Fix $t>0$ and $x\in E$. We first show that $u$ is a viscosity subsolution. To that end, let $\psi\colon (0,\infty)\to C_\kappa(E)$ be a differentiable function with with $\psi(t)\in D(L)$, $\big(\psi(t)\big)(x)=\big(u(t)\big)(x)$ and $\psi(s)\geq u(s)$ for all $s>0$. Then, for $h\in (0,1)$ with $h<t$, the semigroup property implies that
 \begin{align*}
  0&=\frac{P_hP_{t-h}\phi-P_t\phi}{h}=\frac{P_hu(t-h)-u(t)}{h}\leq \frac{P_h\psi(t-h)-u(t)}{h}\\
  &\leq \frac{P_h\psi(t-h)-P_h\psi(t)}{h} +\frac{P_h\psi(t)-\psi(t)}{h} +\frac{\psi(t)-u(t)}{h}\\
  &\leq \bigg(P_h\Big(\psi(t)+\tfrac{\psi(t-h)-\psi(t)}{h}\Big)-P_h\psi(t)\bigg)+\frac{P_h\psi(t)-\psi(t)}{h} +\frac{\psi(t)-u(t)}{h},
 \end{align*}
 where, in the last inequality, we used the convexity of the map $v\mapsto P_h \big(\psi(t)+v\big)-P_h\psi(t)$.
 The strong continuity of the semigroup $P$ and $\psi(t)\in D(L)$ imply that
 \[
  P_h\Big(\psi(t)+\tfrac{\psi(t-h)-\psi(t)}{h}\Big)-P_h\psi(t)\to -\psi'(t)\quad \text{and}\quad \frac{P_h\psi(t)-\psi(t)}{h}\to L\psi(t)
 \]
 as $h\downarrow 0$ in the mixed topology $\tau_\kappa^{\mathscr M}$. Using the equality $\big(u(t)\big)(x)=\big(\psi(t)\big)(x)$, it follows that
 \[
  0\leq -\big(\psi'(t)\big)(x)+\big(L\psi(t)\big)(x).
 \]
 In order to show that $u$ is a viscosity supersolution, let $\psi\colon (0,\infty)\to C_\kappa(E)$ differentiable with $\psi(t)\in D(L)$, $\big(\psi(t)\big)(x)=\big(u(t)\big)(x)$ and $\psi(s)\leq u(s)$ for all $s>0$. Again, using the semigroup property, we find that, for all $h\in (0,1)$ with $h<t$,
 \begin{align*}
  0&=\frac{P_t\phi-P_hP_{t-h}\phi}{h}=\frac{u(t)-P_hu(t-h)}{h}\leq \frac{u(t)-P_h\psi(t-h)}{h}\\
  &= \frac{u(t)-\psi(t)}{h}+\frac{\psi(t)-P_h\psi(t)}{h}+\frac{P_h\psi(t)-P_h\psi(t-h)}{h}\\
  &\leq \frac{u(t)-\psi(t)}{h}+\frac{\psi(t)-P_h\psi(t)}{h}+\bigg(P_h\Big(\psi(t-h)+\tfrac{\psi(t)-\psi(t-h)}{h}\Big)- P_h\psi(t-h)\bigg),
 \end{align*}
   where, in the last step, we used the convexity of the map $v\mapsto P_h \big(\psi(t-h)+v\big)-P_h\psi(t-h)$. Again, the strong continuity of the semigroup $P$ and $\psi(t)\in D(L)$ imply that
 \[
  \frac{\psi(t)-P_h\psi(t)}{h}\to -L\psi(t)\quad \text{and}\quad P_h\Big(\psi(t-h)+\tfrac{\psi(t)-\psi(t-h)}{h}\Big)- P_h\psi(t-h)\to \psi'(t)
 \]
 as $h\downarrow 0$ in the mixed topology $\tau_\kappa^{\mathscr M}$. Since $\big(u(t)\big)(x)=\big(\psi(t)\big)(x)$, we find that $$0\leq -\big(L\psi(t)\big)(x)+\big(\psi'(t)\big)(x),$$ and the proof is complete.
\end{proof}

Finally, we derive a stochastic representation for $P$ using convex expectations. For a measurable space $(\Omega,\mathcal F)$, we denote the space of all bounded $\mathcal F$-measurable functions (random variables) $\Omega\to \R$ by $B_b(\Omega,\mathcal F)$. For two bounded random variables $X,Y\in B_b(\Omega,\mathcal F)$ we write $X\leq Y$ if $X(\omega)\leq Y(\omega)$ for all $\omega\in \Omega$. For a constant $m\in \R$, we do not distinguish between $m$ and the constant function taking that value.

\begin{definition}
	Let $(\Omega,\mathcal F)$ be a measurable space. A functional $\EE\colon B_b(\Omega,\mathcal F)\to\R$ is called a \textit{convex expectation} if, for all $X,Y\in B_b(\Omega,\mathcal F)$ and $\lambda\in [0,1]$,
	\begin{enumerate}
		\item[(i)] $\EE(X)\leq \EE(Y)$ if $X\leq Y$,
		\item[(ii)] $\EE(m)=m$ for all constants $m\in \R$,
		\item[(iii)] $\EE\big(\lambda X+(1-\lambda )Y\big)\leq \lambda \EE(X)+(1-\lambda)\EE(Y)$.
	\end{enumerate}
	We say that $(\Omega,\mathcal F,\EE)$ is a \textit{convex expectation space} if there exists a set of probability measures $\mathcal P$ on $(\Omega,\mathcal F)$ and a function $\alpha\colon \mathcal P \to[0,\infty)$ such that
	\[
	\EE(X)=\sup_{\P\in \mathcal P}\big( \E_\P(X)-\alpha (\P)\big)\quad \text{for all }X\in B_b(\Omega,\mathcal F),
	\]
	where $\E_\P(\cdot)$ denotes the expectation w.r.t. to the probability measure $\P$.
\end{definition}

The following theorem is a consequence of \cite[Theorem 5.6]{denk2018kolmogorov} and the fact that the $\tm$-continuity of $P_t$ on $\tm$-bounded subsets implies the so-called continuity from above or Daniell continuity of $P_t$, for $t\geq 0$.

\begin{theorem}\label{stochrep}
	Assume that $E$ is a Polish space, $\kappa\equiv 1$, and $P$ is a $C_0$-semigroup of increasing convex operators with $P_t m=m$ for all $t\geq 0$ and $m\in \R$. Then, there exists a quadruple $(\Omega,\mathcal F,(\EE^x)_{x\in E},(X(t))_{t\geq 0})$ such that
	\begin{enumerate}
		\item[(i)] $X(t)\colon \Omega\to E$ is $\mathcal F$-$\mathcal B$-measurable for all $t\geq 0$,
		\item[(ii)] $(\Omega,\mathcal F,\EE^x)$ is a convex expectation space with $\EE^x(\vp(X(0)))=\vp(x)$ for all $x\in E$ and $\vp\in C_b(E)$,
		\item[(iii)] For all $0\leq s<t$, $n\in \N$, $0\leq t_1<\ldots <t_n\leq s$ and $\psi\in C_b(E^{n+1})$,
		\begin{equation}\label{eq.stochrep1}
			\EE^x\big(\psi(X(t_1),\ldots,X(t_n),X(t))\big)=\EE^x\left(\big(P_{t-s}\psi(X(t_1),\ldots,X(t_n),\, \cdot\,)\big)(X(s))\right).
		\end{equation}
	\end{enumerate}
	In particular,
	\begin{equation}\label{eq.stochrep2}
		\big(P_t\vp\big)(x)=\EE^x(\vp(X(t))).
	\end{equation}
	for all $t\geq 0$, $x\in E$, and $\vp\in C_b(E)$.
\end{theorem}

Let $E$ be a Polish space. The quadruple $(\Omega,\mathcal F,(\EE^x)_{x\in E},(X(t))_{t\geq 0})$ can be seen as a nonlinear version of a Markov process. As an illustration, we consider the case, where the semigroup $P$ and thus $\EE^x$ is linear for all $x\in E$, and choose $\psi(x,y)=\vp(x)1_B(y)$, for $x,y\in E$, with $\vp\in C_b(E)$ and $B\in \mathcal B^n$, where $\mathcal B^n$ denotes the product $\sigma$-algebra of the Borel $\sigma$-algebra $\mathcal B$. Then, $\EE^x=\E_{\P^x}$ is the expectation w.r.t.\ a probability measure $\P^x$ on $(\Omega,\mathcal F)$ for all $x\in E$. Using the continuity from above and Dynkin's lemma, Equation \eqref{eq.stochrep1} reads as
\[
\E_{\P^x}\big(\vp(X(t))1_B(X(t_1),\ldots, X(t_n))\big)=\E_{\P^x}\big[\big(P_{t-s}\vp\big)(X(s))1_B(X(t_1),\ldots, X(t_n))\big],
\]
which is equivalent to the Markov property
\begin{equation}\label{eq.linmarkprop}
	\E_{\P^x}\big(\vp(X(t))|\mathcal F_s\big)=\big(P_{t-s}u\big)\big(X(s)\big) \quad \P^x\text{-a.s.},
\end{equation}
where $\mathcal F_s:=\sigma\big(\{X(u)\, |\, 0\leq u\leq s\}\big)$. On the other hand, if $\EE^x=\E_{\P^x}$, the Markov property \eqref{eq.linmarkprop} implies Property (iii) from Theorem \ref{stochrep}.

\section{Examples: value functions of optimal control problems}\label{sec.control}

\subsection{A finite-dimensional setting}

In this section, we show that value functions of a large class of optimal control problems are examples of nonlinear $C_0$-semigroups. We illustrate this by means of a simple controlled dynamics in $\R^d$, with $d\in \N$ where the contol acts on the drift of a diffusion process. However, with similar techniques also other classes of controlled diffusions fall into our setup. 
Throughout, let $W=(W(t))_{t\geq 0}$ be a Brownian Motion on a complete filtered probability space $\big(\Omega, \mathcal{F}, (\mathcal F_t)_{t\geq 0},  \mathbb{P}\big)$ satisfying the usual assumptions and $\sigma\in \R^{d\times d}$ is symmetric and positive definite. For $m\in \N$, we consider a fixed nonempty set $A\subset \mathbb{R}^m$ of controls with $0\in A$ and define the set of admissible controls $\mathcal{A}$ as the set of all progressively measurable processes $\alpha\colon \Omega \times [0,T] \to A$ with
\[
\E\bigg(\int_0^t|\alpha(s)|ds\bigg)<\infty.
\]
For a fixed measurable function $b\colon \Omega\times \R^d\times A\to \R^d$, an admissible control $\alpha\in \mathcal A$, and an initial value $x\in \R^d$, we consider the controlled dynamics
\begin{equation}
	\label{dynamics} 
	dX^\alpha(t,x)= b\big(X^{\alpha}(t,x), \alpha(t)\big)dt +\sigma dW(t),\quad \text{for } t\geq 0, \quad X^{\alpha}(0,x)=x.
\end{equation}
We assume that the drift term $b$ satisfies the following Lipschitz and growth conditions: there exists a constant $C\geq 0$ such that
\begin{align*}\label{assumption.on.b}
	& b(x,0)=0,\quad  \mathbb P\text{-a.s.}, \quad \text{for all }x\in \R^d,\\
	& |b(x_1,a)-b(x_2,a)| \leq C|x-y|,\quad \mathbb P\text{-a.s.}, \quad \text{for all } x_1,x_2 \in \R^d \text{ and }a\in A, \\
	& |b(x,a)| \leq C\big(1+|x|+|a|\big),\quad \mathbb P\text{-a.s.}, \quad \text{for all }x\in \R^d\text{ and }a\in A. 
\end{align*}
Under these assumptions, by standard SDE theory, for each initial value $x\in \R^d$ and every admissible control $\alpha \in \mathcal{A}$, there exists a unique strong solution $(X^{\alpha}(t,x))_{t\geq 0}$ to the controlled SDE \eqref{dynamics}. 

We consider the weight function $\kappa\equiv 1$ and a running cost function $g\colon  A\to [0,\infty)$ with $g(0)=0$ and 
\[
\overline g^*(y):=\sup_{a\in A} \big(|a|y-g(a)\big)<\infty
\]
for all $y\geq 0$. For $\vp\in C_b(\R^d)$, we consider the value function
\[
V(t,x;\vp):=\sup_{\alpha\in \mathcal A}\mathbb E\bigg(\vp\big(X^{\alpha}(t,x)\big)-\int_0^tg\big(\alpha(s)\big)ds\bigg),
\]
and we define $\big(P_t\vp\big)(x):=V(t,x;\vp)$ for all $t\geq 0$ and $x\in \R^d$. We first show that $P_t\colon C_b(\R^d)\to C_b(\R^d)$ is well-defined with $\|P_t\vp\|_\infty\leq \|\vp\|_\infty$ for all $\vp\in C_b(\R^d)$. Using the Lipschitz condition of $b$ together with Gronwall's lemma, we obtain the a priori estimate
\[
\E\big(\big|X^{\alpha}(t,x_1)-X^{\alpha}(t,x_2)\big|\big)\leq e^{Ct} |x_1-x_2|
\]
for all $t\geq 0$, $x_1,x_2\in \R^d$, and $\alpha \in \mathcal A$. This shows that the value function $V$ is continuous in the $x$-variable. Moreover, $\|V(t,\,\cdot\, ,\vp)\|_\infty\leq \|\vp\|_\infty$ for all $\vp\in C_b(\R^d)$,since $g(0)=0$. Since the value function $V$ satisfies the dynamic programming principle, cf.\ Pham \cite{pham} or Fabbri et al.\ \cite{Fabbri17}, the family $P=(P_t)_{t\geq 0}$ is a semigroup.

Using the linear growth of $b$ together with Gronwall's lemma,
\begin{equation}\label{eq.integrability-controlled-dynamics}
	\E\big(\big|X^{\alpha}(t,x)-x\big|\big)+|x|\leq e^{Ct} \bigg( |x|+\|\sigma\|\sqrt t+ Ct +\int_0^t C|\alpha(s)|ds\bigg)
\end{equation}
for all $t\geq 0$, $x\in \R^d$, and $\alpha\in \mathcal A$.  Let $\ep>0$, $\vp\in C_b(\R^d)$,  and $t\geq 0$. Then, for every $r\geq 0$, there exists some $\delta>0$ such that
\[
|\vp(y)-\vp(x)|<\frac\ep 2\quad \text{for all }x,y\in \R^d\text{ with }|x|\leq r\text{ and }|x-y|<\delta.
\]
Hence, for all $x\in \R^d$ with $|x|\leq r$, Equation \ref{eq.integrability-controlled-dynamics} implies that
\begin{align*}
	V(t,x;\vp)-\vp(x)&\leq \frac{\ep}{2}+2\|\vp\|_\infty\E\Big(1_{\{|X^{\alpha}(t,x)-x|>\delta\}}\Big)-\E\bigg(\int_0^tg\big(\alpha(s)\big)ds\bigg)\\
	&	\leq  \frac{\ep}{2}+\frac{2\|\vp\|_\infty}{\delta}\E\big(|X^{\alpha}(t,x)-x|\big)-\E\bigg(\int_0^tg\big(\alpha(s)\big)ds\bigg)	\\
	&\leq \frac{\ep}{2}+(e^{Ct}-1)|x|+e^{Ct}(Ct+\|\sigma\|\sqrt t)+t\overline g^*\Big(\tfrac{2\|\vp\|_\infty}{\delta} Ce^{Ct}\Big).
\end{align*}
On the other hand, for all $x\in \R^d$ with $|x|\leq r$,
\[
\vp(x)-V(t,x;\vp)\leq \frac\ep 2+2\|\vp\|_\infty\E(1_{\{|\sigma W(t)|>\delta\}})\leq \frac\ep 2+\frac{2\|\vp\|_\infty}\delta \sigma \sqrt{t}.
\]
We thus see, that $P_t\vp\to \vp$ uniformly on compact sets.

Now, let $R\geq 0$, $\ep>0$, and $\vp_1,\vp_2\in C_b(\R^d)$ with $\|\vp_i\|_\infty\leq R$, for $i=1,2$, and
\[
\sup_{|y|\leq r}|\vp_1(y)-\vp_2(y)|< \frac{\ep}{3}\quad \text{for sufficiently large }r>0.
\]
We observe that, for $\vp\in C_b(\R^d)$ with $\|\vp\|_\infty\leq R$, $t\geq 0$, $x\in \R^d$, and $\alpha\in \mathcal A$ with
\begin{equation}\label{eq.apriori-controlled-dynamics1}
	V(t,x;\vp)\leq \frac{\ep}3+\E\bigg(\vp\big(X^\alpha(t,x)\big)-\int_0^tg\big(\alpha(s)\big)ds\bigg),
\end{equation}
it follows that
\begin{equation}\label{eq.apriori-controlled-dynamics}
	\E\bigg(\int_0^t|\alpha(s)|ds\bigg)\leq  t\overline g^*(1)+\E\bigg(\int_0^tg\big(\alpha(s)\big)\bigg)\leq \frac{\ep}3+t\overline g^*(1)+2R
\end{equation}
Let $T,c\geq 0$. Then, for $t\in[0,T]$, $x\in \R^d$ with $|x|\leq c$, and $\alpha\in \mathcal A$ satisfying  Equation \eqref{eq.apriori-controlled-dynamics1} for $\vp=\vp_1$.
\begin{align*}
	V(t,x;\vp_1)-V(t,x;\vp_2)&\leq  \frac{2\ep}{3}+2R\E\big(1_{\{|X^{\alpha}(t,x)|>r\}}\big)
	\leq  \frac{2\ep}{3}+\frac{2R}{r}\E\big(\big|X^{\alpha}(t,x)\big|\big)\\
	&\leq \frac{2\ep}{3}+\frac{2R}{r} e^{CT} \Big(\frac{\ep}{3}+2R+c+\|\sigma\|\sqrt T+ T\big(C+\overline g^*(1)\big)\Big),
\end{align*}
where, in the last step, we used Equation \eqref{eq.integrability-controlled-dynamics} and Equation \eqref{eq.apriori-controlled-dynamics}.
Choosing $r>0$ sufficiently large, a symmetry argument yields that 
\[
\sup_{|x|\leq c}\big|V(t,x;\vp_1)-V(t,x;\vp_2)\big|<\ep\quad \text{ for all }t\in [0,T].
\]
With similar arguments together with It\^o's formula, one finds that the generator $L$ of $P$ on $C_b^2(E)$ is given by
\[
(L\vp)(x)=\frac{1}{2}{\rm tr} \big(\sigma^2 \nabla^2 \vp(x)\big)+ \sup_{a\in A} \Big(b(x,a) \nabla\vp(x)-g(a)\Big).
\]
By Theorem \ref{thm.viscosity}, we thus obtain that $(t,x)\mapsto V(t,x;\vp)=\big(P_t\vp\big)(x)$ is a $C_b^2(\R^d)$-viscosity solution to the HJB equation
\[
\partial_t v(t,x)=\frac{1}{2}{\rm tr}\big(\sigma^2\nabla_{xx}^2 v(t,x)\big)+ \sup_{a\in A} \Big(b(x,a) \nabla_xv(t,x)-g(a)\Big),\quad v(0,x)=\vp(x).
\]

\subsection{An infinite-dimensional example with linear growth}\label{Section7.2}

In this section, we consider a similar setup as in the previous subsection in a separable Hilbert space $H$ with orthonormal base $(e_k)_{k\in \N}\subset H$, endowed with the $bw$-topology. Throughout, let $W=(W(t))_{t\geq 0}$ be a Brownian Motion with trace class covariance operator $\Sigma\colon H\to H$ on a complete filtered probability space $\big(\Omega, \mathcal{F}, (\mathcal F_t)_{t\in [0,T]},  \mathbb{P}\big)$ satisfying the usual assumptions. For $p\in(1,2]$, we define the the set of admissible controls $\mathcal{A}$ as the set of all progressively measurable processes $\alpha\colon \Omega \times [0,T] \to U$ with
\[
\E\bigg(\int_0^t|\alpha_s|_H^pds\bigg)<\infty.
\]
For every admissible control $\alpha\in \mathcal A$ and every initial value $x\in H$, we consider the controlled dynamics
\begin{equation}
	\label{dynamics2} 
	X^{\alpha}(t,x)= x+\int_0^t\alpha(s)ds +W(t)\quad \text{for all } t\geq 0.
\end{equation}
We consider the weight function $\kappa(x):=(1+|x|_H)^{-1}$, for $x\in H$, and a running cost function $g\colon  H\to [0,\infty)$ with $g(0)=0$ and 
\[
\overline g_p^*(y):=\sup_{a\in H} \big(|a|_H^py-g(a)\big)<\infty
\]
for all $y\geq 0$. This implies that, for all $q\in [1,p]$ and $y\in \geq0$,
\[
\overline g_q^*(y):=\sup_{a\in H} \big(|a|_H^qy-g(a)\big)<\infty.
\]
For $\vp\in C_\kappa(H_{bw})$, we consider the value function
\[
V(t,x;\vp):=\sup_{\alpha\in \mathcal A}\mathbb E\bigg(\vp\big(X^{\alpha}(t,x)\big)-\int_0^tg\big(\alpha(s)\big)ds\bigg),
\]
and we define $\big(P_t\vp\big)(x):=V(t,x;\vp)$ for all $t\geq 0$ and $x\in H$. We first show that $P_t\colon  C_\kappa(H_{bw})\to  C_\kappa(H_{bw})$ is well-defined. Let $\vp\in C_b(H_{bw})$ such that,there exists a constant $L\geq 0$ and some $n\in \N$ with 
\be{eq.finitebased}
|\vp(x)-\vp(y)|\leq L \sum_{i=1}^n|\langle x-y,e_i\rangle|\quad\text{for all }x,y\in H.
\ee
Then, for all $t\geq 0$ and $x,y\in H$
\[
|V(t,x;\vp)-V(t,y;\vp)|\leq L \sum_{i=1}^n|\langle x-y,e_i\rangle|.
\]
Moreover, for all $q\in [1,p]$,
\[
\E\big(\big|X_t^{x,\alpha}\big|_H^p\big)^{1/q}\leq  |x|_H+\sqrt{\|\Sigma\|_{\rm tr} t}  +\E\bigg(\int_0^t|\alpha(s)|_H^qds\bigg)^{1/q}
\]
for all $t\geq 0$, $x\in H$, and $\alpha\in \mathcal A$. Using this estimate, we find that, for $t\geq 0$, $x\in H$, and $\vp\in C_\kappa(H_{bw})$,
\begin{align*}
	V(t,x;\vp)&\leq \|\vp\|_\kappa\Big(1+\E\big(\big|X^{\alpha}(t,x)\big|_H\big)\Big)-\E\bigg(\int_0^tg\big(\alpha(s)\big)\, {\rm d}s\bigg)\\
	&\leq \|\vp\|_\kappa  \big(1+|x|_H+\sqrt {\|\Sigma\|_{\rm tr} t}\big)+t\overline g_1^*\big(\|\vp \|_\kappa\big).
\end{align*}
Moreover, for all $t\geq 0$, $x\in H$, and $\vp\in C_\kappa(H_{bw})$,
\[
V(t,x;\vp)\geq - \|\vp\|_\kappa \big(1+|x|_H+\sqrt {\|\Sigma\|_{\rm tr} t}\big),
\]
which shows that
\[
\|P_t\vp\|_\kappa\leq \Big(\|\vp\|_\kappa +\overline g^*\big(\|\vp \|_\kappa\big)\Big)\big(1+t+\sqrt {\|\Sigma\|_{\rm tr}t}\big).
\]
Now, let $R\geq 0$, $\ep>0$, and $\vp_1,\vp_2\in C_\kappa(H_{bw})$ with $\|\vp_i\|_\kappa\leq R$, for $i=1,2$, and
\[
\sup_{|y|_H\leq r}|\vp_1(y)-\vp_2(y)|< \frac{\ep}{3}\quad \text{for sufficiently large }r\geq 1.
\]
We observe that, for $\vp\in C_\kappa(H_{bw})$ with $\|\vp\|_\kappa\leq R$, $t\geq 0$, $x\in \R^d$, and $\alpha\in \mathcal A$ with
\begin{equation}\label{eq.apriori-controlled-dynamics12}
	V(t,x;\vp)\leq \frac{\ep}3+\E\bigg(\vp\big(X^\alpha(t,x)\big)-\int_0^tg\big(\alpha(s)\big)ds\bigg),
\end{equation}
it follows that
\[
\E\bigg(\int_0^tg\big(\alpha(s)\big)\bigg)\leq \frac{\ep}3+4R\Bigg(\big(1+|x|_H+2\sqrt {\|\Sigma\|_{\rm tr} t}\big)^p+\E\bigg(\int_0^t|\alpha(s)|^pds\bigg)\Bigg),
\]
which implies that
\begin{equation}\label{eq.apriori-controlled-dynamics2}
	\E\bigg(\int_0^t|\alpha(s)|^pds\bigg)\leq \frac{\ep}3+4R\big(1+|x|_H+\sqrt {\|\Sigma\|_{\rm tr} t}\big)^p+t\overline g_p^*(1+4R)
\end{equation}
Let $T,c\geq 0$. Then, for $t\in[0,T]$, $x\in H$ with $|x|_H\leq c$, and $\alpha\in \mathcal A$ satisfying  Equation \eqref{eq.apriori-controlled-dynamics12} for $\vp=\vp_1$.
\begin{align*}
	V(t,x;\vp_1)-V(t,x;\vp_2)&\leq  \frac{2\ep}{3}+2R\E\Big(\big(1+|X^\alpha(t,x)|\big)1_{\{|X^{\alpha}(t,x)|>r\}}\Big)\\
	& \leq  \frac{2\ep}{3}+\frac{4R}{r^{p-1}}\E\big(\big|X^{\alpha}(t,x)\big|^p\big)\\
	&\leq \frac{2\ep}{3}+\frac{16(1+R)}{r^{p-1}} \Big(\frac{\ep}{3}+\big(1+c+\sqrt{\|\Sigma\|_{\rm tr} T}\big)^p+T\overline g^*(1+4R)\Big),
\end{align*}
where, in the last step, we used Equation \eqref{eq.apriori-controlled-dynamics2}.
Choosing $r>0$ sufficiently large, a symmetry argument yields that 
\[
\sup_{|x|_H\leq c}\big|V(t,x;\vp_1)-V(t,x;\vp_2)\big|<\ep\quad \text{ for all }t\in [0,T].
\]
Since the value function $V$ satisfies the dynamic programming principle, cf.\ Fabbri et al.\ \cite{Fabbri17}, the family $P=(P_t)_{t\geq 0}$ is a semigroup.

Let $\ep>0$, $\vp\in C_b(H_{bw})$ with \eqref{eq.finitebased} with $L\geq 0$ and $n\in \N$, and $t\geq 0$. Then, for all $x\in H$,
\begin{align*}
	V(t,x;\vp)-\vp(x)&\leq L \sqrt{\|\Sigma\|_{\rm tr}t}+\E\bigg( \int_0^t L|\alpha(s)|-g\big(\alpha(s)ds\big)\bigg)\\
	&\leq L \sqrt{\|\Sigma\|_{\rm tr}t}+t\overline g^*(L) \to 0\quad \text{as }t\to 0.
\end{align*}
Moreover, for all $x\in H$,
\[
\vp(x)-V(t,x;\vp)\leq L\sqrt{ \|\Sigma\|_{\rm tr}t}.
\]
In particular, since the bounded finitely based Lipschitz functions are dense in $C_\kappa(H_{bw})$, it follows that $P_t\vp\to \vp$ uniformly on compacts for all $\vp\in C_\kappa(H_{bw})$.

By It\^o's formula and Theorem \ref{thm.viscosity}, we obtain that $(t,x)\mapsto V(t,x;\vp)=\big(P_t\vp\big)(x)$ is a $C_b^2(H_{bw})$-viscosity solution to the HJB equation
\[
\partial_t v(t,x)=\frac{1}{2}{\rm tr}\big(\Sigma \nabla_{xx}^2 v(t,x)\big)+ \sup_{a\in A} \Big(\big\langle a, \nabla_xv(t,x)\big\rangle -g(a)\Big),\quad v(0,x)=\vp(x).
\]

\appendix
\section{Proof of part (c) of Remark \ref{rem2.2}}\label{app.A0}\noindent
Note first, that $E$ is a Hausdorff topological vector space and thus completely regular, see \cite[Theorem 2.9.2]{jarchow}.
We say that a set $B\subset E$ is  $bw$-closed if its intersection with every weak$^\star$-compact set is weak$^\star$-closed.
The corresponding topology is completely regular, and is  known as the $bw^{\ast}$-topology, \cite[pages 427-428]{DS} or \cite[Section 2.7]{Megginson}. Clearly, the $bw^{\ast}$-topology coincides with the weak$^\star$-topology on every weak$^\star$-compact set, and therefore weak$^\star$-compactness is equivalent to $bw^{\ast}$-compactness. As a consequence, any function $\vp\colon E\to\R$ that is continuous on all weak$^\star$-compacts of $E$ endowed with the weak$^\star$ topology is continuous on $E$ endowed with the $bw^{\ast}$-topology.  In fact, the $bw^{\ast}$-continuous functions are precisely the sequentially  weak$^\star$-continuous functions.\ Thus, part (3) of Hypothesis \ref{hyp_space} holds. Let us recall that a dual of an infinite-dimensional Banach space endowed with its weak$^\star$ topology is never a $k_f$-space, see \cite[Theorem 5.1]{Noble} and \cite[Corollary 1.14]{Gabriyelyan}.\\
To prove part (2) of Hypothesis \ref{hyp_space}, we note first that $\la E,bw^{\ast}\ra$ being a countable union of compact Polish spaces, is a Souslin space,
see Theorem 6.6.6 in vol. 2 of \cite{bogachev}. Then by Theorem 6.7.7 in vol. 2 of \cite{bogachev}  $\la E,bw^{\ast}\ra$ is a perfectly normal topological space,
hence $\mc B\la E,bw^{\ast}\ra=\mathrm{Ba}\la E,bw^{\ast}\ra$. 
	\\
	Since balls in $E$ equipped with the weak$^\star$-topology are metrizable, and weak$^\star$-compacts are norm-bounded, (1) of Hypothesis \ref{hyp_space} also holds.
\section{Some facts on the mixed topology}\label{app.A}
In this section we collect some general properties of the mixed topology that was introduced in Section \ref{section2} in a special case suitable for our purposes.

 We follow  \cite{wiweger}, and introduce the mixed topology $\tau^{\mc M}$ for a linear space $X$, endowed with two topologies $\tau_1$ and $\tau_2$. We assume that $\la X,\tau^1\ra $ and $\la X,\tau^2\ra$ are Hausdorff topological vector spaces with $\tau^1\subset \tau^2$ and corresponding bases $\mc U(\tau_1)$ and $\mc U(\tau_2)$ of neighbourhoods of zero. 
 For a sequence  $\gamma=\la U_n^1\ra\subset \mc U(\tau^1)$, and any $U^2\in\mc U(\tau^2)$, we define a set 
 \[U\la\gamma,U^2\ra=\bigcup_{n=1}^\infty\sum_{k=1}^n\la U_k^1\cap kU^2\ra.\]
 Then, the family 
 \[
 \Big\{U\la\gamma,U^2\ra\colon \gamma=\la U_n^1\ra\subset\mc U\la\tau^1\ra,\,U^2\in\mc U\la\tau^2\ra\Big\}
 \]
 forms a basis of neighborhoods of zero for a topology $\tau^{\mc M}=\tau^{\mc M}\la\tau^1,\tau^2\ra$. Then $(X,\tau^{\mc M})$ is a Hausdorff topological vector space, and the topology $\tau^{\mc M}$ is known as the mixed topology. In the present paper, we use this definition only in the case, where $X=C_\kappa(E)$ with a completely regular topological Hausdorff space $E$, $\tau_1=\tau_\kappa^{\mc C}$, and $\tau_2=\tau_\kappa^{\mc U}$, cf.\ Section \ref{section2} for the notations.
 We list some basic properties of the mixed topology in this case. 

Recall that a subset of a locally convex space is bounded if it is absorbed by every neighbourhood of zero.

\begin{proposition}\label{t12}\
	\begin{enumerate}
		\item[(a)] (\cite[Section 2.2]{wiweger}) The topology $\tau_\kappa^{\mc M}$ is the strongest locally convex topology on $C_\kappa (E)$ that coincides with $\tau_\kappa^{\mc C}$ on bounded sets of $\tau_\kappa^{\mc U}$.
		\item[(b)] (\cite[Corollary on p.\ 56]{wiweger}) A set $B\subset C_\kappa(E)$ is bounded in the 
		topology $\tm$ if and only if it is bounded in the 
		topology $\tu$.
		\item[(c)] (\cite[Corollary 2.2.5]{wiweger}) The topology $\tau_\kappa^{\mc M}$ can be defined as the weakest topology $\tau$ on $C_\kappa(E)$ such that for every locally convex space $F$ and every linear operator $T\colon C_\kappa(E)\to F$, $T$ is $\tau$-continuous if and only if $T$ is $\tau_\kappa^{\mc C}$-continuous on $\tu$-bounded sets.
	\end{enumerate} 
\end{proposition}

\begin{proposition}\label{t14}(Theorem 2.3.1 in \cite{wiweger})
	A sequence $\left(\phi_n\right)\subset C_\kappa(E)$ is $\tm$-convergent to 
	$\phi\in C_\kappa(E)$ if and only if
	\[\sup_{n\ge 1}\left\|\phi_n\right\|_\kappa<\infty\quad \text{and}\quad \lim_{n\to\infty}\phi_n=\phi\quad\mbox{\rm in the topology }\tau_\kappa^{\mc C}.\]
\end{proposition}
The next theorem is well known for the case $\kappa\equiv 1$, see \cite{wiweger}. It seems to be new for $\kappa$ satisfying Hypothesis \ref{hyp-kappa}. From now on we assume that Hypothesis \ref{hyp-kappa} holds.
 \bt{lem_ck_mixed}
The family of seminorms 
\[\{p_{\kappa,(a_n),(C_n)};\,0<a_n\to 0,\,\, C_n\,\,\mathrm{compact}\}\] 
generates the mixed topology $\tm$ on $\ck$. 
 \et
\begin{proof}
 The lemma was proved for $\kappa=1$ on pages 65-66 of \cite{wiweger}.
 For general $\kappa$ we need to check that the assumptions of Theorem 3.1.1 in \cite{wiweger} are satisfied.
 We first note that 
 \[\|\vp\|_\kappa=\sup_Cp_{\kappa,C}(\vp)\,,\]
 where the supremum is taken over all compact subsets of $E$. It remains to show that condition (r) from \cite{wiweger} holds as well.
 Let $\ve>0$ and let $\vp\in\ck$. For any $n\ge 1$ consider compact sets $C_i\subset E$, $i=1,\ldots, n$, and define sets $C=\bigcup_{k=1}^nC_k$ and 
 \[U=\{x\in E;\,|\kappa(x)\vp(x)|<p_{\kappa,C}(\vp)+\ve\}\,.\]
 Then $C\subset U$, $U$ is open and 
\be{eq_xxx}
\sup_{x\in U}|\kappa(x)\vp(x)|\le p_{\kappa,C}(\vp)+\ve=\max\la p_{\kappa,C_1}(\vp),\ldots,p_{\kappa,C_n}(\vp)\ra+\ve.
\ee
Let $f:E\to[0,1]$ be a continuous function, such that $f(x)=0$ for $x\in C$ and $f(x)=1$ for $x\in E\setminus U$. 
For any $\vp\in\ck$ we have 
\[\vp=(1-f)\vp+f\vp=\psi+\eta\,.\]
Then we have $\eta(x)=0$ for $x\in C$ or, equivalently
\[p_{\kappa,C_i}( \eta)=0,\quad i=1,\ldots, n\,.\]
Next, using \eqref{eq_xxx} we obtain 
\[\|\psi\|_\kappa=\sup_{x\in E}|\kappa(x)\psi(x)|=\sup_{x\in U}|\kappa(x)\psi(x)|\le \max\la p_{\kappa,C_1}(\vp),\ldots,p_{\kappa,C_n}(\vp)\ra+\ve,\]
and condition (r) holds. Now, the theorem follows immediately from Theorem 3.1.1 in \cite{wiweger}. 
\end{proof}
\begin{corollary}\label{cor_nowa}
The mixed topology $\tm$ is also determined by the family of seminorms 
$\{p_{\kappa,(a_n),(C_n)};\,0<a_n\to 0,\,\, C_n\,\,\mathrm{compact}\}$ with $a_n\downarrow 0$ and $C_n\subset C_{n+1}$. 
\end{corollary}
\begin{proof}
We have for all $\vp\in C_\kappa(E)$ and $\tilde C_n := \bigcup_{l=1}^n C_l$
\[\sup_{n\in\N}a_n p_{\kappa,C_n}(\vp)\leq\sup_{n\in\N}\Big(\sup_{k\geq n} a_k \Big)p_{\kappa,\tilde C_n}(\vp)\]
\end{proof}
Let $\mc W$ denote the class of all bounded upper semicontinuous functions $w:E\to[0,\infty)$ such that for every $\ve>0$ the set $\{x\in E;\, w(x)\ge\ve\}$ is compact.
For every $w\in\mc W$ define 
\[p_{w}(\vp)=\sup_{x\in E}|w(x)\kappa(x)\vp(x)|,\quad \vp\in\ck\,.\]
Let $\tau_{\kappa}^{\mc W}$ denote the locally convex topology on $\ck$ defined by the family of seminorms $\{p_w;\,w\in\mc W\}$. We aim to prove the following 
\bt{th_w}
We have 
\[\tau_\kappa^{\mc W}=\tau^{\mc M}_\kappa\,.\]
\et
\begin{proof}
Since $I_C\in\mc W$ for every compact $C\subset E$, the topology $\tau_{\kappa}^{\mc W}$ is stronger than $\tau_{\kappa}^{\mc C}$.
For $B_\kappa=\{\vp\in\ck;\, \|\vp\|_\kappa\le 1\}$, $\ve>0$ and $w\in\mc W$ define  $U=\{\vp\in B_\kappa;\,\|w\kappa\vp\|_\infty<\ve\}$.
Let $K=\{x\in E;\, w(x)\ge\frac{\ve}{2}\}$ and 
\[V=\left\{\vp\in B_\kappa;\, p_{\kappa,K}(\vp) <\frac{\ve}{2(1+\|w\|_\infty)}\right\}\in\tau_\kappa^{\mc C} \cap B_\kappa\,.\]
If $\vp\in V$, then 
\[\|\kappa w\vp\|_\infty=\max\la p_K(\kappa w\vp),p_{E\setminus K}(\kappa w\vp)\ra\le\max\la \|w\|_\infty\frac{\ve}{2(1+\|w\|_\infty)},\frac{\ve}{ 2}\|\vp\|_{\kappa}\ra=\frac{\ve}{ 2} <\ve.
\]
Therefore, $V\subset U$ and on $B_\kappa$ we have $\tau_\kappa^{\mc W}\subset\tau_\kappa^{\mc C}$, hence $\tau_\kappa^{\mc W}=\tau_\kappa^{\mc C}$ on $B_\kappa$.
By Proposition \ref{t12} (a) we hence have $\tau_\kappa^{\mc W}\subset\tm$. \\
Let $U=\{\vp\in\ck;\,p_{\kappa,(a_n),(C_n)}(\vp)< 1\}$. By Corollary \ref{cor_nowa} we can assume that $C_n\subset C_{n+1}$ and $a_n\downarrow 0$.
Define $w$ as
\[w(x)=\left\{\begin{array}{lll}
a_1&x\in C_1,&\\
a_n&x\in C_n\setminus C_{n-1},&n\ge 2,\\
0&x\in E\setminus\bigcup_{n=1}^\infty C_n
\end{array}\right.
\]
Clearly, $w\in\mc W$ and $\{\vp\in C_\kappa(E) : \|w\kappa\vp\|_\infty <1\}\subset U$.
Hence Theorem \ref{lem_ck_mixed} and Proposition \ref{t12} imply $\tau_\kappa^{\mc M}\subset \tau_\kappa^{\mc W}$.
\end{proof}
\begin{theorem}\label{cor_comp}
	If Hypothesis \ref{hyp_space} holds, then the space $\left(C_\kappa(E),\tm\right)$ is complete. 
\end{theorem}
\begin{proof} 
The proof is a modification of the proof of a similar result in \cite{gk}. Let 
$\left(\vp_{\alpha}\right)\subset\left(\ck,\tm\right
)$ be a Cauchy net. Then for every seminorm 
$p_{\kappa,\left(C_n\right),\left(a_n\right)}$ and $\ve >0$ 
there exists $\alpha_0$ such that for all $\alpha ,\beta\ge\alpha_
0$ 
\[p_{\kappa,\left(C_n\right),\left(a_n\right)}\left(\vp_{\alpha}-\vp_{\beta}\right
)<\ve\,,\]
hence the net $\left(\vp_{\alpha}\right)$ is pointwise convergent to a certain 
\[\vp=\lim_{\alpha}\vp_{\alpha},\]
and we find that $p_{\kappa,\left(C_n\right),\left(a_n\right)}\left(\vp_{
\alpha}-\vp\right)<\epsilon$ for $\alpha$ 
sufficiently large.
In particular, $\vp_\alpha -\vp\in C_\kappa(E)$, hence so is $\vp$.
 Since $p_C\la \vp_\alpha-\vp\ra\to 0$ for every compact $C\subset E$, we find that $\vp$ is continuous on every compact, hence continuous on $E$ by Hypothesis \ref{hyp_space}  and finally $\vp\in\ck$. 
\end{proof}
 Recall from Section 2 that $M_\kappa(E)$ denotes the space of all Radon measures $\mu$ on $E$ such that 
  \[\|\mu\|_\kappa=\int_E\frac{|\mu|(dx)}{\kappa(x)}<\infty\,,\]
  and we write $M_b(E)$ if $\kappa\equiv 1$.
{
\bt{th_md}\label{thm_B9} 
  Assume that Hypotheses \ref{hyp_space} and \ref{hyp-kappa} hold. Then we have $\la\ck,\tm\ra^\star=M_\kappa(E)$.
  Furthermore, let $\ell\in\left(C_\kappa(E),\tau^{\mc M}\right)^\ast$ with corresponding $\mu_\ell\in M_\kappa(E),K_m\subset E, m\in\N$, compact sets
 and $(0,\infty)\ni a_m\to 0$ such that
 \[ |\ell(\vp)| \leq p_{\kappa,(K_m),(a_m)}(\vp) \quad\text{for all }\vp\in C_\kappa(E).\]
 Then
 \begin{align}\label{eq:b1_prime}
 \int_E |\vp| \text d|\mu_\ell|\leq 2\sup\{|\ell(g)|:g\in C_\kappa(E), 0\leq g\leq |\vp|\} \leq 2 p_{\kappa,(K_m),(a_m)}(\vp) \quad\text{for all }\vp\in C_\kappa(E). 
 \end{align}
\et

\begin{proof}
Let $\mu\in M_\kappa(E)$.
Then obviously $C_\kappa(E)\subset L^1(E,|\mu|)$.
Define $\ell : C_\kappa(E) \to\R$ by
\[ \ell(\vp) := \int_E\vp\;\text d\mu,\ \vp\in C_\kappa(E). \]
Since $|\mu|$ is Radon, there exist compact $K_n\subset E, n\in\N$, such that
\begin{align}\label{eq:b2}
\int_{E\setminus K_n} \frac{|\mu|(\text dx)}{\kappa(x)} \leq \frac{1}{n+1} 2^{-n},\ n\in\N,
\end{align}
in particular, $|\mu|(E\setminus\bigcup_{n=1}^\infty K_n)=0$.
Define $b_n :=\frac 1n, n\in\N$, and let $\vp\in C_\kappa(E)$ such that
\[ p_{\kappa,(K_n),(b_n)}(\vp) = 1.\]
Then
\[ p_{\kappa,K_n}(\vp)\leq n \quad\text{for all }n\in\N\]
and thus by \eqref{eq:b2}
\begin{align*}
|\ell(\vp)| \leq \int |\vp| \;\text d|\mu| \leq \ &p_{\kappa,K_1}(\vp)\,\int_{K_1} \frac{|\mu|(\text dx)}{\kappa(x)} +\\
												& +\sum_{n=1}^\infty p_{\kappa,K_{n+1}}(\vp)\,\int_{K_{n+1}\setminus K_n} \frac{|\mu|(\text dx)}{\kappa(x)}\\
												& \leq\int_E \frac{|\mu|(\text dx)}{\kappa(x)} +1 < \infty.
\end{align*}
Hence, $\ell\in\left(C_\kappa(E),\tau^{\mc M}_\kappa\right)^\ast$.\\
Conversely, let $\ell\in\left(C_\kappa(E),\tau^{\mc M}\right)^\ast$.
Then there exist increasing compacts $K_m\subset E, m\in\N$, and $(0,\infty)\ni a_m\downarrow 0$ such that
\begin{align}\label{eq:b3}
|\ell(\vp)|\leq p_{\kappa,(K_m),(a_m)}(\vp)\quad \text{for all }\vp\in C_\kappa(E).
\end{align}
First assume that $\ell\geq 0$, i.e. $\ell(\vp)\geq 0\ \forall\vp\in C_\kappa(E),\ \vp\geq 0$.
Recall that $C_\kappa(E) = C(E) \cap \frac 1\kappa B_b(E)$; so $\ck$ is trivially a Stone lattice,
and since for all $h\in C_b(E),\ h\geq 0$,
\[h = \lim\limits_{n\to\infty} \frac {1}{\ka_1} (\ka_1\cdot h\wedge n) \text{ pointwise on }E,\]
Hypothesis \ref{hyp_space} (2) implies that $\mc B(E) = \sigma(C_b(E))\subset \sigma(C_\ka(E))\subset \mc B(E)$.
Furthermore $\ell$ is Daniell continuous on $\ck$,
since for $\vp_n\in\ck$, $n\in\N$, $\vp_n\downarrow0$ pointwise on $E$,
we have by Dini's theorem that $\ka\vp_n\to 0$ uniformly on compacts $C\subset E$.
Therefore, by Proposition \ref{t14} we have $\vp_n\to 0$ in $\tm$,
and thus $\ell(\vp_n)\to 0$ as $n\to\infty$.
So, by the Daniell-Stone theorem (see \cite[Satz 39.4]{bauer}) there exists a unique nonnegative measure $\mu_\ell$ on $\la E,\mathscr B(E)\ra$ such that $\ck\subset L^1(E,\mu_\ell)$ and
\begin{align}\label{eq:b4}
\ell(\vp) = \int_E \vp\; \text d\mu_\ell \quad \text{for all }\vp\in \ck.
\end{align}
Furthermore, by Hypothesis \ref{hyp-kappa} and \eqref{eq:b3}
\begin{align}\label{eq:b5}
\int_E \frac 1\ka \;\text d\mu_\ell = \lim\limits_{j\to\infty} \int_E \frac {1}{\ka_j} \;\text d\mu_\ell
\leq \lim\limits_{j\to\infty} p_{\kappa,(K_m),(a_m)}\bigg(\frac {1}{\ka_j}\bigg) \leq \sup_{m\in\N} a_m < \infty,
\end{align}
since $\frac {1}{\ka_j} \in\ck$ for all $j\in\N$.
It remains to show that $\mu_\ell$ is Radon, equivalently $\frac 1\ka \mu_\ell$ is Radon.
Since by \eqref{eq:b5} and \cite[Satz 40.6]{bauer} we have for all $B\in \mc B(E) = \sigma(C_b(E))$
\begin{align}\label{eq:b6}
\la \frac 1\ka \mu_\ell\ra (B) = \sup \left\{\la \frac 1\ka \mu_\ell\ra (A) \mid A\subset B, A\text{ closed}\right\},
\end{align}
it suffices to prove that
\[\lim\limits_{n\to\infty}\la \frac 1\ka \mu_\ell\ra \la E\setminus K_n\ra =0.\]
But by \eqref{eq:b6} there exist closed $A_n\subset E\setminus K_n, A_{n+1}\subset A_n$, and
\begin{align}\label{eq:b6'}
\la \frac 1\ka \mu_\ell\ra \la(E\setminus K_n)\setminus A_n\ra < \frac 1n,\ n\in\N.
\end{align}
Since $E$ is completely regular there exist $f_{A_n}\in C_b(E), 1\geq f_{A_n}\geq 0$,
such that $f_{A_n} = 1$ on $A_n$ and $f_{A_n} = 0$ on $K_n, n\in\N$.
Hence, since $\frac{1}{\ka_j} f_{A_n} \in\ck, j\in\N$, it follows by \eqref{eq:b3}, \eqref{eq:b4} and \eqref{eq:b6'} that
\begin{align*}
\la \frac 1\ka \mu_\ell\ra \la E\setminus K_n\ra &\leq \frac 1n + \la \frac 1\ka \mu_\ell\ra \la A_n\ra\\
												& \leq\frac 1n + \lim\limits_{j\to\infty} p_{\kappa,(K_m),(a_m)}\la\frac{1}{\ka_j}  f_{A_n}\ra\\
												&= \frac 1n + \sup_{m> n} a_m p_{K_m}\la \frac{\ka}{\ka_j} f_{A_n}\ra \leq \frac 1n +a_n.
\end{align*}
Thus $\frac 1\ka \mu_\ell$ is Radon, hence so is $\mu_\ell$.\\
Now consider an arbitrary $\ell\in\la \ck,\tm\ra^\ast$ and the well-known decomposition $\ell = \ell^+-\ell^-$,
where for $\vp\in\ck,\vp\geq 0$
\begin{align}\label{eq:b7}
\ell^+(\vp) := \sup\left\{\ell(g):g\in\ck, 0\leq g\leq\vp\right\}, &\ \ell^-(\vp) := -\inf\left\{\ell(g):g\in\ck, 0\leq g\leq\vp\right\}\\
\text{ and }\ell^+(\vp):= \ell^+(\vp^+)-\ell^+(\vp^-),&\ \ell^-(\vp):= \ell^-(\vp^+)-\ell^-(\vp^-)\text{ for }\vp\in\ck,\notag\\
\vp^+ := \vp\vee0,&\ \vp^- := \vp\wedge 0.\notag
\end{align}
Since for $\vp_1,\vp_2\in\ck$ with $0\leq\vp_1\leq\vp_2$ we have by \eqref{eq:b3}
\begin{align}\label{eq:b8}
|\ell(\vp_1)| \leq p_{\kappa,(K_m),(a_m)}(\vp_1)\leq p_{\kappa,(K_m),(a_m)}(\vp_2),
\end{align}
it follows that the first two quantities defined in \eqref{eq:b7} are finite real numbers.
Therefore, a standard proof (see e.g. \cite[Chap.7, Proof of Theorem 7.8.3]{bogachev}) implies that both $\ell^+,\ell^-$ are linear functions on $\ck$
such that $\ell^+(\vp),\ell^-(\vp)\geq 0$ if $\vp\in\ck,\vp\geq 0$, and $\ell =\ell^+-\ell^-$.
Furthermore, for every $\vp\in\ck$ it follows by \eqref{eq:b8} that
\[|\ell^+(\vp)|\leq \ell^+(|\vp|) = \sup_{0\leq g\leq |\vp|}\ell(g) \leq p_{\kappa,(K_m),(a_m)}(\vp).\]
Hence $\ell^+\in\left(C_\kappa(E),\tau^{\mc M}\right)^\ast$ and hence so is $\ell^-$.\\
Now consider the corresponding measures $\mu_{\ell^+},\mu_{\ell^-}\in M_\ka(E)$ constructed above.\\
Then $\mu_\ell := \mu_{\ell^+}-\mu_{\ell^-}\in M_\ka(E)$ is the required measure corresponding to $\ell$.
To prove the last part of the assertion we note that by \eqref{eq:b7} and \eqref{eq:b8}
\[ \int |\vp|\;\text d\mu_{\ell^\pm}\leq \sup\left\{|\ell(g)|:g\in\ck, 0\leq g\leq|\vp|\right\}\leq p_{\kappa,(K_m),(a_m)}(\vp).\]
Hence \eqref{eq:b1_prime} follows since $|\mu_\ell| \leq \mu_{\ell^+} + \mu_{\ell^-}$.
\end{proof}
}

\begin{corollary}\label{cor_dense}
Let $\mc A\subset C_b(E)$ be an algebra that separates the points of $E$ and such
that for each $x\in E$ there exists $a\in \mc A$ with $a(x)\neq 0$.
Then $\mc A$ is $\tm$-dense in $\ck$.
In particular, $C_b(E)$ is dense in $\ck$.
\begin{proof}
{Since $\tm\subset \tu$ and $\kappa$ is bounded,} we may assume that $\mc A$ is closed {w.r.t. uniform convergence on $E$}.
Let $\ell\in \la\ck,\tm\ra^\ast$ such that
\begin{align}\label{eq:b9}
\ell(f) = 0 \quad \forall f\in \mc A.
\end{align}
If we can show that then $\ell=0$, the assertion follows by Mazur's separation
theorem for locally convex spaces.\\
First we note that by the Stone-Weierstrass theorem there exist polynomials $p_n, n\in\N$,
such that $p_n\to|\cdot|$ uniformly on $[-1,1]$.
Hence for every $f\in \mc A$ we have $|f|\in \mc A$, i.e. $\mc A$ is a vector lattice.
To see this we may assume that $|f|\leq1$.
Then, as $n\to\infty$, $\mathscr A\ni p_n(f) \to |f|$ uniformly on $E$, hence $|f|\in \mathscr A$.\\
By Theorem \ref{thm_B9} there exist $\mu_\ell,\mu_{\ell^+}, \mu_{\ell^-}\in M_\ka(E)$ corresponding to $\ell,\ell^+$ and $\ell^-$ respectively.
Hence by \eqref{eq:b9}
\begin{align}\label{eq:b10}
\int_E f\,\text d\mu_{\ell^+} = \int_E f \,\text d\mu_{\ell^-}\quad \text{for all } f\in \mc A.
\end{align}
Since $\mu_{\ell^+},\mu_{\ell^-}$ are Radon measures there exist increasing compacts $K_m\subset E, m\in\N$, such that for $K_\sigma := \bigcup\limits_{m=1}^\infty K_m$
\begin{align}\label{eq:b11}
\mu_{\ell^+}(E\setminus K_\sigma) = \mu_{\ell^-}(E\setminus K_\sigma) = 0.
\end{align}
Since $E= \bigcup\limits_{a\in\mc A}\{a\neq 0\}$ and $\mc A$ is a vector lattice,
there exist functions $\chi_m\in\mc A$ such that $\chi_m\geq0$ and $\chi_m>0$ on $K_m,m\in\N$.
Since $\mc A$ is an algebra, \eqref{eq:b10} implies
\begin{align}\label{eq:b12}
\int_E f\chi_m\, \text d\mu_{\ell^+} = \int_E f\chi_m\, \text d\mu_{\ell^-} \quad\text{for all } f\in\mc A\cup\{1\}, m\in\N.
\end{align}
Hence by a standard monotone class argument and by \eqref{eq:b11} we obtain that
\begin{align}\label{eq:b13}
\chi_m\mu_{\ell^+} = \chi_m\mu_{\ell^-} \quad \text{on }\sigma(\mc A)\cap K_\sigma \text{ for all } m\in\N.
\end{align}
Below every subset $A\subset E$ is considered to be equipped with the trace topology induced
by $E$ and $\mc B(A)$ denotes the corresponding Borel $\sigma$-algebra on $A$.
Since $K_\sigma$ is a strong Lindel\"of space by Hypothesis \ref{hyp_space} (1),
$K_\sigma\times K_\sigma$ is a Lindel\"of space.
Hence, since by assumption
\[ K_\sigma\times K_\sigma = \bigcup\limits_{f\in\mc A} \left\{(x,y)\in K_\sigma\times K_\sigma | (f(x),f(y))\in\R^2\setminus D\right\},\]
where $D$ denotes the diagonal in $\R^2$, there exists a countable set $\tilde{\mc A}\subset \mc A$ still separating the points of $K_\sigma$.
Then for every $x\in K_\sigma$
\[\{x\} = \bigcap\limits_{f\in\tilde{\mc A}}\{y\in K_\sigma : f(y) = f(x)\}\in \sigma(\tilde{\mc A})\cap K_\sigma\]
and $\sigma\!\la \tilde{\mc A}_{\upharpoonright K_\sigma}\ra$ is countably generated,
where $\tilde{\mc A}_{\upharpoonright K_\sigma}$ denotes the set of restrictions to $K_\sigma$ of all functioins in $\tilde{\mc A}$.
Furthermore, since each $K_m$ is compact and metrizable, hence becomes a complete separable metric space,
it is a standard Borel space, i.e. in the terminology of \cite{preston} a separable type $\mc B$ space.
Then by \cite[Proposition 6.2 (4)]{preston} $K_\sigma$ with its Borel $\sigma$-algebra $\mc B(K_\sigma)$
is also a standard Borel space.
Therefore, since $\sigma(\tilde{\mc A}_{\upharpoonright K_\sigma})\subset \mc B(K_\sigma)$,
we may apply a variant of Kuratowski's theorem (see \cite[Theorem 6.4]{preston}) to conclude
\begin{align}\label{eq:b14}
\sigma(\tilde{\mc A}_{\upharpoonright K_\sigma}) = \mc B(K_\sigma).
\end{align}
Since ${\mc A}_{\upharpoonright K_\sigma}$ consists of continuous functions on $K_\sigma$ and thus by \eqref{eq:b14}
\[ \mc B(K_\sigma) = \sigma({\mc A}_{\upharpoonright K_\sigma}) = \sigma(\mc A)\cap K_\sigma,\]
\eqref{eq:b13} implies that for all $m\in\N$
\[\chi_m\mu_{\ell^+} = \chi_m\mu_{\ell^-}\quad\text{on } \mc B(K_\sigma)\supset\mc B(K_m),\]
which by the strict positivity of every $\chi_m$ in turn implies that
\[\mu_{\ell^+} = \mu_{\ell^-} \quad \text{on } \mc B(K_m).\]
Thus for every open $U\subset E$, since $U\cap K_m\in\mc B(K_m)$,
\begin{align*}
\mu_{\ell^+}(U) &=\mu_{\ell^+}(U\cap K_\sigma) = \lim\limits_{m\to\infty} \mu_{\ell^+}(U\cap K_m)\\
			&= \lim\limits_{m\to\infty} \mu_{\ell^-}(U\cap K_m)=\mu_{\ell^-}(U\cap K_\sigma)\\
			&= \mu_{\ell^-}(U).
\end{align*}
Hence $\mu_{\ell^+}=\mu_{\ell^-}$, so $\ell^+=\ell^-$ and thus $\ell=0$.
\end{proof}
\end{corollary}

\bibliographystyle{abbrv}

\begin{thebibliography}{10}

\bibitem{albanese2016}
A.~A. Albanese and D.~Jornet.
\newblock Dissipative operators and additive perturbations in locally convex
  spaces.
\newblock {\em Math. Nachr.}, 289(8-9):920--949, 2016.

\bibitem{afpp}
L.~Angiuli, S.~Ferrari, and D.~Pallara.
\newblock Functional inequalities for some generalised {M}ehler semigroups.
\newblock {\em Preprint arXiv:2106.04241}, 2021.

\bibitem{bauer}
H.~Bauer.
\newblock Wahrscheinlichkeitstheorie und {G}rundz\"{u}ge der {M}asstheorie.
\newblock {\em de Gruyter} 1978

\bibitem{bliedtner}
J.~Bliedtner and W.~Hansen.
\newblock {\em Potential theory}.
\newblock Universitext. Springer-Verlag, Berlin, 1986.
\newblock An analytic and probabilistic approach to balayage.

\bibitem{blumenthal}
R.~M. Blumenthal and R.~K. Getoor.
\newblock {\em Markov processes and potential theory}.
\newblock Pure and Applied Mathematics, Vol. 29. Academic Press, New
  York-London, 1968.

\bibitem{bogachev}
V.~I. Bogachev.
\newblock {\em Measure theory. {V}ol. {I}, {II}}.
\newblock Springer-Verlag, Berlin, 2007.

\bibitem{bogachev2}
V.~I. Bogachev.
\newblock {\em Weak convergence of measures}, volume 234 of {\em Mathematical
  Surveys and Monographs}.
\newblock American Mathematical Society, Providence, RI, 2018.

\bibitem{bogachev15}
V.~I. Bogachev, G.~Da~Prato, M.~R\"{o}ckner, and S.~V. Shaposhnikov.
\newblock An analytic approach to infinite-dimensional continuity and
  {F}okker-{P}lanck-{K}olmogorov equations.
\newblock {\em Ann. Sc. Norm. Super. Pisa Cl. Sci. (5)}, 14(3):983--1023, 2015.

\bibitem{BKR}
V.~I. Bogachev, N.~V. Krylov, and M.~R\"{o}ckner.
\newblock On regularity of transition probabilities and invariant measures of
  singular diffusions under minimal conditions.
\newblock {\em Comm. Partial Differential Equations}, 26(11-12):2037--2080,
  2001.

\bibitem{fpke}
V.~I. Bogachev, N.~V. Krylov, M.~R\"{o}ckner, and S.~V. Shaposhnikov.
\newblock {\em Fokker-{P}lanck-{K}olmogorov equations}, volume 207 of {\em
  Mathematical Surveys and Monographs}.
\newblock American Mathematical Society, Providence, RI, 2015.

\bibitem{bogroeschmu}
V.~I. Bogachev, M.~R\"{o}ckner, and B.~Schmuland.
\newblock Generalized {M}ehler semigroups and applications.
\newblock {\em Probab. Theory Related Fields}, 105(2):193--225, 1996.

\bibitem{bff}
C.~Budde and B.~Farkas.
\newblock Intermediate and extrapolated spaces for bi-continuous semigroups.
\newblock {\em Journal of Evolution Equations}, pages 321--359, 2019.

\bibitem{farkas-budde}
C.~Budde and B.~Farkas.
\newblock A {D}esch-{S}chappacher perturbation theorem for bi-continuous
  semigroups.
\newblock {\em Math. Nachr.}, 293(6):1053--1073, 2020.

\bibitem{budde-wegner}
C.~Budde and S.-A. Wegner.
\newblock A {L}umer--{P}hillips type generation theorem for bi-continuous
semigroups.
\newblock {\em Z. Anal. Anwend.}, 41(1):65--80, 2022.

\bibitem{cerrai}
S.~Cerrai.
\newblock A {H}ille-{Y}osida theorem for weakly continuous semigroups.
\newblock {\em Semigroup Forum}, 49(3):349--367, 1994.

\bibitem{cg}
S.~Cerrai and F.~Gozzi.
\newblock Strong solutions of {C}auchy problems associated to weakly continuous
  semigroups.
\newblock {\em Differential Integral Equations}, 8(3):465--486, 1995.


\bibitem{dz3}
G.~Da~Prato and J.~Zabczyk.
\newblock {\em Second order partial differential equations in {H}ilbert
  spaces}, volume 293 of {\em London Mathematical Society Lecture Note Series}.
\newblock Cambridge University Press, Cambridge, 2002.

\bibitem{dz}
G.~Da~Prato and J.~Zabczyk.
\newblock {\em Stochastic equations in infinite dimensions}, volume 152 of {\em
  Encyclopedia of Mathematics and its Applications}.
\newblock Cambridge University Press, Cambridge, second edition, 2014.

\bibitem{denk2018kolmogorov}
R.~Denk, M.~Kupper, and M.~Nendel.
\newblock Kolmogorov-type and general extension results for nonlinear
  expectations.
\newblock {\em Banach Journal of Mathematical Analysis}, 12(3):515--540, 2018.

\bibitem{dkn0}
R.~Denk, M.~Kupper, and M.~Nendel.
\newblock A semigroup approach to nonlinear {L}\'{e}vy processes.
\newblock {\em Stochastic Process. Appl.}, 130(3):1616--1642, 2020.


\bibitem{DS}
N.~Dunford and J.~T. Schwartz.
\newblock {\em Linear operators. {P}art {I}}.
\newblock Wiley Classics Library. John Wiley \& Sons, Inc., New York, 1988.
\newblock General theory, With the assistance of William G. Bade and Robert G.
  Bartle, Reprint of the 1958 original, A Wiley-Interscience Publication.

\bibitem{dynkin}
E.~B. Dynkin.
\newblock {\em Markov processes. {V}ols. {I}, {II}}, volume 122 of {\em Die
  Grundlehren der mathematischen Wissenschaften, Band 121}.
\newblock Academic Press, Inc., Publishers, New York; Springer-Verlag,
  Berlin-G\"{o}ttingen-Heidelberg, 1965.
\newblock Translated with the authorization and assistance of the author by J.
  Fabius, V. Greenberg, A. Maitra, G. Majone.


\bibitem{engel}
K.-J. Engel and R.~Nagel.
\newblock {\em One-parameter semigroups for linear evolution equations}, volume
  194 of {\em Graduate Texts in Mathematics}.
\newblock Springer-Verlag, New York, 2000.
\newblock With contributions by S. Brendle, M. Campiti, T. Hahn, G. Metafune,
  G. Nickel, D. Pallara, C. Perazzoli, A. Rhandi, S. Romanelli and R.
  Schnaubelt.
  

\bibitem{farkas-sarhir}
A.~Es-Sarhir and B.~Farkas.
\newblock Positivity of perturbed {O}rnstein-{U}hlenbeck semigroups on
  {$C_b(H)$}.
\newblock {\em Semigroup Forum}, 70(2):208--224, 2005.

\bibitem{ethier}
S.~N. Ethier and T.~G. Kurtz.
\newblock {\em Markov processes}.
\newblock Wiley Series in Probability and Mathematical Statistics: Probability
  and Mathematical Statistics. John Wiley \& Sons, Inc., New York, 1986.
\newblock Characterization and convergence.

\bibitem{Fabbri17}
G.~Fabbri, F.~Gozzi, and A.~\'{S}wi\polhk ech.
\newblock {\em Stochastic optimal control in infinite dimension}, volume~82 of
  {\em Probability Theory and Stochastic Modelling}.
\newblock Springer, Cham, 2017.
\newblock Dynamic programming and HJB equations, With a contribution by Marco
  Fuhrman and Gianmario Tessitore.

\bibitem{falb}
P.~L. Falb and M.~Q. Jacobs.
\newblock On differentials in locally convex spaces.
\newblock {\em J. Differential Equations}, 4:444--459, 1968.

\bibitem{federico}
S.~Federico and M.~Rosestolato.
\newblock {$C_0$}-sequentially equicontinuous semigroups.
\newblock {\em Kyoto J. Math.}, 60(3):1131--1175, 2020.

\bibitem{freidlin}
M.~Freidlin.
\newblock {\em Markov processes and differential equations: asymptotic
  problems}.
\newblock Lectures in Mathematics ETH Z\"{u}rich. Birkh\"{a}user Verlag, Basel,
  1996.

\bibitem{haydon}
D.~H. Fremlin, D.~J.~H. Garling, and R.~G. Haydon.
\newblock Bounded measures on topological spaces.
\newblock {\em Proc. London Math. Soc. (3)}, 25:115--136, 1972.

\bibitem{fr}
M.~Fuhrman and M.~R\"{o}ckner.
\newblock Generalized {M}ehler semigroups: the non-{G}aussian case.
\newblock {\em Potential Anal.}, 12(1):1--47, 2000.

\bibitem{fukushima}
M.~Fukushima, Y.~Oshima, and M.~Takeda.
\newblock {\em Dirichlet forms and symmetric {M}arkov processes}, volume~19 of
  {\em De Gruyter Studies in Mathematics}.
\newblock Walter de Gruyter \& Co., Berlin, extended edition, 2011.

\bibitem{Gabriyelyan}
S.~Gabriyelyan.
\newblock On the {A}scoli property for locally convex spaces.
\newblock {\em Topology Appl.}, 230:517--530, 2017.

\bibitem{gk}
B.~Goldys and M.~Kocan.
\newblock Diffusion semigroups in spaces of continuous functions with mixed
  topology.
\newblock {\em J. Differential Equations}, 173(1):17--39, 2001.

\bibitem{jarchow}
H.~Jarchow.
\newblock {\em Locally convex spaces}.
\newblock B. G. Teubner, Stuttgart, 1981.
\newblock Mathematische Leitf\"{a}den. [Mathematical Textbooks].

\bibitem{kolokoltsov}
V.~N. Kolokoltsov.
\newblock {\em Markov processes, semigroups and generators}, volume~38 of {\em
  De Gruyter Studies in Mathematics}.
\newblock Walter de Gruyter \& Co., Berlin, 2011.

\bibitem{komura}
T.~Komura.
\newblock Semigroups of operators in locally convex spaces.
\newblock {\em J. Functional Analysis}, 2:258--296, 1968.

\bibitem{kraaij}
R.~Kraaij.
\newblock Strongly continuous and locally equi-continuous semigroups on locally
  convex spaces.
\newblock {\em Semigroup Forum}, 92(1):158--185, 2016.


\bibitem{kruse-seifert}
K.~Kruse and C.~Seifert.
\newblock A note on the {L}umer--{P}hillips theorem for bi-continuous semigroups.
\newblock {\em arXiv:2206.00887}, 2022.

\bibitem{MR3941868}
F.~K\"{u}hn.
\newblock Viscosity solutions to {H}amilton-{J}acobi-{B}ellman equations
  associated with sublinear {L}\'{e}vy(-type) processes.
\newblock {\em ALEA Lat. Am. J. Probab. Math. Stat.}, 16(1):531--559, 2019.

\bibitem{kuhnemund}
F.~K\"{u}hnemund.
\newblock A {H}ille-{Y}osida theorem for bi-continuous semigroups.
\newblock {\em Semigroup Forum}, 67(2):205--225, 2003.


\bibitem{kunze}
M.~Kunze.
\newblock Continuity and equicontinuity of semigroups on norming dual pairs.
\newblock {\em Semigroup Forum}, 79(3):540--560, 2009.

\bibitem{Le-Cam}
L.~LeCam.
\newblock Convergence in distribution of stochastic processes.
\newblock {\em Univ. California Publ. Statist.}, 2:207--236, 1957.

\bibitem{lescot}
P.~Lescot and M.~R\"{o}ckner.
\newblock Generators of {M}ehler-type semigroups as pseudo-differential
  operators.
\newblock {\em Infin. Dimens. Anal. Quantum Probab. Relat. Top.},
  5(3):297--315, 2002.

\bibitem{liggett}
T.~M. Liggett.
\newblock {\em Continuous time {M}arkov processes}, volume 113 of {\em Graduate
  Studies in Mathematics}.
\newblock American Mathematical Society, Providence, RI, 2010.
\newblock An introduction.

\bibitem{LR15}
W.~Liu and M.~R\"{o}ckner.
\newblock {\em Stochastic partial differential equations: an introduction}.
\newblock Universitext. Springer, Cham, 2015.

\bibitem{ms}
B.~Maslowski and J.~Seidler.
\newblock On sequentially weakly {F}eller solutions to {SPDE}'s.
\newblock {\em Atti Accad. Naz. Lincei Cl. Sci. Fis. Mat. Natur. Rend. Lincei
  (9) Mat. Appl.}, 10(2):69--78, 1999.

\bibitem{Megginson}
R.~E. Megginson.
\newblock {\em An introduction to {B}anach space theory}, volume 183 of {\em
  Graduate Texts in Mathematics}.
\newblock Springer-Verlag, New York, 1998.

\bibitem{NR}
M.~Nendel and M.~R\"ockner.
\newblock Upper envelopes of families of {F}eller semigroups and viscosity
  solutions to a class of nonlinear {C}auchy problems.
\newblock {\em SIAM J. Control Optim.}, 59(6):4400--4428, 2021.

\bibitem{nendel}
M.~Nendel.
\newblock Lower semicontinuity of monotone functionals in the mixed topology on $C_b$.
\newblock {\em Preprint arXiv:2210.09133}, 2022.


\bibitem{NutzNeuf}
A.~Neufeld and M.~Nutz.
\newblock Nonlinear {L}\'evy processes and their characteristics.
\newblock {\em Trans. Amer. Math. Soc.}, 369(1):69--95, 2017.

\bibitem{Noble}
N.~Noble.
\newblock Ascoli theorems and the exponential map.
\newblock {\em Trans. Amer. Math. Soc.}, 143:393--411, 1969.

\bibitem{partha}
K.~R. Parthasarathy.
\newblock {\em Probability measures on metric spaces}.
\newblock Probability and Mathematical Statistics, No. 3. Academic Press, Inc.,
  New York-London, 1967.

\bibitem{pazy}
A.~Pazy.
\newblock {\em Semigroups of linear operators and applications to partial
  differential equations}, volume~44 of {\em Applied Mathematical Sciences}.
\newblock Springer-Verlag, New York, 1983.

\bibitem{peng2010nonlinear}
S.~Peng.
\newblock {\em Nonlinear expectations and stochastic calculus under
  uncertainty}, volume~95 of {\em Probability Theory and Stochastic Modelling}.
\newblock Springer, Berlin, 2019.
\newblock With robust CLT and G-Brownian motion.

\bibitem{pham}
H.~Pham.
\newblock {\em Continuous-time stochastic control and optimization with
  financial applications}, volume~61 of {\em Stochastic Modelling and Applied
  Probability}.
\newblock Springer-Verlag, Berlin, 2009.

\bibitem{preston}
C.~Preston.
\newblock {\em Some Notes on Standard Borel and Related Spaces}.
\newblock arXiv:0809.3066v1.

\bibitem{priola}
E.~Priola.
\newblock On a class of {M}arkov type semigroups in spaces of uniformly
  continuous and bounded functions.
\newblock {\em Studia Math.}, 136(3):271--295, 1999.

\bibitem{jiagang}
J.~Ren, M.~R\"{o}ckner, and F.-Y. Wang.
\newblock Stochastic generalized porous media and fast diffusion equations.
\newblock {\em J. Differential Equations}, 238(1):118--152, 2007.

\bibitem{R}
M.~R\"{o}ckner.
\newblock {$L^p$}-analysis of finite and infinite-dimensional diffusion
  operators.
\newblock In {\em Stochastic {PDE}'s and {K}olmogorov equations in infinite
  dimensions ({C}etraro, 1998)}, volume 1715 of {\em Lecture Notes in Math.},
  pages 65--116. Springer, Berlin, 1999.

\bibitem{RT}
M.~R\"{o}ckner and G.~Trutnau.
\newblock A remark on the generator of right-continuous {M}arkov process.
\newblock {\em Infin. Dimens. Anal. Quantum Probab. Relat. Top.},
  10(4):633--640, 2007.

\bibitem{rogers}
L.~C.~G. Rogers and D.~Williams.
\newblock {\em Diffusions, {M}arkov processes, and martingales. {V}ol. 2}.
\newblock Cambridge Mathematical Library. Cambridge University Press,
  Cambridge, 2000.
\newblock It\^{o} calculus, Reprint of the second (1994) edition.

\bibitem{sentilles}
F.~D. Sentilles.
\newblock Bounded continuous functions on a completely regular space.
\newblock {\em Trans. Amer. Math. Soc.}, 168:311--336, 1972.

\bibitem{sharpe}
M.~Sharpe.
\newblock {\em General theory of {M}arkov processes}, volume 133 of {\em Pure
  and Applied Mathematics}.
\newblock Academic Press, Inc., Boston, MA, 1988.

\bibitem{stroock}
D.~W. Stroock.
\newblock {\em An introduction to {M}arkov processes}, volume 230 of {\em
  Graduate Texts in Mathematics}.
\newblock Springer, Heidelberg, second edition, 2014.

\bibitem{tong}
H.~Tong.
\newblock Some characterizations of normal and perfectly normal spaces.
\newblock {\em Duke Math. J. 19} 1952, 289--292.

\bibitem{vakhania}
N.~N. Vakhania, V.~I. Tarieladze, and S.~A. Chobanyan.
\newblock {\em Probability distributions on {B}anach spaces}, volume~14 of {\em
  Mathematics and its Applications (Soviet Series)}.
\newblock D. Reidel Publishing Co., Dordrecht, 1987.
\newblock Translated from the Russian and with a preface by Wojbor A.
  Woyczynski.

\bibitem{neerven-Zabczyk}
J.~M. A.~M. van Neerven and J.~Zabczyk.
\newblock Norm discontinuity of {O}rnstein-{U}hlenbeck semigroups.
\newblock {\em Semigroup Forum}, 59(3):389--403, 1999.

\bibitem{wheeler1}
R.~F. Wheeler.
\newblock A survey of {B}aire measures and strict topologies.
\newblock {\em Exposition. Math.}, 1(2):97--190, 1983.

\bibitem{wiweger}
A.~Wiweger.
\newblock Linear spaces with mixed topology.
\newblock {\em Studia Math.}, 20:47--68, 1961.

\bibitem{yosida}
K.~Yosida.
\newblock {\em Functional analysis}.
\newblock Classics in Mathematics. Springer-Verlag, Berlin, 1995.
\newblock Reprint of the sixth (1980) edition.

\end{thebibliography}

\end{document}